\documentclass[11pt]{amsart}

\usepackage{amssymb} 
\usepackage{amsmath,amscd}
\usepackage{amsthm}
\usepackage{graphicx,psfrag}
\usepackage{hyperref}

\usepackage{epstopdf}

\theoremstyle{plain}
\newtheorem{theorem}{Theorem}[section]

\newtheorem{lemma}[theorem]{Lemma}
\newtheorem{corollary}[theorem]{Corollary}
\newtheorem{proposition}[theorem]{Proposition}

\theoremstyle{definition}
\newtheorem{definition}[theorem]{Definition}

\newtheorem{notation}[theorem]{Notation}

\theoremstyle{remark}

\numberwithin{figure}{section}

\newcommand{\Int}{\mathrm{int}}

\begin{document}

\title{An algorithm to determine the Heegaard genus of a 3-manifold}

\author{Tao Li}
\thanks{Partially supported by an NSF grant} %DMS 0705285.

\address{Department of Mathematics \\
 Boston College \\
 Chestnut Hill, MA 02467}
\email{taoli@bc.edu}

\begin{abstract}
We give an algorithmic proof of the theorem that a closed orientable irreducible and atoroidal 3-manifold has only finitely many Heegaard splittings in each genus, up to isotopy.  The proof gives an algorithm to determine the Heegaard genus of an atoroidal 3-manifold.
\end{abstract}

\maketitle

\begin{psfrags}

\setcounter{tocdepth}{1}
\tableofcontents

\section{Introduction}\label{Sintro}
A Heegaard splitting of a closed and orientable 3-manifold $M$ is a decomposition $M=H_1\cup_S H_2$ of $M$ into a pair of handlebodies $H_1$ and $H_2$ along a closed surface $S$.  Every  3-manifold has a Heegaard splitting.  The minimal genus of the Heegaard surface $S$ among all Heegaard splittings is called the Heegaard genus of $M$.  In \cite{L4}, the author proved the so-called generalized Waldhausen conjecture.

\begin{theorem}[\cite{L4}]\label{TLi}
A closed orientable irreducible and atoroidal 3-manifold has only finitely many Heegaard splittings in each genus, up to isotopy.
\end{theorem}

An interesting feature in 3-manifold topology is that most decision problems are solvable. For example the word problem is solvable for 3-manifold groups but unsolvable for higher dimensional manifolds.  One goal in 3-manifold topology is to find algorithms to determine all the geometric and algebraic properties of any given 3-manifold.  However, the proof of Theorem~\ref{TLi} in \cite{L4} is not algorithmic because of a compactness argument on the projective measured lamination space of a branched surface.  In this paper, we give an algorithmic proof of Theorem~\ref{TLi}.   This proof in fact gives us a slightly stronger theorem.

\begin{theorem}\label{Tmain}
Given a closed, orientable, irreducible and atoroidal 3-manifold $M$, there is an algorithm to produce a finite list of all possible Heegaard splittings of $M$ in each genus, up to isotopy. 
\end{theorem}

Note that Theorem~\ref{TLi} is not true for toroidal manifolds as one may generate infinitely many non-isotopic Heegaard splittings by Dehn twists along an incompressible torus.   In most part of this paper, we will assume $M$ is non-Haken and in section~\ref{SHaken} we will show that the proof can be easily extended to atoroidal Haken manifolds.  Note that the Haken case was previously proved by Johannson \cite{Jo1, Jo2}, 

Although Theorem~\ref{Tmain} gives a complete list of Heegaard splittings of bounded genus, there may be repetition in the list, i.e., we do not have an algorithm to determine whether or not two Heegaard splittings are isotopic.  Nevertheless, Theorem~\ref{Tmain} immediately gives an algorithm to determine the Heegaard genus of a 3-manifold and this answers one of a few major decision problems left in 3-manifold topology.  

\begin{corollary}\label{Cmain}
There is an algorithm to determine the Heegaard genus of a closed orientable and atoroidal 3-manifold.
\end{corollary}

Note that if the 3-manifold is toroidal, a theorem of Johannson \cite{Jo1, Jo2} gives an algorithm to determine the Heegaard genus.  In Corollary~\ref{Cmain} we do not have to assume the manifold is irreducible, since the Heegaard genus is additive under connected sum.

By a theorem of Casson and Gordon \cite{CG}, in a non-Haken 3-manifold, an unstabilized Heegaard splitting is strongly irreducible.  So we can assume our Heegaard splittings are strongly irreducible.  By a theorem of Rubinstein \cite{R} and Stocking \cite{St}, every strongly irreducible Heegaard surface is isotopic to a normal or an almost normal surface. In this paper, we use branched surfaces to study almost normal Heegaard surfaces.  In section~\ref{Spre}, we briefly review and prove some properties of 0-efficient triangulations and branched surfaces, which are the basic tools for the proof of the main theorem.  From normal surface theory, we know that there is a finite set of fundamental solutions that generate all normal and almost normal surfaces. Using a 0-efficient triangulation, we may assume the surfaces in the fundamental solutions have non-positive Euler characteristic.  Since the genus of our Heegaard surface is bounded, one can express the strongly irreducible Heegaard surface as a sum $S=F+\sum n_i T_i$ where each $T_i$ is a normal torus in the fundamental set of solutions and there are only finitely many choices for $F$.  If the coefficients $n_i$ are all bounded, then there are only finitely many possible such Heegaard surfaces and the main theorem follows.  Our task is to study the situation that the coefficients $n_i$ are unbounded, and our main tool is branched surface.
 
In section~\ref{SD2}, we perform some isotopies and compressions on a Heegaard surface carried by a branched surface to eliminate certain ``bubbles" called $D^2\times I$ regions.  These isotopies are very simple and canonical, but it is surprisingly difficult to prove that the process ends in finitely many steps.  To prove this, we have to compress the Heegaard surface and recover the original Heegaard surface in the end.  A key part of the proof of the main theorem is an analysis of the intersection of normal tori carried by a branched surface.  In sections \ref{Ssum} and \ref{Sinter}, we study the intersection of normal tori carried by a branched surface. We show that either those coefficients $n_i$ above are bounded and hence the main theorem holds, or the normal tori have a nice intersection pattern.  In section~\ref{Sengulf}, we show that in the latter case, the normal tori with nice intersection pattern lie in some nice solid tori.  Then we apply a theorem of Scharlemann \cite{S} on the intersection of a solid torus and a strongly irreducible Heegaard surface to show that all but finitely many such Heegaard surfaces are isotopic.

In \cite{L5}, the author proved a much stronger theorem, which says that there are only finitely many irreducible Heegaard splittings in a non-Haken 3-manifold, without the genus constraint.  The proof in \cite{L5} also uses measured laminations, so one would hope the methods in this paper can be used to give an algorithmic proof of this theorem.

\begin{notation}
Throughout this paper, for any topological space $X$, we use $\Int(X)$, $|X|$ and $\overline{X}$ to denote the interior, number of components and closure of $X$ respectively.  Unless specified, we will assume each of our surfaces is embedded and assume the surfaces are in general position if they intersect.
\end{notation}

\section{Branched surfaces and normal surfaces}\label{Spre}

Similar to the proof in \cite{L4}, we make extensive use of branched surfaces and 0-efficient triangulations.  In this section, we first recall some properties of branched surfaces and 0-efficient triangulations.  We refer the reader to \cite[Section 2]{L4} for a slightly more detailed discussion.

A \textbf{branched surface} in $M$ is a union of finitely many compact smooth surfaces glued together to form a compact subspace (of $M$)  locally modeled on Figure~\ref{branch}(a).

\begin{figure}
\begin{center}
\psfrag{(a)}{(a)}
\psfrag{(b)}{(b)}
\psfrag{horizontal}{$\partial_hN(B)$}
\psfrag{v}{$\partial_vN(B)$}
\includegraphics[width=4.0in]{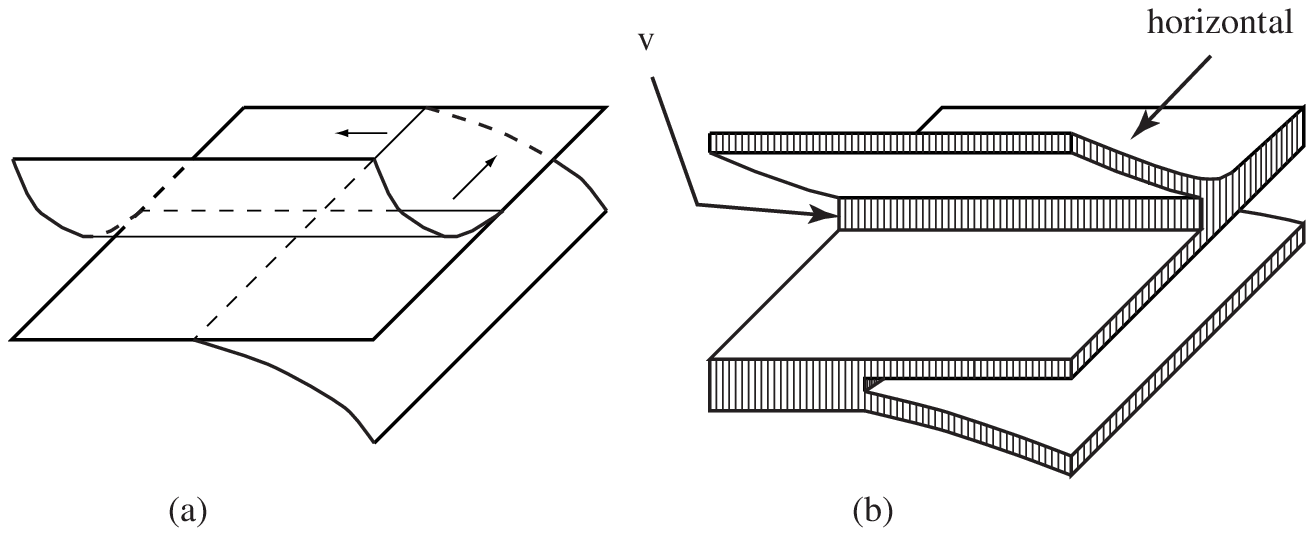}
\caption{}\label{branch}
\end{center}
\end{figure}

Given a branched surface $B$ embedded in a 3-manifold $M$, we denote by $N(B)$ a regular neighborhood of $B$, as shown in Figure~\ref{branch}(b).  One can regard $N(B)$ as an $I$-bundle over $B$, where $I$ denotes the interval $[0,1]$.  The boundary of $N(B)$ is divided into two parts: the horizontal boundary $\partial_hN(B)$ and the vertical boundary $\partial_vN(B)$, see Figure~\ref{branch}(b).  In particular, $\partial_vN(B)$ is a collection of annuli.  Throughout this paper, we denote by $\pi : N(B)\to B$ the projection that collapses every $I$-fiber to a point. We call an arc in $N(B)$ a \textbf{vertical arc} if it is a subarc of an $I$-fiber of $N(B)$.  We call an annulus $A\subset N(B)$ a \textbf{vertical annulus} in $N(B)$ if $A=S^1\times I$ with each $\{x\}\times I$ ($x\in S^1$) a subarc of an $I$-fiber of $N(B)$.  Although we use the phrase ``vertical arc" to denote a subarc of an $I$-fiber of $N(B)$, for any vertical annulus $A=S^1\times I$ in this paper, a vertical arc of $A$ means an arc $\{x\}\times I$ ($x\in S^1$) in $A$ rather than a subarc of $\{x\}\times I$. 
We say a surface $F$ is \textbf{carried} by $B$ (or carried by $N(B)$) if $F\subset N(B)$ and $F$ is transverse to the $I$-fibers of $N(B)$.  Note that in general $F$ may have boundary and $F$ may not be compact.  For example, $\partial_hN(B)$ is regarded as a surface carried by $N(B)$.  We say $F$ is \textbf{fully carried} by $B$ (or $N(B)$) if in addition $F$ intersects every $I$-fiber of $N(B)$.

The \textbf{branch locus} of $B$ is $L=\{b\in B:$ $b$ does not have a neighborhood in $B$ homeomorphic to $\mathbb{R}^2  \}$.  We call the closure (under path metric) of a component of $B-L$ a \textbf{branch sector}. We associate with every smooth arc in $L$ a vector (in $B$) pointing in the direction of the cusp, as shown in Figure~\ref{branch}(a).  We call it the \textbf{branch direction} of this arc, and each component of $\partial_vN(B)$ has an induced branch direction which is the normal direction for $\partial_vN(B)$ pointing into $N(B)$.  If a subset of $B$ itself is a branched surface (without boundary), then we call it a \textbf{sub-branched surface} of $B$.  For example, if there is a branch sector $S$ with branch direction of each arc in $\partial S$ pointing out of $S$, then $B-\Int(S)$ is a sub-branched surface of $B$, see \cite[Figure 2.1]{L1} for a picture.  

Let $F\subset N(B)$ be a surface carried by $N(B)$ (or $B$), and let $S$ be a branch sector of $B$.  We say that $F$ \textbf{passes through} the branch sector $S$ if $F\cap\pi^{-1}(\Int(S))\ne\emptyset$, where $\pi : N(B)\to B$ is the collapsing map. 
If $F$ is a closed surface and $F$ is carried but not fully carried by $N(B)$, then the union of all the $I$-fibers that intersect $F$ can be viewed as the fibered neighborhood of a sub-branched surface of $B$ that fully carries $F$.  In fact, this sub-branched surface can be obtained from $B$ by deleting all the branch sectors that $F$ does not pass through.

Let $B$ be a branched surface in $M$ and $F\subset N(B)$ a closed surface carried by $B$.  Let $L$ be the branch locus of $B$, and suppose $b_1,\dots, b_N$ are the components of $B-L$.  For each $b_i$, let $z_i$ be a point in $b_i$ and $x_i=|F\cap\pi^{-1}(z_i)|$.  Then one can describe $F$ using a non-negative integer point $(x_1,\dots,x_N)\in\mathbb{R}^N$, and  $(x_1,\dots,x_N)$ is a solution to the system of branch equations of $B$ (each branch equation is of the form $x_k=x_i+x_j$), see \cite{FO,O} for more details.  $F$ is fully carried by $B$ if and only if each $x_i$ is positive.   We use $\mathcal{S}(B)$ to denote the set of non-negative integer solutions to the system of branch equations of $B$.  This gives a one-to-one correspondence between closed surfaces carried by $B$ and points in $\mathcal{S}(B)$.  We call the sum of all the coordinates $\sum_{i=1}^Nx_i$ the \textbf{weight} of $F$ with respect to $B$.

Let $F_1$ and $F_2$ be embedded closed surfaces carried by $N(B)$ and suppose $F_1\cap F_2\ne\emptyset$.  In general, there are two directions to perform  cutting and pasting along an intersection curve of $F_1\cap F_2$, but only one of them results in a surface still transverse to the $I$-fibers of $N(B)$.  We call this cutting and pasting the \textbf{canonical cutting and pasting}.  This is similar to the Haken sum in normal surface theory.  We use $F_1+F_2$ to denote the surface after the canonical cutting and pasting.  This is a very natural operation, because for any $F_1=(x_1,\dots,x_N)$ and $F_2=(y_1,\dots,y_N)$ in $\mathcal{S}(B)$, $F_1+F_2=(x_1+y_1,\dots,x_N+y_N)$.  Moreover, this sum preserves the Euler characteristic, $\chi(F_1)+\chi(F_2)=\chi(F_1+F_2)$.  For a set of closed surfaces $\mathcal{F}=\{F_1,\dots,F_k\}$ carried by $B$, throughout this paper, we use $\mathcal{S}(\mathcal{F})=\mathcal{S}(F_1,\dots,F_k)$ to denote the set of (possibly disconnected) surfaces of the form $\sum_{i=1}^k n_iF_i$, where each $n_i$ is a non-negative integer.

\begin{definition}
An isotopy of $N(B)$ is called a $B$-\emph{isotopy} if it is invariant on each $I$-fiber of $N(B)$.  We say two surfaces carried by $N(B)$ are $B$-isotopic or $B$-parallel if they are isotopic via a $B$-isotopy of $N(B)$. 
\end{definition}

As we explained in section~\ref{Sintro}, since we are mainly dealing with non-Haken 3-manifolds, we may assume our Heegaard splittings are strongly irreducible.  By \cite{R, St}, we may assume our Heegaard surfaces are normal or almost normal surfaces with respect a triangulation of $M$.  We will refer reader to \cite[Section 2]{L4} and \cite{St} for the definition of almost normal surface.  

Given any normal or almost normal surface, as in \cite{FO}, by identifying all the normal disks of the same type, we obtain a branched surface fully carrying this surface.  Since there are only finitely many normal disk types and almost normal pieces, one can trivially construct a finite collection of branched surfaces such that each normal or almost normal surface is fully carried by a branched surface in this collection.

\begin{proposition}[Proposition 2.5 in \cite{L4}]\label{Pfinite}
There is a finite collection of branched surfaces in $M$ with the following properties.
\begin{enumerate}
\item each branched surface is obtained by gluing normal disks and at most one almost normal piece, similar to \cite{FO},
\item after isotopy, every strongly irreducible Heegaard surface is fully carried by a branched surface in this collection.
\end{enumerate}
\end{proposition}\qed

By \cite{JR}, every irreducible and atoroidal 3-manifold $M$ admits a 0-efficient triangulation unless $M$ is $S^3$ or $RP^3$ or $L(3,1)$.  Since the Heegaard splittings of lens spaces are standard \cite{BO}, we may assume $M$ is not a lens space and admits a 0-efficient triangulation.  By \cite{JR} there is an algorithm to change any triangulation of $M$ into a 0-efficient one.  A 0-efficient triangulation for $M$ has only one vertex and the only normal $S^2$ is the vertex-linking sphere.  Thus, by taking a sub-branched surface if necessary, we may assume no branched surface in Proposition~\ref{Pfinite} carries any normal $S^2$.  

One of the most useful techniques in \cite{JR} is the so-called \emph{barrier surfaces} or \emph{barriers}, which is a surface or complex that is a barrier for the normalization process of a (non-normal) surface.  We refer reader to \cite[Section 3.2]{JR} and \cite[Section 5]{L4} for details.  Next, we give several useful facts well-known to people who are familiar with 0-efficient triangulations.  The proofs of these facts use the barrier technique.

In this section, we assume $M$ is a closed orientable atoroidal 3-manifold with a $0$-efficient triangulation.  We also assume $M$ is not a lens space.

\begin{lemma}\label{LnoS2}
Let $M$ be a 3-manifold with a 0-efficient triangulation as above. In particular $M$ is not a lens space.  Then $M$ does not contain any almost normal $S^2$ and $M$ does not contain any normal or almost normal projective plane.
\end{lemma}
\begin{proof}
Suppose $M$ contains an almost normal 2-sphere $S$.  Since $M$ is irreducible, $S$ bounds a 3-ball $E$ in $M$.  We now try to normalize $S$ in $M-E$. 
By \cite[Theorem 3.2]{JR}, an almost normal surface is a barrier surface for $M-E$.  So either (1) $S$ is isotopic to a normal $S^2$ in $M-\Int(E)$ or (2) $S$ vanishes during the normalization process (i.e. $S$ can be isotoped into a tetrahedron).  In possibility (2), $S$ bounds a 3-ball outside $E$ and hence $M$ is $S^3$, a contradiction on our hypothesis on $M$.  In possibility (1), since the only normal $S^2$ is vertex-linking, $M-\Int(E)$ must also be a 3-ball and $M$ must be $S^3$, a contradiction again.

If $M$ contains a projective plane $P$, then a neighborhood of $P$ is a twisted $I$-bundle over $P$ whose boundary is a 2-sphere $S'$.  Since $M$ is irreducible, $S'$ bounds a 3-ball, which implies that $M$ is $RP^3$, a contradiction to our hypothesis on $M$.
\end{proof}

\begin{lemma}\label{LimmersedS2}
Let $M$ be a 3-manifold with a 0-efficient triangulation as above.  
Let $B$ be a branched surface in $M$ that does not carry the vertex-linking normal 2-sphere.  Then $B$ does not carry any immersed normal or almost normal 2-sphere, i.e., there is no immersed normal or almost normal 2-sphere in $N(B)$ transverse to the $I$-fibers of $N(B)$.
\end{lemma}
\begin{proof}
Suppose there is an immersed normal or almost normal $S^2$ carried by $N(B)$.  Then we perform a canonical cutting and pasting at each double curve of the immersed $S^2$ and eventually we obtain a collection of embedded surfaces carried by $B$.  The canonical cutting and pasting dose not change the Euler characteristic since the surface is immersed.  So the total Euler characteristic of the resulting embedded surface is 2.  This means that a component of the resulting surface must have positive Euler characteristic and hence there is a normal or an almost normal 2-sphere or projective plane carried by $B$.  By Lemma~\ref{LnoS2}, $M$ has no embedded almost normal $S^2$ and has no embedded normal or almost normal projective plane.  Since $B$ does not carry the normal $S^2$ in $M$, we have a contradiction.
\end{proof}

\begin{lemma}[Lemma 5.1 of \cite{L4}]\label{Ltorus}
Suppose $M$ is irreducible and atoroidal and $M$ is not a lens space.  Let $T$ be a normal torus with respect to a 0-efficient triangulation of $M$.  Then, we have the following.
\begin{enumerate}
\item $T$ bounds a solid torus in $M$.  
\item Let $N$ be the solid torus bounded by $T$. Then, $M-\Int(N)$ is irreducible and $T$ is incompressible in $M-\Int(N)$.
\end{enumerate}
\end{lemma}\qed

\begin{corollary}\label{Cmeridion}
Let $T_1$ and $T_2$ be normal tori in $M$ and suppose each curve in $T_1\cap T_2$ is essential in both $T_1$ and $T_2$.  Then a curve in $T_1\cap T_2$ is a meridian of $T_1$ if and only if it is a meridian of $T_2$.
\end{corollary}
\begin{proof}
Let $\gamma$ be a curve of $T_1\cap T_2$.  If $\gamma$ is a meridian of $T_1$, then $\gamma$ bounds an embedded disk $D$ in $M$.  By our hypothesis, $\gamma$ is an essential curve in $T_2$.  Since $D$ is embedded in $M$ and $\gamma$ is an essential curve in $T_2$, if $\gamma$ is not a meridian of $T_2$, then every curve in $D\cap T_2$ must be trivial in $T_2$ and hence we can isotope $D$ (fixing $\gamma$) so that $\Int(D)\cap T_2=\emptyset$ after isotopy. 
By part (2) of Lemma~\ref{Ltorus}, $D$ must be in the solid torus bounded by $T_2$ and hence $\gamma$ must be a meridian of $T_2$, a contradiction.
\end{proof}

\begin{corollary}\label{Cklein}
Suppose $M$ is orientable, irreducible and atoroidal with a 0-efficient triangulation as above.  Suppose $M$ is not a Seifert fiber space.  Then $M$ does not contain any normal or almost normal Klein bottle.
\end{corollary} 
\begin{proof}
Suppose $M$ contains a normal Klein bottle $P$.  Then the boundary of a small neighborhood of $P$ is a normal torus $T$.  By Lemma~\ref{Ltorus}, $T$ bounds a solid torus in $M$.  So $M$ is the union of a solid torus and a twisted $I$-bundle over a Klein bottle, which means that $M$ is a Seifert fiber space.

If $M$ contains an almost normal Klein bottle $P'$, then the boundary of a small neighborhood of $P'$ is a torus $T'$ containing two almost normal pieces.  Next we try to normalize $T'$ in $M-P$.  By \cite[Theorem 3.2]{JR}, $P'$ is a barrier for the normalization process.  Hence, either (1) $T'$ is isotopic to a normal torus in $M-P'$ or (2) some compression occurs in the normalization process, in which case $T'$ becomes a 2-sphere after the compression.  By Lemma~\ref{Ltorus} and since $M$ is irreducible, in both cases, $T'$ bounds a solid torus in $M$.  This means that $M$ the union of a solid torus and a twisted $I$-bundle over a Klein bottle and $M$ is a Seifert fiber space, a contradiction again.
\end{proof}

\begin{definition}\label{Dcomp}
Let $B$ be a branched surface and let $X$ be a component of $M-\Int(N(B))$.  We say $X$ is a \textbf{$D^2\times I$ component} of $M-\Int(N(B))$ if $X$ is a 3-ball with $\partial X$ consisting of a component of $\partial_vN(B)$ and two disk components of $\partial_hN(B)$.  Suppose $X$ is a  $D^2\times I$ component of $M-\Int(N(B))$ and let $\alpha$ be an arc properly embedded in $X$.  We say $\alpha$ is a \textbf{vertical arc} in $X$ if 
\begin{enumerate}
\item $\alpha$ is unknotted (i.e.~$\partial$-parallel) in the 3-ball $X$,
\item the two endpoints of $\alpha$ lie in different disk components of $\partial_hN(B)$.
\end{enumerate}
For a more general component $X$ of $M-\Int(N(B))$, let $A_1,\dots A_n$ be the components of $\partial_vN(B)$ that lie in $\partial X$ and let $c_i$ be a core curve of $\Int(A_i)$ for each $i$.  Suppose $B$ carries (but not necessarily fully carries) a closed surface $S$. 
We say $X$ is an \textbf{almost $D^2\times I$ component} (with respect to $S$) of $M-\Int(N(B))$  if
\begin{enumerate}
\item $c_2,\dots c_n$ bound disjoint disks $\Delta_2,\dots,\Delta_n$ such that each $\Delta_i$ is carried by $N(B)$ and $\Delta_i\cap S=\emptyset$ (such a $\Delta_i$ is usually called a disk of contact, see\cite{FO}), 
\item if we split $N(B)$ along $\Delta_2,\dots,\Delta_n$ and get a fibered neighborhood of a new branched surface $N(B')$, then $X$ becomes a $D^2\times I$ component of $M-\Int(N(B'))$ after the splitting.
\end{enumerate}
Suppose $X$ is an almost $D^2\times I$ component of $M-\Int(N(B))$ and let $\alpha$ be an arc properly embedded in $X$ with $\partial\alpha\subset\partial_hN(B)$.  We say $\alpha$ is a \textbf{vertical arc} in $X$ if after splitting $N(B)$ along $\Delta_2,\dots,\Delta_n$ above, $\alpha$ becomes a vertical arc of the resulting $D^2\times I$ component of $M-\Int(N(B'))$. 
\end{definition}

\begin{definition}\label{Dvert}
Suppose $B$ carries (but not necessarily fully carries) a separating closed surface $S$.  Let $\alpha$ be an arc properly embedded in a component of $\overline{M-S}$.  We say $\alpha$ is an \textbf{almost vertical arc} with respect to $B$ and $S$ if  
\begin{enumerate}
\item $\alpha\cap N(B)$ consists of vertical arcs in $N(B)$ (i.e. subarcs of $I$-fibers of $N(B)$) 
\item there are a collection of disjoint subarcs $\alpha_1,\dots,\alpha_n$ of $\alpha$ such that $\alpha-\bigcup_{i=1}^n\alpha_i\subset N(B)$ and each $\alpha_i$ is a vertical arc of either a $D^2\times I$ component or an almost $D^2\times I$ component (with respect to $S$) of $M-\Int(N(B_i))$ for some sub-branched surface $B_i$ of $B$ that carries $S$.
\end{enumerate}
The second requirement makes sense if we view a sub-branched surface $B_i$ of $B$ as a branched surface obtained by deleting certain branch sectors from $B$, and view $N(B_i)$ as the object obtained by deleting from $N(B)$ the part corresponding to the deleted branch sectors and then smoothing out the corners.  The reason for this technical requirement will become clear in the next section.  We define the \textbf{length} of $\alpha$ to be $n+1$, where $n$ is the number of subarcs $\alpha_i$'s in part (2) of the definition.  Note that we allow $n$ to be 0, and $n=0$ if and only if $\alpha$ is a vertical arc in $N(B)$. 
Moreover, let $B_s$ be the sub-branched surface of $B$ that fully carries $S$.  Recall that $\alpha$ is properly embedded in $\overline{M-S}$, so it follows from the definition that either $\alpha$ is a vertical arc in $N(B)$, or after we have isotoped $S$ so that part of $\partial_hN(B_s)$ lies in $S$, we can make $\alpha$ to be an arc properly embedded in a component of $M-\Int(N(B_s))$.  
\end{definition}

\begin{definition}\label{Dai}
Let $B$ be a branched surface and $S$ a closed separating surface carried by $B$.  We say that $S$ is an \textbf{almost Heegaard surface} with respect to $B$ if there are finitely many disjoint arcs $\alpha_1,\dots,\alpha_k$ such that 
\begin{enumerate}
\item each $\alpha_i$ is an almost vertical arc with respect to $S$ and $B$, 
\item after adding tubes to $S$ along the $\alpha_i$'s, we get a Heegaard surface of $M$.
\end{enumerate}
Adding a tube along $\alpha_i$ is the operation deleting from $S$ two small disks that contain $\partial\alpha_i$ and then connecting the resulting boundary circles using a small tube/annulus along $\alpha_i$. 
We call  these $\alpha_i$'s the set of (almost vertical) arcs \textbf{associated} to $S$ and we call the Heegaard surface obtained by adding such tubes to $S$ the Heegaard surface \textbf{derived} from $S$.  If the Heegaard surface derived from $S$ is strongly irreducible, then we say $S$ is an \textbf{almost strongly irreducible Heegaard surface} with respect to $B$.  Note that the definition of almost Heegaard surface is basically for a pair $(S,J)$, where $J$ is the set of arcs $\alpha_i$'s above, but to simplify notation, we use only $S$ to denote the pair $(S,J)$.  Moreover, to simplify notation, we will regard a Heegaard surface as an almost Heegaard surface (with the set of associated almost vertical arcs $J=\emptyset$).
\end{definition}

Note that an almost strongly irreducible Heegaard surface $S$ may not be connected.  If we assume $S$ is normal or almost normal and assume our branched surface $B$ does not carry any normal or almost normal 2-sphere as above, then no component of $S$ is a 2-sphere.

Scharlemann's no-nesting lemma \cite[Lemma 2.2]{S} says that if a simple closed curve in a strongly irreducible Heegaard surface bounds an embedded disk in $M$, then it must bound a disk properly embedded in one of the two handlebodies (or compression bodies) in the Heegaard splitting.  The following lemma is a mild extension of Scharlemann's lemma to an almost strongly irreducible Heegaard surface.

\begin{lemma}\label{Lnonesting}
Let $S\subset N(B)$ be a surface carried by $B$ and suppose $S$ is an almost strongly irreducible Heegaard surface with respect to $B$.  Let $\gamma$ be an essential simple closed curve in $S$ that bounds an embedded disk in $M$.  Then $\gamma$ bounds a compressing disk for $S$.
\end{lemma}
\begin{proof}
By Definition~\ref{Dai}, there is a collection of almost vertical arcs $J$ associated to $S$ such that if we add tubes along these arcs, we obtain a strongly irreducible Heegaard surface $S'$.  
We can choose these arcs $J$ so that their endpoints are not in $\gamma$ and hence we may view $\gamma\subset S'$.  By Scharlemann's no-nesting lemma \cite[Lemma 2.2]{S}, $\gamma$ bounds a compressing disk $D_\gamma$ for $S'$. 
Since $S'$ is the surface obtained by adding tubes along arcs in $J$ and since $\gamma$ is disjoint from $J$, we may assume the disk $D_\gamma$ is disjoint from these tubes (and after isotopy, disjoint from the meridional disks of these tubes). As we can obtain $S$ from $S'$ by compressing $S'$ along the meridional disks of these tubes, we may view $D_\gamma$ as a disk with $\partial D_\gamma=\gamma\subset S$ and $\Int(D_\gamma)\cap S=\emptyset$.  By our hypothesis that $\gamma$ is an essential curve in $S$, this means that $D_\gamma$ is a compressing disk for $S$.
\end{proof}

\begin{lemma}\label{L3-ball}
Suppose $S$ is either a strongly irreducible Heegaard surface or a surface obtained by compressing a strongly irreducible Heegaard surface.  In the later case we suppose $S$ does not contain any 2-sphere component.  Then no component of $S$ is contained in a 3-ball in $M$.
\end{lemma}
\begin{proof}
The case that $S$ is a strongly irreducible Heegaard surface is trivial, since a handlebody is irreducible and $M\ne S^3$.  Suppose $S$ is obtained by compressing a strongly irreducible Heegaard surface $S'$.   Let $H_1$ and $H_2$ be the two handlebodies in the Heegaard splitting along $S'$.  Since $S'$ is strongly irreducible, the compressions occur only on one side, say in $H_2$.  So one side of $S$ is a union of handlebodies lying in $H_2$ and the other side of $S$, denoted by $W$, is obtained by adding 2-handles to $H_1$.  Since no component of $S$ is a 2-sphere, $\partial W$ has no sphere component and $S'$ can be viewed as a strongly irreducible Heegaard surface of the manifold with boundary $W$.  By \cite{CG}, $W$ is irreducible.  So none of the two submanifolds of $M$ bounded by $S$ is reducible.  Since $M$ is not $S^3$, this implies that no component of $S$ is contained in a 3-ball.
\end{proof}

\begin{definition}\label{Dkmin}
Let $S_i$ ($i=1,2$) be an almost Heegaard surface carried by $B$, let $J_i$ be the almost vertical arcs associated to $S_i$, and let $S'_i$ be the Heegaard surface obtained by adding tubes to $S_i$ along arcs in $J_i$.  We say $(S_1, J_1)$ and $(S_2, J_2)$ are similar if the Heegaard surfaces $S_1'$ and $S_2'$ are isotopic.
Suppose the total length of the arcs in $J_1$ is bounded from above by a number $K$ (note that this implies that the number of components of $J_1$ is at most $K$).  We say $(S_1, J_1)$ or simply say $S_1$ is \textbf{$K$-minimal} if, for any $(S_2, J_2)$ similar to $(S_1, J_1)$ and with $length(J_2)\le K$, we have $weight(S_1)\le weight(S_2)$, where $weight(S_1)$ is the weight of $S_1$ with respect to $B$ defined at the beginning of this section.
\end{definition}

We finish this section with following observation.

\begin{lemma}\label{Lfinite}
Let $B$ be a branched surface in $M$ as above. In particular, $B$ does not carry any normal or almost normal 2-sphere. 
Let $S\subset N(B)$ be a normal or an almost normal surface carried by $N(B)$, and we suppose $S$ is either a strongly irreducible Heegaard surface or an almost strongly irreducible Heegaard surface for $B$.  If $S$ is an almost Heegaard surface, we suppose the total length of the almost vertical arcs associated to $S$ is bounded from above by a fixed number $K$ and suppose $S$ is $K$-minimal. Let $A\subset N(B)$ be a vertical annulus in $N(B)$.  Suppose $A\cap S$ consists of simple closed curves that are essential in $S$.  Suppose a core curve of $A$ bounds an embedded disk $D$ carried by $N(B)$ and $D$ does not pass through the possible almost normal piece in $B$.  Then there is an number $k$ depending on $B$, $M$ and the fixed number $K$ above, such that $|A\cap S|\le k$.  Moreover, $k$ can be algorithmically calculated.
\end{lemma}
\begin{proof}
Let $\gamma\subset A$ be a core curve of $A$ with $\partial D=\gamma$. For any point $x$ in $\gamma$, we fix a normal direction for $A$ at $x$ pointing into the disk $D$.  We call it the positive direction for $A$.  

Let $P\subset N(B)$ be a compact subsurface of $S$ with $\partial P\subset A$.  We say that $P$ is on the negative side of $A$ if for each point $x\in\partial P$, the direction pointing from $x$ into $P$ is the negative direction for $A$.  

\vspace{8pt}
\noindent
\emph{Claim 1}. There is no planar subsurface $P$ of $S$ as above on the negative side of $A$. 
\begin{proof}[Proof of Claim 1]
Suppose there is a planar surface $P$ as above on the negative side of $A$.  Then we can cap off each curve in $\partial P$ using a disk $B$-isotopic to $D$ and obtain a (possibly immersed) 2-sphere carried by $B$.  Since $D$ does not pass through the almost normal piece, the immersed 2-sphere is normal or almost normal, which contradicts Lemma~\ref{LimmersedS2}.
\end{proof}

Suppose $S$ splits $M$ into two submanifolds $H_1$ and $H_2$.  If $S$ is an almost Heegaard surface then $H_i$ may not be connected, but by our assumption on $B$ and $S$, no component of $S$ is a 2-sphere and no component of $H_i$ is a 3-ball.  If $S$ is an almost strongly irreducible Heegaard surface, since $A$ is vertical in $N(B)$, we may assume the almost vertical arcs associated to $S$ are disjoint from $A$, see Definition~\ref{Dvert}.

Since each component of $A\cap S$ is an essential curve in $S$ and bounds an embedded disk $B$-isotopic to the disk $D$ above, by Scharlemann's no-nesting lemma \cite[Lemma 2.2]{S} and by Lemma~\ref{Lnonesting}, each curve in $A\cap S$ must bound a compressing disk in $H_1$ or $H_2$ ($M=H_1\cup_SH_2$).  
Since $S$ is either a strongly irreducible Heegaard surface or an almost strongly irreducible Heegaard surface, curves in $A\cap S$ cannot bound compressing disks in both $H_1$ and $H_2$.  Thus we may assume the compressing disks (for $S$) bounded by $A\cap S$ all lie in $H_2$.  

The curves $A\cap S$ cut $A$ into a collection of subannuli properly embedded in $H_1$ and $H_2$.  Let $A_1,\dots, A_p$ be those subannuli of $A$ properly embedded in $H_1$ and clearly $p\ge\frac{1}{2}|A\cap S|-1$.  

\vspace{8pt}
\noindent
\emph{Claim 2}. Each annulus $A_i$ ($i=1,\dots, p$) above is $\partial$-parallel in $H_1$.
\begin{proof}[Proof of Claim 2]
Suppose on the contrary that $A_i$ is not $\partial$-parallel in $H_1$. 

If $S$ is a strongly irreducible Heegaard surface, then $A_i$ is $\partial$-compressible in $H_1$ and after a $\partial$-compression on $A_i$, we obtain a compressing disk for $H_1$ disjoint from $A_i$.  However, this contradicts that $S$ is a strongly irreducible Heegaard surface because $\partial A_i$ bounds compressing disks in $H_2$.

Next we suppose $S$ is an almost strongly irreducible Heegaard surface.  Let $J$ be the collection of almost vertical arcs associated to $S$ and let $S'$ be the strongly irreducible Heegaard surface obtained by adding tubes to $S$ along arcs in $J$.  Let $H_1'$ and $H_2'$ be the two handlebodies in the Heegaard splitting of $M$ along $S'$ corresponding to $H_1$ and $H_2$ respectively.  Since we have assumed that $J$ is disjoint from $A$, we may view $A_i$ as an annulus properly embedded in $H_1'$.  By our assumption on $A\cap S$ above, each component of $\partial A_i$ bounds a compressing disk in $H_2'$.  Since $S'$ is a strongly irreducible Heegaard surface, as in the argument above, $A_i$ must be $\partial$-parallel in $H_1'$.  Let $\Gamma_i\subset S'$ be the annulus in $S'$ bounded by $\partial A_i$ and parallel to $A_i$ in $H_1'$.  Since we have assumed $A_i$ (as an annulus in $H_1$) is not $\partial$-parallel in $H_1$, $\Gamma_i$ must contain a tube that we added to $S$.  Next we compress $S'$ along the meridional disks of these tubes to get back the surface $S$.   By our construction, since $S$ has no 2-sphere component, the meridional curve of each tube is an essential curve in $S'$.  As the annulus $\Gamma_i$ contains some tube, the compressions on $S'$ changes $\Gamma_i$ into a pair of disks.  This means that $\partial A_i=\partial\Gamma_i$ is a pair of trivial circles in $S$, contradicting our hypothesis that $A\cap S$ consists of essential curves in $S$.  Thus Claim 2 is also true if $S$ is an almost strongly irreducible Heegaard surface.
\end{proof}

Let $\Gamma_i\subset S$ be the annulus in $S$ bounded by $\partial A_i$ and parallel to $A_i$ in $H_1$.  Since $A_i$ is a vertical annulus in $N(B)$ and $\Gamma_i$ is transverse to the $I$-fibers of $N(B)$, the solid torus bounded by $A_i\cup\Gamma_i$ in $H_1$ must contain a component of $M-N(B)$.  Since $|M-N(B)|$ is bounded and the $A_i$'s are disjoint in $H_1$, if $|A\cap S|$ is sufficiently large, some of the $\Gamma_i$'s must be nested in $S$.  Without loss of generality, we suppose $\Gamma_1\subset\Gamma_2\subset\dots\subset \Gamma_q$.  By assuming $|A\cap S|$ to be large, we may also choose these $\Gamma_i$'s so that $(\Int(\Gamma_i)-\Gamma_{i-1})\cap A=\emptyset$ for each $1<i\le q$.  Note that $q$ is large if $|A\cap S|$ is large.

Since each $\Gamma_i$  is parallel to $A_i$ in $H_1$, a small neighborhood of $\partial\Gamma_j$ in $\Gamma_j$ (for any $j$) must be a pair of annuli lying on the same side of $A$.  By Claim 1, $\Gamma_i$ is not on the negative side of $A$.  Hence a small neighborhood of $\partial\Gamma_j$ in $\Gamma_j$  must be a pair of annuli lying on the positive side of  $A$, for each $j$.  
This means that, for any $j=2,\dots,q$, each of the two annular components of $\Gamma_{j}-\Int(\Gamma_{j-1})$ connects the positive side of $A$ to the negative side of $A$, as shown in Figure~\ref{wrap}(a).  Let $\Gamma'$ be a component of $\Gamma_q-\Int(\Gamma_1)$ and let $\Gamma_j'$ ($j=2,\dots,q$) be the component of $\Gamma_{j}-\Int(\Gamma_{j-1})$ lying in $\Gamma'$.  Since $S$ contains at most one almost normal piece, we may choose $\Gamma'$ to be the component of $\Gamma_q-\Int(\Gamma_1)$ that does not contain an almost normal piece. 
Let $A_j'\subset A$ be the subannulus of $A$ bounded by $\partial\Gamma_j'$.  By our assumption that $(\Int(\Gamma_i)-\Gamma_{i-1})\cap A=\emptyset$ for each $1<i\le q$, we have $A\cap\Gamma_j'=\partial\Gamma_j'$ and hence $\Gamma_j'\cup A_j'$ is an embedded torus or Klein bottle in $M$.  Since $\Gamma_j'$ connects the positive side of $A$ to the negative side of $A$, $T_j'=\Gamma_j'\cup A_j'$ can be perturbed slightly into a surface $T_j$ transverse to the $I$-fibers of $N(B)$, see Figure~\ref{wrap}(c).   
By our assumption that $\Gamma'$ does not contain an almost normal piece, each $T_j$ is a normal torus or Klein bottle.  By Corollary~\ref{Cklein}, $T_j$ cannot be a Klein bottle.  Hence $T_j$ is a normal torus carried by $B$ and by Lemma~\ref{Ltorus} each $T_j$ bounds a solid torus in $M$.  

So each $T_i'$ bounds a solid torus and the annulus $A_i'\subset T_i'$ is part of the boundary of the solid torus.  Since each $\Gamma_j'$ connects the positive side of $A$ to the negative side of $A$ and $A\cap\Gamma_j'=\partial\Gamma_j'$, $\Gamma_{i-1}'$ lies either totally in the solid torus bounded by $T_i'$ or totally outside the solid torus.  As shown in Figure~\ref{wrap}(b, c), this implies that these tori $T_i$'s ($j=2,\dots,q$) are disjoint after a small perturbation.  A theorem of Kneser \cite{Kn} says that a compact 3-manifold contains only finitely many disjoint non-parallel normal surfaces.  Thus if $q$ is large, $T_j$ and $T_{j-1}$ are $B$-isotopic for some $j$,  see Figure~\ref{wrap}(c).  This means that $\Gamma_j'$ and $\Gamma_{j-1}'$ are $B$-isotopic and the annulus $\Gamma_j'\cup\Gamma_{j-1}'$ wraps around the normal torus $T_j$ more than once, see Figure~\ref{wrap}(b).  Thus we can unwrap it by a Dehn twist on $T_j$ (which is an isotopy in $M$ because $T_j$ bounds a solid torus) and get a surface $S''$ isotopic to $S$ with smaller weight.

Note that in the argument above, if $q$ is large, then the number of $B$-parallel tori $T_i$'s is large.
If $S$ is an almost Heegaard surface, the almost vertical arcs associated to $S$ are properly embedded in $\overline{M-S}$ and (by our assumptions above) are disjoint from the vertical annulus $A$.  By the definition of $K$-minimal, see Definition~\ref{Dkmin}, the total number of the almost vertical arcs associated to $S$ is at most $K$, thus by assuming $q$ to be sufficiently large, we may suppose the number of $B$-parallel tori $T_i$'s is so large that we can choose the annulus $\Gamma_j'\cup\Gamma_{j-1}'$ above to be disjoint from the almost vertical arcs associated to $S$.  Hence the Dehn twist on $T_j$ above does not affect the arcs associated to $S$.  This means that $S''$ is an almost Heegaard surface isotopic to $S$ and with the same set of almost vertical arcs.  Furthermore, the two Heegaard surfaces derived from $S$ and $S''$ are isotopic, since the Dehn twist is an isotopy on $M$.  This contradicts the hypothesis that $S$ is $K$-minimal.

Therefore $q$ and $|A\cap S|$ must be bounded by a number $k$ that depends on $B$, $M$ and $K$, and it follows from the proof above that $q$ and $k$ can be algorithmically calculated.
\end{proof}

\begin{figure}
\begin{center}
\psfrag{(a)}{(a)}
\psfrag{(b)}{(b)}
\psfrag{(c)}{(c)}
\psfrag{A}{$A$}
\psfrag{G'}{$\Gamma'$}
\psfrag{Gi}{$\Gamma_i$}
\psfrag{Gi1}{$\Gamma_{i-1}$}
\psfrag{Tj}{$T_j$}
\psfrag{Tj1}{$T_{j-1}$}
\includegraphics[width=4.0in]{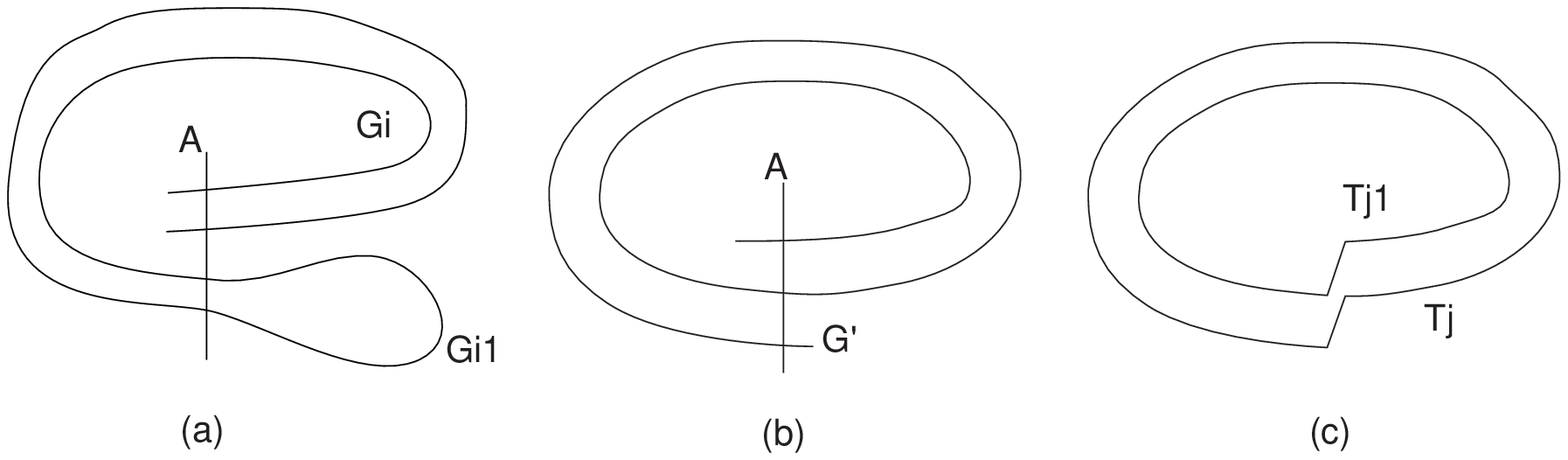}
\caption{}\label{wrap}
\end{center}
\end{figure}

\section{$D^2\times I$ regions for a surface carried by a branched surface}\label{SD2}

\begin{notation}\label{NB}
Throughout this paper, we assume our manifold $M$ is not a Seifert fiber space and admits a 0-efficient triangulation. 
 Unless specified, we use $B$ to denote a branched surface obtained by gluing normal disks and at most one almost normal piece as in Proposition~\ref{Pfinite}.  As in section~\ref{Spre}, we may assume $B$ does not carry any normal $S^2$.  By Lemma~\ref{LnoS2}, $B$ does not carry any almost normal 2-sphere neither.  
\end{notation}

The goal of this section is to eliminate certain ``bubbles" called $D^2\times I$ regions (see Definition~\ref{Dball} below) for a strongly irreducible Heegaard surface carried by $B$.  To achieve this, we have to compress the Heegaard surface into an almost Heegaard surface for $B$ (see Definition~\ref{Dai}).  Moreover, we also need a bound on the length of the associated almost vertical arcs (see Definitions~\ref{Dvert} and \ref{Dai}) to be able to algorithmically recover the original Heegaard surface in the end.

\begin{definition}\label{Dball}
Let $B$ be a branched surface as above and let $S$ be an orientable surface carried by $N(B)$.  Let $A\subset N(B)$ be a vertical annulus with $\partial A\subset S$.  Suppose both curves in $\partial A$ are trivial in $S$.  Let $D_0$ and $D_1$ be the two disks bounded by $\partial A$ in $S$. Suppose $D_0\cup A\cup D_1$ is an embedded sphere bounding a 3-ball $E$ in $M$.  We call $E$ a \emph{$D^2\times I$ region} for $S$ and $B$. We call $D_0\cup D_1$ the horizontal boundary of $E$, denoted by $\partial_hE$ and call $A$ the vertical boundary of $E$, denoted by $\partial_vE$. If $E\subset N(B)$, i.e. $D_0$ is $B$-isotopic to $D_1$, then we say the $D^2\times I$ region is trivial, otherwise we say $E$ is non-trivial.  If $E\cap S=D_0\cup D_1$ (i.e.~$\Int(E)\cap S=\emptyset$), then we say $E$ is a \emph{simple} $D^2\times I$ region, otherwise we call $E$ a \emph{stuffed $D^2\times I$ region}.  We say $E$ is a \emph{good} $D^2\times I$ region if $S\cap E$ consists of disks with boundary circles in $A$.  To simplify notation, we also say that the $D^2\times I$ region is bounded by $D_0\cup D_1$ or bounded by $A$.  
\end{definition}

There is a slight ambiguity for a $D^2\times I$ region depending on whether $D_0$, $D_1$ and $E$ in Definition~\ref{Dball} are on the same side of $A$.  We say $E$ is of type I if after collapsing $A$ into a circle, $D_0\cup A\cup D_1$ becomes a 2-sphere with a cusp pointing out of $E$, see Figure~\ref{type} for a 1-dimensional schematic picture.  We say $E$ is of type II if after collapsing $A$ into a circle, $D_0\cup A\cup D_1$ becomes a 2-sphere with a cusp pointing into $E$, see Figure~\ref{type}.  We say $E$ is of type III if after a small perturbation, $D_0\cup A\cup D_1$ becomes a 2-sphere carried by $N(B)$, see Figure~\ref{type}.  

\begin{figure}
\begin{center}
\psfrag{I}{type I}
\psfrag{II}{type II}
\psfrag{III}{type III}
\psfrag{D0}{$D_0$}
\psfrag{D1}{$D_1$}
\psfrag{A}{$A$}
\psfrag{E}{$E$}
\includegraphics[width=4.0in]{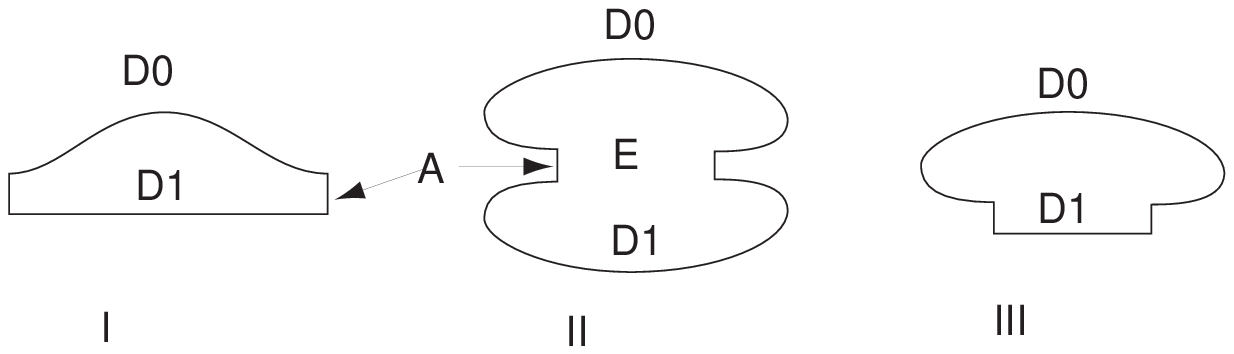}
\caption{}\label{type}
\end{center}
\end{figure}

\begin{lemma}\label{LtypeI}
Let $S\subset N(B)$ be a normal or an almost normal surface carried by $N(B)$, and  suppose $S$ is either a strongly irreducible Heegaard surface or an almost strongly irreducible Heegaard surface. Then every $D^2\times I$ region for $S$ is of type I.
\end{lemma}
\begin{proof}
First, since $B$ does not carry any normal or almost normal $S^2$, type III $D^2\times I$ region does not exist. 
Suppose the lemma is false and there is a  type II $D^2\times I$ region $E$. Let $\partial_hE=D_0\cup D_1$ and $\partial_vE=A$ be as in Definition~\ref{Dball}.  If $A\cap S=\partial A$, then since $E$ is of type II, we can enlarge $E$ into a 3-ball $E'$ so that $\partial E'\cap S=\emptyset$.  However, since $E$ is of type II, $\Int(E)\cap S\ne\emptyset$.  As  $\partial E'\cap S=\emptyset$, this means that a component of $S$ lies in the 3-ball $E'$, which contradicts Lemma~\ref{L3-ball}.   So we may suppose $\Int(A)\cap S\ne\emptyset$.  Let $n=|\Int(A)\cap S|$.  Suppose $n$ is minimal among all type II $D^2\times I$ regions.  

By our construction of $B$, the possible almost normal piece in $B$ totally lies in a branch sector of $B$.  If $S$ is an almost normal surface, the weight of $S$ at the branch sector that contains the almost normal piece is one.  Since $A$ is a vertical annulus in $N(B)$ with $\partial A\subset S$, this implies that if $D_i$ intersects the almost normal piece then it must contain the whole almost normal piece.
As $S$ is normal or almost normal, $D_0$ and $D_1$ cannot both contain almost normal pieces.  So we may suppose $D_0$ does not intersect an almost normal piece.

\vspace{8pt}
\noindent
\emph{Claim}. There is no connected planar subsurface $P$ of $S$ such that $P\cap A=\partial P$, $\partial D_i\subset\partial P$ ($i=0$ or 1) and $P\ne D_i$.
\begin{proof}[Proof of the Claim]
The proof is similar to the proof of Claim 1 of Lemma \ref{Lfinite}. 
Suppose there is such a planar subsurface as in the Claim. As $P\cap A=\partial P$ and $P\ne D_i$, $P$ is properly embedded in either $E$ or $M-\Int(E)$.  Since $E$ is a type II $D^2\times I$ region and $\partial D_i\subset\partial P$ for some $i$ with $P\ne D_i$, $P$ must be properly embedded in $E$.  In particular, a neighborhood of $\partial P$ in $P$ and a neighborhood of $\partial D_i$ in $D_i$ lie on different sides of $A$. This means that we can construct a 2-sphere carried by $N(B)$ by capping off each circle in $\partial P$ using a disk $B$-isotopic to $D_0$.  Since $D_0$ does not contain an almost normal piece, we get a normal or an almost normal $S^2$.  This contradicts Lemma~\ref{LimmersedS2}.  
\end{proof}

Let $\gamma_0=\partial D_0$, $\gamma_1,\dots, \gamma_n$, $\gamma_{n+1}=\partial D_1$ be the curves of $A\cap S$ and we suppose $\gamma_i$ lies between $\gamma_{i-1}$ and $\gamma_{i+1}$ for each $i$.
We first consider the case that some $\gamma_i$ ($1\le i\le n$) is trivial in $S$.  Let $\Delta_i$ be the subdisk of $S$ bounded by $\gamma_i$.  If $\Int(\Delta_i)$ is disjoint from $A$, then $\Delta_i\cup D_0$ and $\Delta_i\cup D_1$ (together with subannuli of $A$) bound two $D^2\times I$ regions.  Since the $D^2\times I$ region $E$ is of type II and there is no type III $D^2\times I$ region, at least one of the two $D^2\times I$ regions is of type II.  This contradicts our hypothesis that $n=|\Int(A)\cap S|$ is minimal among all type II $D^2\times I$ regions.  So we may suppose $\Int(\Delta_i)\cap A\ne\emptyset$.  Let $\gamma_j$ be a component of $\Int(\Delta_i)\cap A$ that is innermost in $\Delta_i$.  By applying the argument for $\gamma_i$ above to $\gamma_j$, we can conclude that $\gamma_j$ must be either $\gamma_0=\partial D_0$ or $\gamma_{n+1}=\partial D_1$. Without loss of generality, we may suppose $\gamma_j=\gamma_0=\partial D_0$ which means that $D_0\subset \Delta_i$.  However, this implies that $\Delta_i$ contains a planar subsurface $P$ as in the Claim above and hence this cannot happen.  Thus every $\gamma_i$ ($1\le i\le n$) must be essential in $S$.

Suppose $S$ splits $M$ into two submanifolds $H_1$ and $H_2$.  If $S$ is an almost Heegaard surface then $H_i$ may not be connected, but by our assumption on $B$ and $S$, no component of $S$ is a 2-sphere and no component of $H_i$ is a 3-ball.  If $S$ is an almost strongly irreducible Heegaard surface, since $A$ is vertical in $N(B)$, we may assume the almost vertical arcs associated to $S$ are disjoint from $A$.

Let $A_i$ be the subannulus of $A$ between $\gamma_i$ and $\gamma_{i+1}$.  We may suppose $A_i$ is properly embedded in $H_1$ if $i$ is odd and properly embedded in $H_2$ if $i$ is even.  
So $A_0\cup D_0$ gives a disk in $H_2$ bounded by $\gamma_1$.  Since we have concluded that each component of $\Int(A)\cap S$ is an essential curve in $S$, by Scharlemann's no-nesting lemma \cite[Lemma 2.2]{S} and by Lemma~\ref{Lnonesting}, each curve in $\Int(A)\cap S$ must bound a compressing disk in $H_1$ or $H_2$ ($M=H_1\cup_SH_2$).  Since $\gamma_1$ bounds a disk in $H_2$ and since $S$ is either a strongly irreducible Heegaard surface or an almost strongly irreducible Heegaard surface, every $\gamma_i$ ($1\le i\le n$) bounds a compressing disk in $H_2$.  
As in Claim 2 of Lemma~\ref{Lfinite}, $A_{2k+1}$ must be $\partial$-parallel in $H_1$ for each $k$.

Let $\Gamma_{2k+1}\subset S$ be the annulus in $S$ with $\partial\Gamma_{2k+1}=\partial A_{2k+1}$ and parallel to $A_{2k+1}$ in $H_1$.  If $\Gamma_{2k+1}$ contains $D_0$ or $D_1$, then $\Gamma_{2k+1}-(D_0\cup D_1)$ contains a planar surface as in the Claim above.  So we may assume no $\Gamma_{2k+1}$ contains $D_0$ nor $D_1$.  

Since each curve in $S\cap\Int(A)$ is essential in $S$, $\Gamma_{2k+1}\cap A$ consists of curves essential in $\Gamma_{2k+1}$.  
Thus if $\Int(\Gamma_{2k+1})\cap A\ne\emptyset$, then one can always find a subannulus $P'$ of $\Gamma_{2k+1}$ properly embedded in the $D^2\times I$ region $E$.  Since $E$ is of type II,  as in the proof of the claim, one can obtain a normal or an almost normal 2-sphere carried by $B$ by capping off each curve in $\partial P'$ using a disk $B$-isotopic to $D_0$, which contradicts Lemma~\ref{LimmersedS2}.  So $\Int(\Gamma_{2k+1})\cap A=\emptyset$.  Moreover, this argument also implies that each $\Gamma_{2k+1}$ must be properly embedded in $M-\Int(E)$, since $E$ is of type II.

Recall that $\Gamma_{2k+1}$ is parallel to $A_{2k+1}$ in $H_1$ for each $k$. 
Let $T_k$ be the solid torus in $H_1$ bounded by $A_{2k+1}\cup\Gamma_{2k+1}$.  Since $\Gamma_{2k+1}$ is properly embedded in $M-\Int(E)$, $T_k$ must lie in $M-\Int(E)$ with $T_k\cap\partial E=A_{2k+1}$.  Moreover, $E\cup T_k$ is a 3-ball.  Let $E'$ be the union of $E$ and all these solid tori $T_k$.  So $E'$ is a 3-ball and by the definition of $A_i$, $\partial E'$ is the union of $D_0\cup D_1$, the $\Gamma_{2k+1}$'s and the $A_{2k}$'s.  We can enlarge $E'$ slightly into a 3-ball $E''$ to enclose $D_0\cup D_1$ and these $\Gamma_{2k+1}$'s.  Since $E$ is a type II $D^2\times I$ region, this means that $\partial E''\cap S=\emptyset$.  Hence a component of $S$ lies in the 3-ball $E''$, which again contradicts Lemma~\ref{L3-ball}. 
Therefore there is no type II $D^2\times I$ region.
\end{proof}

\begin{lemma}\label{Lstuff}
Let $S\subset N(B)$ be a normal or an almost normal surface carried by $N(B)$, and suppose $S$ is either a strongly irreducible Heegaard surface or an almost strongly irreducible Heegaard surface.  
Let $E$ be a stuffed $D^2\times I$ region for $S$ with $\partial_h E=D_0\cup D_1$ and $\partial_vE=A$, as in Definition~\ref{Dball}.  Then each component of $S\cap E$ is either a disk with boundary in $A$ or an unknotted annulus which is $\partial$-parallel in $E$ to a subannulus of $A$.  Moreover, the annuli in $E\cap S$ are non-nested in the sense that they are isotopic (relative to the boundary) to a collection of non-nested subannuli of $A$.
\end{lemma}
\begin{proof} First, by Lemma~\ref{LtypeI}, $E$ is of type I. 
The proof is almost identical to the proof of Lemma~\ref{LtypeI}.  However, since the $D^2\times I$ region in the proof of Lemma~\ref{LtypeI} is of type II and $E$ is of type I here, we need to interchange the roles of $E$ and $M-E$ in the argument.  

Let $\gamma_i$ and $A_i$ ($i=0,\dots, n+1$, $\gamma_0=\partial D_0$ and $\gamma_{n+1}=\partial D_n$) be as in the proof of Lemma~\ref{LtypeI}.  The first case is that some $\gamma_i$ ($1\le i\le n$) bounds a disk $\Delta_i$ in $S$.  If $\Delta_i-E\ne\emptyset$, then as in the proof of the claim in Lemma~\ref{LtypeI}, $\Delta_i-E$ is a planar surface in $M-E$ and we can obtain a normal or an almost normal 2-sphere carried by $B$ by capping off its boundary curves using disks $B$-parallel to $D_0$, contradicting Lemma~\ref{LimmersedS2}.
So we can conclude that the disk $\Delta_i$ is properly embedded in $E$ and $\Delta_i$ cuts $E$ into a pair of smaller $D^2\times I$ regions.  Thus by taking a sub-$D^2\times I$ region of $E$ if necessary, we may assume that each $\gamma_i$ ($1\le i\le n$) is essential in $S$.

Using the same notation as in the proof of Lemma~\ref{LtypeI} and by the same argument, we may assume each $A_{2k+1}$ is $\partial$-parallel in $H_1$.  Let $\Gamma_{2k+1}\subset S$ be the annulus in $S$ bounded by $\partial A_{2k+1}$ and parallel to $A_{2k+1}$ in $H_1$.  As in the proof of Lemma~\ref{LtypeI}, if $\Gamma_{2k+1}-E\ne\emptyset$, then $\Gamma_{2k+1}-E$ is a planar surface in $M-E$ and we can obtain a normal or an almost normal 2-sphere carried by $B$ by capping off its boundary curves using disks $B$-parallel to $D_0$, which contradicts Lemma~\ref{LimmersedS2}.  So $\Gamma_{2k+1}-E=\emptyset$ and $\Gamma_{2k+1}$ must be properly embedded in $E$.  Hence $S\cap E$ consists of these unknotted annuli $\Gamma_{2k+1}$'s as in the proof of Lemma~\ref{LtypeI}.  As each $\Gamma_{2k+1}$ is parallel to $A_{2k+1}$ in $E$, these annuli are non-nested and Lemma~\ref{Lstuff} holds.
\end{proof}

\begin{corollary}\label{Cstuff}
Let $S$ be a normal or an almost normal strongly irreducible Heegaard surface carried by $B$.  Then there is a normal or an almost normal surface $S'$ carried by $B$ such that 
\begin{enumerate}
\item $S'$ is an almost strongly irreducible Heegaard surface and $S$ can be derived from $S'$, see Definition~\ref{Dai}, 
\item  every $D^2\times I$ region for $S'$ is a good $D^2\times I$ region, see Definition~\ref{Dball}.
\end{enumerate}
\end{corollary}
\begin{proof}
Corollary~\ref{Cstuff} follows from Lemma~\ref{Lstuff}.  Let $E$ be a stuffed $D^2\times I$ region for $S$ with $\partial_h E=D_0\cup D_1$ and $\partial_vE=A$.  Since $S$ is normal or almost normal, we may assume one of the two disks in $\partial_hE$, say $D_0$, does not contain an almost normal piece.  By Lemma~\ref{Lstuff}, $S\cap E$ consists of disks and a collection of unknotted  $\partial$-parallel annuli in $E$.  So we can compress each unknotted annulus in $E$ into a pair of disks, and then push all the disks in $\Int(E)$ into disks $B$-isotopic to $D_0$, see Figure~\ref{iso2}(a) for a schematic picture (ignore the vertical dashed arc in the picture after isotopy for now).  Let $S_1$ be the resulting surface and clearly $E\cap S_1$ consists of disks.  Since $D_0$ does not contain an almost normal piece, $S_1$ is either normal or almost normal.  Moreover, since the annuli in $E\cap S$ are non-nested in the sense of Lemma~\ref{Lstuff}, one can recover $S$ from $S_1$ by adding tubes to $S_1$ along some vertical arcs in $N(B)$ connecting disk components of $S_1\cap E$, see the dashed vertical arc in Figure~\ref{iso2}(a) for a picture. So $S_1$ is an almost Heegaard surface and $S$ can be derived from $S_1$.  By repeating this argument on all stuffed $D^2\times I$ regions, we eventually get a desired surface $S'$.  Furthermore, the set of arcs associated to $S'$ is a collection of vertical arcs in $N(B)$.
\end{proof}

\begin{figure}
\begin{center}
\psfrag{(a)}{(a)}
\psfrag{(b)}{(b)}
\psfrag{comp}{compress and isotope}
\psfrag{E}{$E$}
\psfrag{E'}{$E'$}
\psfrag{V}{$V$}
\psfrag{iso}{isotopy}
\includegraphics[width=5.0in]{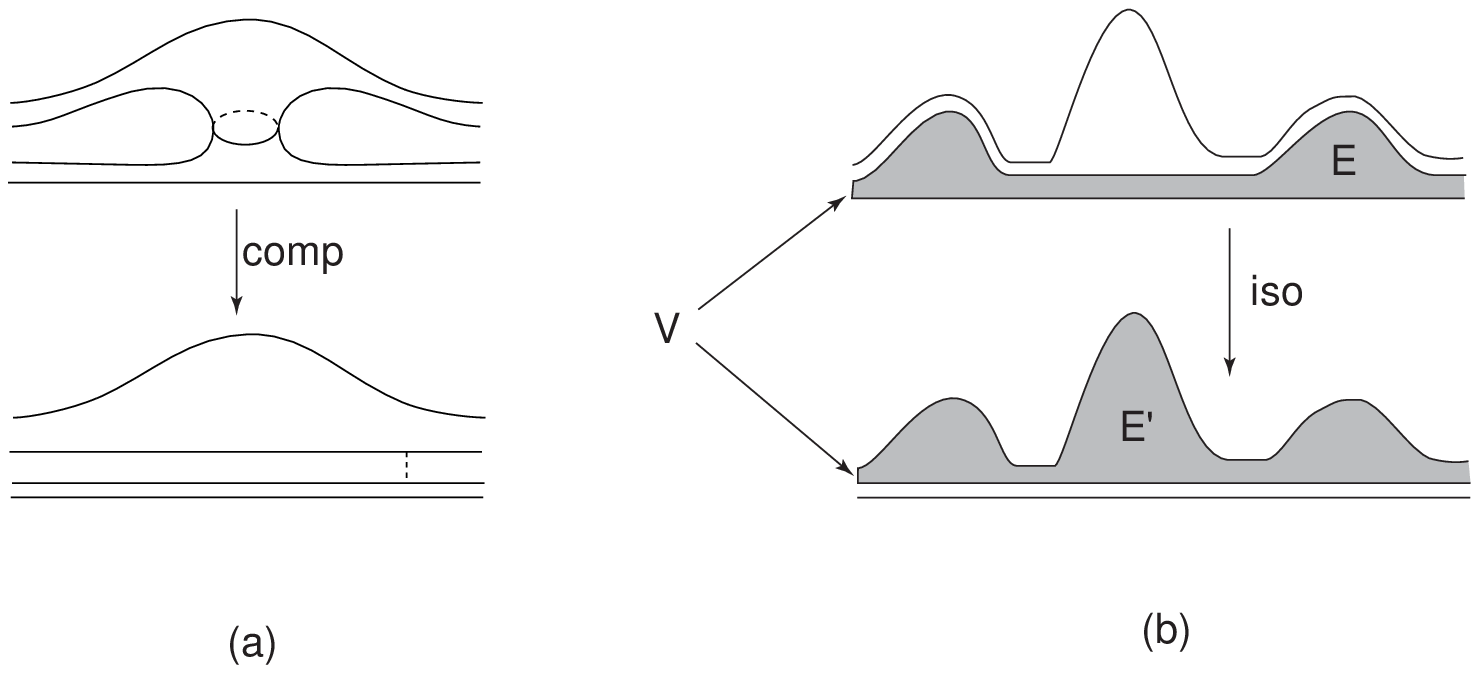}
\caption{}\label{iso2}
\end{center}
\end{figure}

\begin{lemma}\label{Last}
Let $S$ be a  normal or an almost normal surface carried by $B$ and suppose $S$ is either a strongly irreducible Heegaard surface or an almost strongly irreducible Heegaard surface.  Let $V$ be a component of $\partial_vN(B)$ and suppose $N(B)$ does not carry a disk $D$ which is $B$-parallel to a disk in $S$ and with $\partial D\subset\Int(V)$.  
Let $A$ be an embedded vertical annulus in $N(B)$ with $A\supset V$ and $\partial A\subset S$.  Suppose both components of $\partial A$ are trivial in $S$ and each curve in $\Int(A)\cap S$ (if not empty) is essential in $S$.  Then the two disks bounded by $\partial A$ in $S$ are non-nested in $S$ and hence (together with $A$) bound a $D^2\times I$ region for $S$ and $B$.
\end{lemma}
\begin{proof}
Let $D_0$ and $D_1$ be the two disks in $S$ bounded by $\partial A$.  Suppose the lemma is false and $D_0\subset D_1$.  As $A$ is embedded, $\partial D_0$ is disjoint from $\partial D_1$.  Let $\Sigma= D_1-\Int(D_0)$ be the annulus between $\partial D_0$ and $\partial D_1$.  Since $\Int(A)\cap S$ (if not empty) consists of essential curves in $S$ and every curve in $\Sigma$ is trivial in $S$, we have $\Sigma\cap A=\partial\Sigma$ and $\Sigma\cup A$ is an embedded surface.

As the component $V$ of $\partial_vN(B)$ lies in $A$, the branch direction at $V$ induces a normal direction for $A$ and a normal direction for $\partial D_i$ in $S$.  If the induced direction at $\partial D_i$ points into $D_i$, then we can perform a $B$-isotopy on $D_i$ pushing $\partial D_i$ into $\Int(V)$, which contradicts our hypothesis that no such disk exists.  So we may suppose the induced (branch) direction at $\partial D_i$ points out of $D_i$ for both $i=0,1$.  This implies that a small neighborhood of $\partial\Sigma$ in $\Sigma$ is a pair of annuli lying on different sides of $A$.  Thus after a small perturbation, $\Sigma\cup A$ becomes an embedded normal or almost normal torus or Klein bottle carried by $N(B)$.  Since $\Sigma$ connects one side of $A$ to the other side, $\Sigma\cup A$ is a Klein bottle if and only if the two annular components of a small neighborhood of $\partial A$ in $A$ lie on the same side of $S$ (recall that $S$ is separating).  However, by Corollary~\ref{Cklein}, $M$ contains no normal or almost normal Klein bottle, so $\Sigma\cup A$ cannot be a Klein bottle and hence the two annular components of a small neighborhood of $\partial A$ in $A$ lie on different sides of $S$, in other words, $A$ connects one side of $S$ to the other side of $S$.  
If $\Int(A)\cap S=\emptyset$, then this conclusion immediately implies that $S$ is non-separating in $M$, a contradiction.  So we may suppose $\Int(A)\cap S\ne\emptyset$.  
Let $\gamma_0=\partial D_0,\gamma_1,\dots,\gamma_{n+1}=\partial D_1$ be the curves of $A\cap S$ and we suppose $\gamma_i$ lies between $\gamma_{i-1}$ and $\gamma_{i+1}$ for each $i$.  Let $A_0$ and $A_1$ be the subannuli of $A$ bounded by $\gamma_0\cup\gamma_1$ and $\gamma_n\cup\gamma_{n+1}$ respectively.  The conclusion above says that $A_0$ and $A_1$ lie on different sides of $S$.  However, since $\gamma_1$ and $\gamma_n$ bound disks $A_0\cup D_0$ and $A_1\cup D_1$ respectively, by the hypothesis that $\Int(A)\cap S$ is essential in $S$, this means that $\gamma_1$ and $\gamma_n$ bound disjoint compressing disks on different sides of $S$.  This contradicts our hypothesis that $S$ is either a strongly irreducible or an almost strongly irreducible Heegaard surface.
\end{proof}

Our main goal in this section is to eliminate all non-trivial $D^2\times I$ regions.

\begin{lemma}\label{Lalmostcarry}
Let $S$ be a closed normal or almost normal surface fully carried by $B$. Suppose $S$ is either a strongly irreducible Heegaard surface or an almost strongly irreducible Heegaard surface. 
Then there is a normal or an almost normal surface $S'$ carried by $B$ and isotopic to $S$ in $M$ such that every good $D^2\times I$ region for $S'$ and $B$ is a trivial $D^2\times I$ region.
\end{lemma}
\begin{proof}
In the proof, we only consider good $D^2\times I$ regions, in other words, for any $D^2\times I$ region $E$ is this proof, we always assume $E\cap S$ consists of disks, see Definition~\ref{Dball}.

Let $E$ be any simple $D^2\times I$ region with $\partial_hE=D_1\cup D_2$.  Recall from Definition~\ref{Dball},  simple $D^2\times I$ region means that $E\cap S=\partial_hE=D_1\cup D_2$.   We can perform an isotopy on $S$ by pushing $D_1$ across $E$ into a disk $B$-parallel to $D_2$, as shown in Figure~\ref{iso}(a) (ignore the dashed arcs in the picture for now).  Our main task is to show that this process ends in a finite number of steps.   First note that we may assume that after an isotopy in Figure~\ref{iso}(a), the resulting surface remains normal or almost normal.  To see this, if a disk, say $D_1$, contains the almost normal piece (as we explained at the beginning of the proof of Lemma~\ref{LtypeI}, if $D_i$ intersects the almost normal piece then it must contain the whole almost normal piece), then we push $D_1$ across $E$ into a disk $B$-parallel to $D_2$ and the surface after the isotopy is normal.  Moreover, the surface after this isotopy does not pass through the branch sector that contains the almost normal piece.   This means that a $D^2\times I$ region in any future isotopy does not involve the almost normal piece and we always get a normal surface.  For this reason, we may assume that for any simple $D^2\times I$ region in the proof, its horizontal boundary $D_1\cup D_2$ does not contain the almost normal piece.

We say a non-trivial $D^2\times I$ region is innermost if it does not contain any other non-trivial $D^2\times I$ region.   Clearly an innermost good $D^2\times I$ region must be a simple $D^2\times I$ region, see Definition~\ref{Dball}.  Let $V$ be a component of $\partial_vN(B)$.  We say $V$ is a \emph{belt} of a good $D^2\times I$ region $E$ if $V\subset A=\partial_vE$ and $E$ contains the component $Z$ of $M-\Int(N(B))$ with $\partial Z\supset V$.

\vspace{10pt}
\noindent
\emph{Claim 1}.  For any innermost good non-trivial $D^2\times I$ region $E$, a component of $\partial_vN(B)$ is the belt of $E$.
\begin{proof}[Proof of Claim 1]
A key ingredient in the proof is that $N(B)$ fully carries $S$. 
Let $E$ be an innermost good non-trivial $D^2\times I$ region with $\partial_hE=D_1\cup D_2$ and $\partial_vE=A$.  So $E$ must be a simple $D^2\times I$ region.   Since $A$ is vertical in $N(B)$, we can give $A$ a product structure $S^1\times I$ such that $\{x\}\times I$ is a subarc of an $I$-fiber of $N(B)$ for each $x\in S^1$.    If for some $x\in S^1$, $\{x\}\times I$ does not contain a vertical arc of $\partial_vN(B)$, then we can shrink $E$ a little to get a slightly smaller $D^2\times I$ region inside $E$, which contradicts the hypothesis that $E$ is innermost.  So we may assume each $\{x\}\times I$ contains a vertical arc of $\partial_vN(B)$.  If $\alpha\cap\partial_vN(B)$ has more than one component for some vertical arc $\alpha=\{x\}\times I$ of $A$ (this happens only if $\alpha$ corresponds to a double point in the branch locus), then the subarc of $\alpha$ between the two components of $\alpha\cap\partial_vN(B)$ can be horizontally pushed slightly into an $I$-fiber of $N(B)$, see $J_0$ and $J_1$ in Figure~\ref{iso}(c) for a schematic picture of how this subarc is pushed. 
As $S$ is fully carried by $N(B)$, $S$ intersects every $I$-fiber of $N(B)$ and this implies that 
$S$ must non-trivially intersect the subarc of $\alpha$ between the two components of $\alpha\cap\partial_vN(B)$.  So $S\cap\Int(\alpha)\ne\emptyset$ and this contradicts our conclusion above that $E$ is a simple $D^2\times I$ region.  Thus each $\{x\}\times I$ in $A=S^1\times I$ contains exactly one vertical arc of $\partial_vN(B)$ and this implies that $A$ contains a component $V$ of $\partial_vN(B)$.  Let $Z$ be the component of $M-\Int(N(B))$ that contains $V$.  If $Z$ lies outside $E$, then we can shrink $E$ a little to get a smaller $D^2\times I$ region, which contradicts that $E$ is innermost.  So $Z\subset E$ and $V$ is the belt of $E$.
\end{proof}

\begin{figure}
\begin{center}
\psfrag{(a)}{(a)}
\psfrag{(b)}{(b)}
\psfrag{(c)}{(c)}
\psfrag{iso}{isotopy}
\psfrag{E}{$E$}
\psfrag{D0}{$D_0$}
\psfrag{D1}{$D_1$}
\psfrag{N}{$N(B')$}
\psfrag{J}{$J$}
\psfrag{J0}{$J_0$}
\psfrag{J1}{$J_1$}
\includegraphics[width=5.0in]{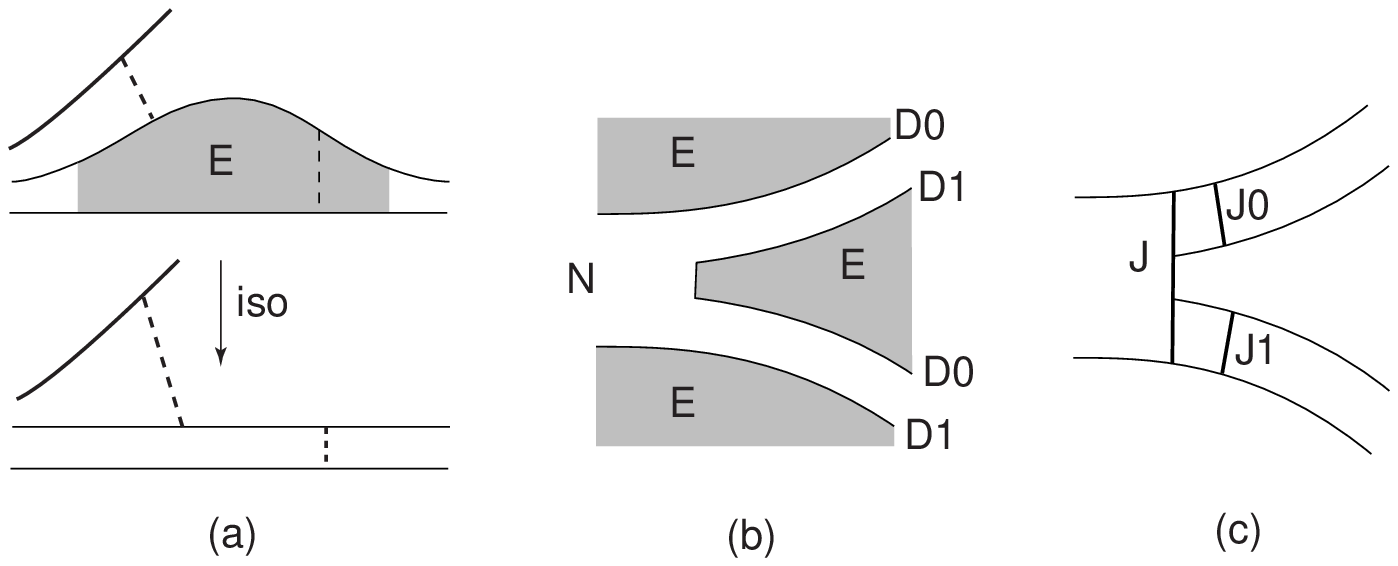}
\caption{}\label{iso}
\end{center}
\end{figure}

Let $P\subset N(B)$ be a compact surface carried by $N(B)$.  We say $P$ is a \emph{splitting surface} for $S$ if $\partial P\subset\Int(\partial_vN(B))$ and $P\cap S=\emptyset$.  Note that we can split $N(B)$ along $P$ (i.e. delete a small neighborhood of $P$ from $N(B)$) to get a fibered neighborhood of a branched surface carrying $S$.  If $P$ is a disk, then $P$ is called a disk of contact, see \cite{FO}.

\vspace{10pt}
\noindent
\emph{Claim 2}.  Let $E$ be a simple $D^2\times I$ region and suppose a component $V$ of $\partial_vN(B)$ is the belt of $E$.  Then there is a splitting surface $P$ (which may be an empty set) lying in $\Int(E)$
 such that each component of $P$ is a planar surface and  after splitting $N(B)$ along $P$, we get a fibered neighborhood $N(B')$ of a branched surface $B'$ that fully carries $S$ and $E-\Int(N(B'))$ is a $D^2\times I$ component of $M-\Int(N(B'))$ (see Definition~\ref{Dcomp}).  Furthermore, if $E$ is an innermost $D^2\times I$ region, then each component of $P$ is a disk.
\begin{proof}[Proof of Claim 2] 
Suppose $\partial_h E=D_0\cup D_1$ and $\partial_vE=A$.  Since $V$ is the belt of $E$, $V\subset A$. We may assume $S\subset\Int(N(B))$.
 
For any $x$ in $\Int(D_0)\cup\Int(D_1)$, let $I_x$ be the $I$-fiber of $N(B)$ that contains $x$ and let $K_x$ be the component of $I_x\cap E$ that contains $x$.  So $K_x$ is either an arc properly embedded in $E$ or an arc with one endpoint $x$ and the other endpoint in $\partial_hN(B)$.  Suppose $K_x$ is not properly embedded in $E$ for some $x\in\Int(D_0)\cup\Int(D_1)$, i.e., one endpoint of $K_x$ is $x$ and the other endpoint, denoted by $x'$, lies in $\partial_hN(B)\cap\Int(E)$.  We show next that such $K_x$ does not contain a vertical arc of $\partial_vN(B)$.  Suppose on the contrary that $K_x$ contains a vertical arc of $\partial_vN(B)$ and let $J_0$ be the component of $K_x-\Int(\partial_vN(B))$ that contains $x'$.  Clearly $J_0\subset\Int(E)$, and since $E$ is a simple $D^2\times I$ region, this means that $J_0\cap S=\emptyset$.  However, as illustrated in Figure~\ref{iso}(c), $J_0$ can be horizontally pushed slightly into an $I$-fiber of $N(B)$.  Since $S$ is fully carried by $N(B)$, $S$ intersects every $I$-fiber of $N(B)$ and this means that $J_0\cap S\ne\emptyset$, a contradiction.  Hence, for any $x\in\Int(D_0)\cup\Int(D_1)$, if $K_x$ is not properly embedded in $E$, then $K_x$ does not contain a vertical arc of $\partial_vN(B)$. 

If $K_x$ is not properly embedded in $E$ for every $x\in\Int(D_0)\cup\Int(D_1)$, then the conclusion above on $K_x$ implies that $D_0$ and $D_1$ are $B$-isotopic to disk components $D_0'$ and $D_1'$ of $\partial_hN(B)$ respectively and $D_0'\cup D_1'\subset E$.  Moreover, $\partial D_0'\cup\partial D_1'$ bounds the component $V$ of $\partial_vN(B)$ ($V$ is the belt of $E$).  So $D_0'\cup V\cup D_1'$ bounds a $D^2\times I$ component of $M-\Int(N(B))$ and the claim holds with $P=\emptyset$.  

Next we assume $K_x$ is properly embedded in $E$ for some $x\in\Int(D_0)\cup\Int(D_1)$.  If $K_x\cap\partial_vN(B)$ contains two vertical arcs of $\partial_vN(B)$, then the subarc $J_1$ of $K_x$ between the two components of $K_x\cap\partial_vN(B)$ can be horizontally pushed slightly into an $I$-fiber of $N(B)$, see Figure~\ref{iso}(c) for a schematic picture of how $J_1$ is pushed.  Since $S$ is fully carried by $N(B)$, as above, we have $S\cap J_1\ne\emptyset$, which implies that $S\cap \Int(K_x)\ne\emptyset$ as $J_1\subset \Int(K_x)$.  However, since $E$ is a simple $D^2\times I$ region, $S\cap\Int(K_x)=\emptyset$, a contradiction.  Thus, if $K_x\cap\partial_vN(B)\ne\emptyset$ for some $x\in\Int(D_0)\cup\Int(D_1)$, then $K_x\cap\partial_vN(B)$ is a single vertical arc of $\partial_vN(B)$.

Let $U$ be the union of all such $K_x$'s ($x\in\Int(D_0)\cup\Int(D_1)$) that are properly embedded in $E$.  Since a component of $\partial_vN(B)$ is the belt of $E$, it follows from the definition of belt that $\partial_vE\cap\overline{U}=\emptyset$.  Hence $U$ must be an $I$-bundle over a compact surface $P$ and the horizontal boundary of $U$, denoted by $\partial_hU$, is a compact subsurface of $\Int(D_0)\cup\Int(D_1)$.  This implies that if a component of $U$ is a twisted $I$-bundle over a non-orientable surface, then one can cap off the boundary of the non-orientable surface using disks and obtain a closed non-orientable surface embedded in the 3-ball $E$, which is impossible.  Thus $U$ is a product $P\times I$, where each component of $P$ is a planar surface.  Now we consider the vertical boundary of $U$, denoted by $\partial_vU$.  So $\partial_vU$ is a collection of vertical annuli properly embedded in $E$.  By our construction of $U$, if $K_x\subset\partial_vU$, then $K_x\cap\partial_vN(B)\ne\emptyset$.  The discussion above says that if $K_x\cap\partial_vN(B)\ne\emptyset$, then $K_x\cap\partial_vN(B)$ is a single vertical arc of $\partial_vN(B)$.  This implies that each component of $\partial_vU$ contains a component of $\partial_vN(B)$.  Thus we may view $P$ as a splitting surface and clearly after we cut $E\cap N(B)$ along $P$, we get a $D^2\times I$ component of $M-\Int(N(B'))$, where $B'$ is the branched surface obtained by splitting $N(B)$ along $P$.  As $P\subset\Int(E)$, by our construction, $S$ is still fully carried by $N(B')$.

If a component $Q$ of $\partial_hU$ is not a disk, then $(D_0\cup D_1)-Q$ has a disk component which determines a smaller $D^2\times I$ region inside $E$.  This means that if $E$ is innermost, then $\partial_hU$ and hence $P$ must be a union of disks.  
\end{proof}

It follows from Claim 2 that if $E$ is an innermost good non-trivial $D^2\times I$ region, then there is exactly one component of $M-\Int(N(B))$ lying in $E$ and this component becomes a $D^2\times I$ component after splitting $N(B)$ along $P$, where $P$ is a collection of disks as in Claim 2.  This means that this component of $E-\Int(N(B))$ is an almost $D^2\times I$ component of $M-\Int(N(B))$, see Definition~\ref{Dcomp}. 

Given any $D^2\times I$ region $X$, we define the complexity $c(X)$ of $X$ to be the number of components of $M-\Int(N(B))$ that lie in $X$. So $c(X)=|X-N(B)|$ and clearly $c(X)\le |M-B|$.  Suppose a component $V$ of $\partial_vN(B)$ is the belt of some good $D^2\times I$ region for $S$, then it follows from the definition of good $D^2\times I$ region that there is a unique simple $D^2\times I$ region $X_V$ for $S$ of which $V$ is the belt.  We define the complexity $C_S(V)$ of $V$ for $S$ to be $C_S(V)=c(X_V)=|X_V-N(B)|$.  

Let $E$ be a non-trivial simple $D^2\times I$ region with $\partial E=D_0\cup A\cup D_1$.  As shown in Figure~\ref{iso}(a) (ignore the dashed arc in the picture for now), we can perform an isotopy on $S$ by pushing $D_0$ (resp. $D_1$) across $E$ to a disk $B$-isotopic to $D_1$ (resp. $D_0$)

\vspace{10pt}
\noindent
\emph{Claim 3}.  Let $E$ be a non-trivial simple $D^2\times I$ region and suppose a component $V$ of $\partial_vN(B)$ is the belt of $E$.  Then one can perform some isotopies on $S$ as shown in Figure~\ref{iso}(a) and get a surface $S_1$ such that either 
\begin{enumerate}
\item $N(B)$ carries but not fully carries $S_1$ or 
\item $V$ is not the belt of any simple $D^2\times I$ region for $S_1$, or 
\item $V$ is the belt of a simple $D^2\times I$ region for $S_1$ but the complexity $C_{S_1}(V)>C_S(V)$.
\end{enumerate}
\begin{proof}[Proof of Claim 3] 
By the proof of Claim 2, we can split $N(B)$ along a planar splitting surface in $E$ and get a fibered neighborhood $N(B')$ of a new branched surface $B'$ such that $E-\Int(N(B'))$ is a $D^2\times I$ component of $M-\Int(N(B'))$.  Note that $N(B')$ still fully carries $S$.  So after some $B'$-isotopy on $S$ pushing $\partial_hE$ to $\partial_hN(B')$, we may assume $\partial_hE=D_0\cup D_1\subset S$ is a pair of disk components of $\partial_hN(B')$, $V=\partial_vE\subset\partial_vN(B')$, and the $D^2\times I$ region $E$ is a $D^2\times I$ component of $M-\Int(N(B'))$.

For any $x\in D_i$, let $I_x$ be the $I$-fiber of $N(B')$ that contains $x$.  Since $D_0\cup D_1\subset\partial_hN(B')$, the interior of $I_x$ does not intersect $D_0\cup D_1$ unless $I_x$ contains a vertical arc of $V$ in which case $\pi(I_x)$ is a point in the branch locus of $B'$, where $\pi:N(B')\to B'$ is the projection that collapses each $I$-fiber to a point.   

Next we show that there is a point $x$ in $D_0$ (or $D_1$) such that $I_x$ does not intersect $D_1$ (or $D_0$ respectively). 

Suppose on the contrary that for every point $x$ in $D_0\cup D_1$, $I_x$ intersects both $D_0$ and $D_1$.  It follows from $D_0\cup D_1\subset\partial_hN(B')$ that if $\pi(I_x)$ is not a point in the branch locus of $B'$, then $\Int(I_x)$ is disjoint from $D_0\cup D_1$.  Since we have assumed that $I_x$ intersects both $D_0$ and $D_1$, if $\pi(I_x)$ is not a point in the branch locus of $B'$, then one endpoint of $I_x$ lies in $D_0$ and the other endpoint of $I_x$ lies in $D_1$.   If $\pi(I_x)$ is a point in the branch locus,  by taking the limit of points $y\in D_0\cup D_1$ near $x$ with $\pi(I_y)$ not in the branch locus, as shown in Figure~\ref{iso}(b), we can also conclude that the two endpoints of $I_x$ lie in different components of $D_0\cup D_1$.  This means that the union of all the $I_x$ ($x\in D_0\cup D_1$) is a fibered neighborhood of a branched surface, in fact, it must be the whole of $N(B')$ and $M-\Int(N(B'))=E$, see Figure~\ref{iso}(b).  However, since $N(B')$ fully carries $S$, the product structure of $E=D^2\times I$ and the $I$-fiber structure of $N(B')-S$ gives an $I$-bundle structure for $\overline{M-S}$ and $\overline{M-S}$ must be an $I$-bundle over a closed surface.  This contradicts our hypothesis on $S$. 

So we may suppose there is a point $x$ in $D_0$  such that $I_x$ does not intersect $D_1$.
Now we push $D_1$ back into $\Int(N(B'))$ and then we perform an isotopy on $S$ by pushing $D_0$ across $E$ into a disk $B'$-parallel to $D_1$ as shown in Figure~\ref{iso}(a).  Since $I_x\cap D_1=\emptyset$, $|I_x\cap S|$ is reduced after the isotopy.  We use $S'$ to denote the surface after this isotopy.  If $S'$ is still fully carried by $B'$ after the isotopy, then since $E$ is a $D^2\times I$ component of $M-\Int(N(B'))$ and $S'$ is separating, we can isotope $S'$ so that $\partial_hE\subset S'$ and the $D^2\times I$ component $E$ becomes a $D^2\times I$ region for $S'$ and $B'$.  So we can perform the same $B'$-isotopy on $S'$ which reduces $|I_x\cap S'|$.  Thus after a finite number of such isotopies, we obtain a surface $S_1$ isotopic to $S$ with $|I_x\cap S_1|=0$.  This means that $S_1$ is carried but not fully carried by $N(B')$. 

By our construction, $N(B')$ is obtained from $N(B)$ by deleting a small neighborhood of a splitting surface.  Thus we may view $N(B')\subset N(B)$ and  view $S_1$ as a surface carried by $B$.  Although $S_1$ is not fully carried by $N(B')$, $S_1$ may still be fully carried by $N(B)$.  Next we suppose part (1) of the claim is not true and $S_1$ is fully carried by $N(B)$.

Suppose part (2) of the lemma is also not true, i.e.~$V$ is the belt of some simple $D^2\times I$ region $E'$ for $S_1$.  We may view $S_1\subset N(B')\subset N(B)$.  By our construction, the splitting planar surface $P$ in Claim 2 (for $E$) remains disjoint from $S_1$ after the isotopies above and $P$ remains $B$-parallel to a subsurface of $S_1$.  Moreover, if $P\ne\emptyset$, $\partial P$ is incident to every component of $E-\Int(N(B))$.  Since $P$ is $B$-parallel to a subsurface of $S_1$ and $V$ is the belt for both $E$ and $E'$, this implies that every component of $E-\Int(N(B))$ must be a component of $E'-\Int(N(B))$, see Figure~\ref{iso2}(b) for a schematic picture.  So $C_{S_1}(V)\ge C_S(V)$.  Furthermore, if $C_{S_1}(V)=C_S(V)$, then $E-\Int(N(B))=E'-\Int(N(B))$ and after splitting $N(B)$ along $P$, we get a $D^2\times I$ component in $E'$.  This means that we have not finished the isotopies above (i.e., $|I_x\cap S_1|\ne 0$) and $S_1$ is still fully carried by $N(B')$, a contradiction.  Thus by our construction of $S_1$, $C_{S_1}(V)>C_S(V)$ and part (3) of the claim holds.
\end{proof}

It follows from the definition that the complexity $C_S(V)\le |M-B|$.  Hence if we perform the isotopy in Claim 3 on simple $D^2\times I$ regions of which $V$ is the belt, then part (3) of Claim 3 cannot occur infinitely many times, in other words, after a finite steps of isotopies as in Claim 3, either the resulting surface is no longer fully carried by $N(B)$ or $V$ is no longer the belt of a simple $D^2\times I$ region for the resulting surface. 
In the next claim, we prove that if a component of $\partial_vN(B)$ is not the belt of a $D^2\times I$ region, then it is not the belt of a $D^2\times I$ region for the surface after the isotopy in Claim 3.

\vspace{10pt}
\noindent
\emph{Claim 4}. Let $S_1$ be the surface after the isotopy in Claim 3 and suppose $S_1$ is still fully carried by $B$. Let $U$ be a component of $\partial_vN(B)$ and suppose $U\ne V$, where $V$ is the belt of the $D^2\times I$ region $E$ in Claim 3.  If $U$ is not the belt of a simple $D^2\times I$ region for $S$ and $B$, then $U$ is not the belt of a simple $D^2\times I$ region for $S_1$ and $B$.  
\begin{proof}[Proof of Claim 4]
Suppose the claim is false and $U$ is the belt of a simple $D^2\times I$ region $E_1$ for $S_1$ and $B$.  Suppose $\partial_hE_1=\Theta_0\cup\Theta_1\subset S_1$ and $\partial_vE_1=A_1$ is a vertical annulus containing $U$.  First note that a core curve of $\Int(U)$ does not bound a disk $\Delta$ that is carried by $N(B)$ and $B$-parallel to a subdisk of $S$. To see this, if such a disk $\Delta$ exists, then the union of $\Theta_i$ and a parallel copy of $\Delta$ is a (possibly immersed) 2-sphere carried by $N(B)$.  Since we have assume at the beginning of the lemma that our $D^2\times I$ region does not contain the almost normal piece, the 2-sphere above is normal or almost normal, which contradicts Lemma~\ref{LimmersedS2}.

Since $U$ is a component of $\partial_vN(B)$ and $S$ is fully carried by $B$, $U$ can be vertically extended to a vertical annulus $A_U\subset N(B)$ that contains $U$ and is properly embedded in a component of $\overline{M-S}$.  Let $\gamma_0$ and $\gamma_1$ be the two boundary curves of $A_U$.  If both $\gamma_0$ and $\gamma_1$ are trivial in $S$, since the core curve of $\Int(U)$ does not bound a disk $\Delta$ that is $B$-parallel to a subdisk of $S$ as above, by Lemma~\ref{Last}, the union of $A_U$ and the two disks bounded by $\gamma_0$ and $\gamma_1$ in $S$ is an embedded 2-sphere bounding a $D^2\times I$ region $E_U$ for $S$.  Let $Z_U$ be the component of $M-\Int(N(B))$ containing the component $U$ of $\partial_vN(B)$.  If $Z_U$ lies outside $E_U$, then the core curve of $\Int(U)$ bounds a disk $B$-parallel to the subdisk of $S$ bounded by $\gamma_i$, a contradiction to our conclusion on $U$ above.  Thus $Z_U$ lies in $E_U$ and hence $U$ is the belt of the simple $D^2\times I$ region $E_U$, contradicting our hypothesis that $U$ is not the belt of a simple $D^2\times I$ region for $S$.  Thus at least one component of $\partial A_U$, say $\gamma_0$, is non-trivial in $S$. Our goal is to show that the isotopy in Claim 3 does not affect $\gamma_0$.

Let $E$ be the simple $D^2\times I$ region for $S$ in the isotopy in the proof of Claim 3.  We use the same notation as Claim 3, in particular $\partial_hE=D_0\cup D_1$ and $\partial_vE\supset V$.   We may suppose the isotopy that changes the surface $S$ to $S_1$ is the operation pushing $D_0$ across the simple $D^2\times I$ region $E$, and suppose $U$ is not a belt before the isotopy but becomes the belt of the simple $D^2\times I$ region $E_1$ for $S_1$ after the isotopy.  
We first show that $U$ lies outside $E$.  Suppose $U\subset E$.  Since $E$ is a simple $D^2\times I$ region, by Claim 2, $U$ must lie in the boundary of a product $P\times I$, where $P$ is a planar splitting surface in Claim 2.  This means that $\partial A_U=\gamma_0\cup\gamma_1$ lies in $\partial_hE=D_0\cup D_1$, which contradicts our conclusion above that $\gamma_0$ is essential in $S$. So $U$ lies outside $E$.

If $\gamma_0\cap D_0=\emptyset$, then the isotopy (which pushes $D_0$ across $E$ to a disk parallel to $D_1$) does not affect $\gamma_0$ and $\gamma_0$ remains an essential curve in $S_1$, which implies that $U$ cannot be the belt of any $D^2\times I$ region for $S_1$, a contradiction.  So we may assume $\gamma_0\cap D_0\ne\emptyset$.  Since $\gamma_0$ is non-trivial in $S$, $\gamma_0\not\subset D_0$ and this means that $\gamma_0\cap\partial D_0\ne\emptyset$.  

Next we show that there is an $I$-fiber of $N(B)$ that intersects $D_0$ in exactly one point.  The main reason for this is that the $D^2\times I$ region $E$ is simple.  Let $J$ be any $I$-fiber in $\pi^{-1}(L)$ where $L$ is the branch locus of $B$ and $\pi:N(B)\to B$ is the collapsing map.  Note that each component of $J-\Int(\partial_vN(B))$ can be horizontally pushed slightly into an $I$-fiber of $N(B)$, see Figure~\ref{iso}(c) for a 1-dimension lower schematic picture (in this picture, the two components of $J-\Int(\partial_vN(B))$ are pushed into $J_0$ and $J_1$). 
Let $x\in\gamma_0\cap\partial D_0$ and let $J_x$ be the $I$-fiber of $N(B)$ containing $x$.  
Using the notation in Claim 3, $V$ is the belt of the simple $D^2\times I$ region $E$.  Hence $\pi(\partial D_i)=\pi(V)$ and $\pi(\gamma_0)=\pi(U)$ are  curves in the branch locus $L$.  As $x\in\gamma_0\cap\partial D_0$, $\pi(J_x)$ is a double point in the branch locus and $J_x$ contains a vertical arc $\alpha_v$ of $V$ and a vertical arc $\alpha_u$ of $U$.  Let $J'$ be the component of $J_x-\Int(\partial_vN(B))$ that lies between $\alpha_u$ and $\alpha_v$.  Next we show $J'\cap S=x$.  As $U$ lies outside $E$, we have $x\in J'$.  We may suppose $S\subset\Int(N(B))$ and hence $x\in\Int(J')$.  So $x$ cuts $J'$ into two arcs $J'_u$ and $J'_v$ with $\partial J_u'-x$ and $\partial J_v'-x$ being endpoints of $\alpha_u$ and $\alpha_v$ respectively.  Note that $J_u'$ is a subarc of a vertical arc of the annulus $A_U$  and $J_v'$ is a subarc of a vertical arc of the annulus $\partial_vE$. 
Since $A_U$ is properly embedded in $\overline{M-S}$, $J_u'\cap S=x$.  Since the $D^2\times I$ region $E$ is simple, $J_v'\cap S=x$.  Hence $J'\cap S=x$ is a single point.  As illustrated in Figure~\ref{iso}(c), we can horizontally push $J'$ to be an $I$-fiber $J_0$ of $N(B)$ such that $J_0\cap S=J_0\cap D_0$ is a single point. 

Therefore, after the isotopy pushing $D_0$ across $E$ to a disk parallel to $D_1$, $J_0$  does not intersect the resulting surface $S_1$, which means that $S_1$ is not fully carried by $N(B)$. This contradicts our hypothesis on $S_1$.
\end{proof}

Lemma~\ref{Lalmostcarry} follows from the 4 claims above.  Suppose there is a non-trivial good $D^2\times I$ region for $S$ and $B$.  By Claim 1, a component $V$ of $\partial_vN(B)$ must be the belt of a simple $D^2\times I$ region.   
Then we perform the isotopy in the proof of Claim 3.  Since the complexity $C_S(V)$ is bounded by $|M-B|$, after a finite steps of isotopies across simple $D^2\times I$ regions of which $V$ is the belt, either the resulting surface is no longer fully carried by $B$ or $V$ is no longer the belt of a good $D^2\times I$ region for the resulting surface.  By Claim 4, if the resulting surface is still fully carried by $N(B)$, then $V$ will not be the belt of a good $D^2\times I$ region after any future isotopy.  As $|\partial_vN(B)|$ is bounded, after finitely many isotopies as in Claim 3, either the resulting surface is no longer fully carried by $N(B)$ or no component of $\partial_vN(B)$ is the belt of a good $D^2\times I$ region.  In the latter case, by Claim 1, there is no non-trivial good $D^2\times I$ region for the resulting surface and Lemma~\ref{Lalmostcarry} holds.  If the surface after isotopy is no longer fully carried by $B$, then we take a sub-branched surface of $B$ that fully carries the surface and apply the argument above on this sub-branched surface.  Since the number of sub-branched surfaces of $B$ is bounded, part (1) of Claim 3 can only happen a bounded number of times.  Therefore the isotopies end in a finite number of steps and Lemma~\ref{Lalmostcarry} holds.
\end{proof}

If a $D^2\times I$ region $E$ for $S$ and $B$ is not good, then by Lemma~\ref{Lstuff}, we may assume every non-disk component of $E\cap S$ is an unknotted annulus.  Next we will compress those unknotted annuli in a stuffed $D^2\times I$ region as in Corollary~\ref{Cstuff}.  However, since we are interested in an algorithm to list all the Heegaard surfaces of bounded genus, we need to keep track of the compressions and be able to algorithmically recover the original Heegaard surface in the end.

\begin{lemma}\label{Lal}
Suppose $S$ is a strongly irreducible Heegaard surface fully carried by $B$, where $B$ is as above and $S$ is a normal or an almost normal surface.  Then there is an almost Heegaard surface $S'$ with respect to $B$ (see Definition~\ref{Dai}) and a set of almost vertical arcs $J$ associated to $S'$ such that 
\begin{enumerate}
\item $S'$ is normal or almost normal,
\item $S$ can be derived from $S'$ by adding tubes to $S'$ along arcs in $J$,
\item the length of each arc in $J$ is bounded from above by a number that depends only on $M$ and $B$,
\item there is no non-trivial $D^2\times I$ region for $S'$ and $B$.  
\end{enumerate}
Furthermore, given $S'$, there is an algorithm to construct a finite set of Heegaard surfaces that contains a Heegaard surface isotopic to $S$.
\end{lemma}
\begin{proof}
If there is a stuffed $D^2\times I$ region for $S$ that is not good, then we can compress $S$ as in Corollary~\ref{Cstuff} to get an almost Heegaard surface $S_1$ with respect to $B$ such that every $D^2\times I$ region for $S_1$ is a good $D^2\times I$ region.  Moreover, the set of  arcs $\Gamma_1$ associated to the almost Heegaard surface $S_1$ are vertical arcs in $N(B)$.  

Now we perform isotopies as in Lemma~\ref{Lalmostcarry} to eliminate all the good non-trivial $D^2\times I$ regions.  We also extend the isotopies to the associated arcs $\Gamma_1$ so that the surface after the isotopies remains an almost strongly irreducible Heegaard surface.   
Let $S_1$ and $\Gamma_1$ be as above.  Since $B$ does not carry any normal or almost normal $S^2$, $S_1$ has no 2-sphere component and the total number of arcs in $\Gamma_1$ is less than the genus of $S$.  Before we proceed, we first explain why the isotopies may change an arc in $\Gamma_1$ into an almost vertical arc with respect to $B$, see Definition~\ref{Dvert}.

In the proof of Lemma~\ref{Lalmostcarry}, we perform isotopies on a simple non-trivial $D^2\times I$ region $E$.  Although a simple $D^2\times I$ region may not be innermost, we can divide the isotopy on $E$ into several steps with each step being such an isotopy on an innermost $D^2\times I$ region inside $E$.  Thus we can assume the $D^2\times I$ region $E$ in the isotopy in Claim 3 of Lemma~\ref{Lalmostcarry} is innermost. 
As in Claim 2 in the proof of Lemma~\ref{Lalmostcarry}, since $E$ is innermost, $E-\Int(N(B))$ is an almost $D^2\times I$ component of $M-\Int(N(B))$.  After the isotopies in Claim 3 of Lemma~\ref{Lalmostcarry}, an arc in $\Gamma_1$ may be stretched through the almost $D^2\times I$ component $E-\Int(N(B))$ and may no longer be a vertical arc in $N(B)$ after the isotopy.  So after the isotopies, an arc in $\Gamma_1$ may become the union of a vertical arc in the almost $D^2\times I$ component $E-\Int(N(B))$ (see Definition~\ref{Dcomp}) and possibly a pair of vertical arcs in $N(B)$ at the two ends, see the dashed arcs in Figure~\ref{iso}(a) for a schematic picture of such change on arcs associated to an almost Heegaard surface.  The resulting arc is by definition an almost vertical arc with respect to $B$, see Definition~\ref{Dai}.

To prove the lemma, we use the following inductive argument. We perform the isotopies in the proof of Lemma~\ref{Lalmostcarry}.  Suppose we are at a certain stage of the isotopies and let $S_2$ be the current surface which is isotopic to $S_1$ above.  Suppose 
$S_2$ is an almost strongly irreducible Heegaard surface and let $\Gamma_2$ be the union of the associated almost vertical arcs.  Let $B_2$ be the sub-branched surface of $B$ fully carrying $S_2$.  Suppose we are to perform the isotopies in Claim 3 of Lemma~\ref{Lalmostcarry} on an innermost $D^2\times I$ region $E_2$ for $S_2$ and $B_2$.

First note that arcs in $\Gamma_2$ may intersect $E_2$.  As $\partial_vE_2$ is a vertical annulus in $N(B_2)$, we may assume $\Gamma_2\cap \partial_vE_2=\emptyset$.  Recall that each arc in $\Gamma_2$ is properly embedded in $\overline{M-S_2}$.  This implies that $\Gamma_2\cap E_2$ is a collection of properly embedded arcs in the innermost $D^2\times I$ region $E_2$.  Notice that the annulus $\partial_vE_2$ is properly embedded in a component of $H$ of $\overline{M-S_2}$ and is incompressible in $H-\Int(E_2)$, because otherwise, the union of a compressing disk for $\partial_vE_2$ outside $E_2$ and a disk in $E_2$ parallel to a component of $\partial_hE_2$ is an embedded 2-sphere separating the two disks of $\partial_hE_2$, which means that a component of $S_2$ lies in a 3-ball contradicting Lemma~\ref{L3-ball}. 

Next we claim that $\Gamma_2\cap E_2$ is a single unknotted arc connecting the two disk components of $\partial_hE_2$.   To see this, we need the assumption that after we add tubes to $S_2$ along arcs in $\Gamma_2$, the resulting Heegaard surface is strongly irreducible.  As the $D^2\times I$ region $E_2$ is innermost, we can first shrink $E_2$ a little to get a slightly smaller 3-ball $E_2'\subset E_2$ disjoint from $S_2$.   We may view the strongly irreducible Heegaard surface $S$ as the surface obtained by adding tubes to $S_2$ along $\Gamma_2$.  Hence we may view the intersection of $S$ and the 3-ball $E_2'$ as a collection of annuli along arcs in $E_2'\cap\Gamma_2$.  
A theorem of Scharlemann \cite[Theorem 2.1]{S} says that if the intersection of 
the boundary 2-sphere of a 3-ball with the two handlebodies of a strongly irreducible Heegaard splitting is incompressible in the corresponding handlebodies, then the intersection of the 3-ball with the Heegaard surface is a planar unknotted surface.  Note that by our construction, the meridional curve of each added tube is an essential curve in $S$ (otherwise $S_2$ contains a 2-sphere component carried by $B$, which contradicts our hypotheses on $B$).  Let $H_1$ and $H_2$ be the two handlebodies in the Heegaard splitting along $S$ and suppose the meridional disks of the added tubes are compressing disks in $H_1$.  So $\partial E_2'\cap H_1$ is a collection of compressing disks.  Now we consider $\partial E_2'\cap H_2$.  By our conclusion above that the annulus $\partial_vE_2$ is incompressible outside $E_2$, any compressing disk for $\partial E_2'\cap H_2$ in $H_2$ can be isotoped disjoint from $\partial_vE_2$.  This implies that a maximal compression on $\partial E_2'\cap H_2$ in $H_2$ cuts $E_2'$ into a collection of smaller 3-balls, and by Scharlemann's theorem above, the intersection of each 3-ball with $S$ is a single unknotted annulus.  This implies that $\Gamma_2\cap E_2$ is a collection of unknotted and unlinked arcs in the 3-ball $E_2$.  If the two endpoints of an arc in $\Gamma_2\cap E_2$ lie in the same disk component of $\partial_hE_2$, then since arcs in $\Gamma_2\cap E_2$ are unknotted and unlinked, adding tubes along $\Gamma_2$ yields a stabilized Heegaard surface, a contradiction.  Thus every arc in $\Gamma_2\cap E_2$ connects the two disk components of $\partial_hE_2$.  Furthermore, the argument above implies that if $\Gamma_2\cap E_2$ contains more one arc then tubing along $\Gamma_2$ also yields a stabilized Heegaard surface.  Thus if $\Gamma_2\cap E_2\ne\emptyset$, then $\Gamma_2\cap E_2$ is a single unknotted arc connecting the two disks of $\partial_hE_2$.

Let $\hat{\beta}$ be an arc in  $\Gamma_2\cap E_2$.  The isotopy in Claim 3 of Lemma~\ref{Lalmostcarry} (as illustrated in Figure~\ref{iso}(a)) on $E_2$ changes $\hat{\beta}$ to an arc isotopic to a vertical arc in $N(B_2)$ (since $E_2$ is innermost and each arc in $\Gamma_2$ is properly embedded in $\overline{M-S_2}$), see the dashed vertical arc in Figure~\ref{iso}(a) for a schematic picture of $\hat{\beta}$.  So we may choose the arc $\hat{\beta}$ after this isotopy on $S_2$ to be a vertical arc in $N(B_2)$, and this isotopy reduces the length of $\hat{\beta}$ back to 1 (see Definition~\ref{Dvert}).

Let $Y=E_2-\Int(N(B_2))$ be the almost $D^2\times I$ component of $M-\Int(N(B_2))$ inside the innermost $D^2\times I$ region $E_2$. 
For an arc $\alpha$ in $\Gamma_2$ that does not intersect $\Int(Y)$, the isotopy may stretch $\alpha$ into a longer arc which is equivalent to adding a vertical arc of $Y$ (plus a pair of possible vertical arcs in $N(B_2)$ at the two ends) to the original $\alpha$, see the dashed non-vertical arc in Figure~\ref{iso}(a) for a schematic picture of $\alpha$.   As $\alpha$ is an almost vertical arc by our hypothesis, the new arc after this isotopy remains an almost vertical arc (for the surface after isotopy) with respect to $B$.  Moreover, the length of the resulting almost vertical arc is increased by at most 2, since the isotopy may stretch both ends of $\alpha$ through $Y$.  Therefore, after we finish all the isotopies in Lemma~\ref{Lalmostcarry}, we obtain an almost strongly irreducible Heegaard surface $S'$ and there is no non-trivial good $D^2\times I$ region for $S'$ and $B$.

Logically it is possible that there is a $D^2\times I$ region $E_g$ for $S'$ that is not good, though by our construction, no such $D^2\times I$ region exists for $S_1$.  Let $\Gamma'$ be the set of almost vertical arcs associated to $S'$.  By Lemma~\ref{Lstuff}, the non-disk components of $E_g\cap S'$ are a collection of $\partial$-parallel annuli non-nested in $E_g$.   This case seems trickier than Corollary~\ref{Cstuff} because a compressing disk of $E_g\cap S'$ might intersect $\Gamma'$, but next we show that this cannot happen.  Let $A_g$ be an annular component of $E_g\cap S'$ and let $\Delta_g$ be a compressing disk for $A_g$ in $E_g$, where $\partial\Delta_g$ is a core curve of $A_g$.   We claim that $\partial\Delta_g$ is an essential curve in $S'$.  Suppose $\partial\Delta_g$ bounds a disk $O_g$ in $S'$.  As before, by capping off the boundary circle of a disk component of $O_g-\Int(A_g)$ using a disk $B$-parallel to a component of $\partial_h E_g$, we get a normal or an almost normal 2-sphere carried by $B$, which contradicts Lemma~\ref{LimmersedS2}.  Thus $\partial\Delta_g$ is an essential curve in $S'$.  Let $H_1'$ and $H_2'$ be the two submanifolds of $M$ bounded by $S'$ and suppose $\Delta_g$ is a compressing disk in $H_1'$.  We may view the Heegaard surface $S$ as the surface obtained by adding tubes to $S'$ along arcs in $\Gamma'$.  Let $H_1$ and $H_2$ be the two handlebodies in the Heegaard splitting along $S$ corresponding to $H_1'$ and $H_2'$ respectively.  If $\Delta_g\cap\Gamma'\ne\emptyset$ (or if $\Delta_g$ and $\Gamma'$ lie on the same side of $S'$), then since $\Delta_g\subset H_1'$, the meridional disks of the added tubes along $\Gamma'$ are compressing disks in $H_2$.  As $S'$ is the surface obtained from compressing $S$ along the meridional disks of these tubes, $H_1'$ is the manifold obtained by adding 2-handles to $H_1$ (along the meridians the tubes).  However, by \cite{CG}, the hypothesis that $S$ is strongly irreducible implies that $\partial H_1'$ is incompressible in $H_1'$, which contradicts the assumption that $\Delta_g$ is a compressing disk in $H_1'$.  Thus $\Delta_g$ and $\Gamma'$ lie on different sides of $S'$, in particular, $\Delta_g\cap\Gamma'=\emptyset$.
So the compressing disks for the annular components of $E_g\cap S'$ are disjoint from $\Gamma'$.  Hence we can compress and isotope $S'$ as in Corollary~\ref{Cstuff} (see Figure~\ref{iso2}(a)) and change $E_g$ into a good $D^2\times I$ region.  Note that the total number of compressions is bounded by the genus of $S$, so after finitely many operations and isotopies as above, we may assume there is no non-trivial $D^2\times I$ region for our final almost strongly irreducible Heegaard surface $S'$.  

Recall that in the proof of Lemma~\ref{Lalmostcarry}, after each step of the isotopies (i.e. all the isotopies in Claim 3 of Lemma~\ref{Lalmostcarry}) on a  non-trivial simple $D^2\times I$ region $E$, either 
\begin{enumerate}
\item the resulting surface is no longer fully carried by the branched surface and we have to take a sub-branched surface in the next step of the isotopies, or
\item the belt of $E$ is no longer a belt of any simple $D^2\times I$ region for the surface after isotopy, or
\item the belt $V$ of $E$ is a belt of a simple $D^2\times I$ region for the surface after isotopy, but the complexity of $V$ is increased.  Note that by definition, the complexity $C_S(V)\le |M-B|$.
\end{enumerate}

Let $E$ be the simple $D^2\times I$ region in Claim 3 of Lemma~\ref{Lalmostcarry}.  By our argument above, if an almost vertical arc associated to the surface lies in $E$, then the isotopy (pushing a disk in $\partial_hE$ across $E$) changes this arc back to a vertical arc of $N(B)$.  This means that, although we may need to push disks across $E$ several times to get a surface not fully carried by $N(B')$ in Claim 3 of Lemma~\ref{Lalmostcarry}, the length of each associated almost vertical arc is increased by at most 2 after we finish all the isotopies across $E$ in Claim 3 of Lemma~\ref{Lalmostcarry}.

If (1) and (2) above do not occur after one step of the isotopies (i.e. all the isotopies in Claim 3 of Lemma~\ref{Lalmostcarry}), then we perform another isotopy on a simple $D^2\times I$ region of which $V$ is the belt.  As the complexity $C_S(V)\le |M-B|$, after finite steps of such isotopies, either the resulting surface is no longer fully carried by $B$ or $V$ is no longer the belt of a $D^2\times I$ region.  If the resulting surface is no longer fully carried by $N(B)$, then we take a sub-branched surface of $B$ that fully carries the surface.  As $B$ has only finitely many sub-branched surfaces, the possibility (1) above can occur only finitely many times.  
Guaranteed by Claim 4 of Lemma~\ref{Lalmostcarry}, this means that there is a number $K$ depending on the branched surface $B$, such that we only need at most $K$ steps of such isotopies to eliminate all the non-trivial $D^2\times I$ regions.  Thus the above argument means that the length of each almost vertical arc associated to the final surface $S'$ is at most $2K+1$.  
So by enumerate all possible sub-branched surfaces and counting components of the branch locus, we can calculate (an upper bound for) $K$.  We would like to emphasize that the proof above shows the existence of such a surface $S'$.  There is an algorithm to calculate $K$, but we do not have an algorithm to find $S'$.  Nonetheless, we will show below that if the surface $S'$ above is given, then we can use the bound on the length of associated arcs to recover the original Heegaard surface. 

Suppose we have found our final almost strongly irreducible Heegaard surface $S'$ carried by the branched surface $B$ in the lemma. Up to isotopy, there are only finitely many subarcs of $I$-fibers properly embedded in a component of $\overline{M-S'}$.  Moreover, as $B$ has only finitely many possible sub-branched surfaces, up to isotopy, there are only finitely many almost vertical arcs for $S'$ with length at most $2K+1$.  By enumerating and tubing along all possible almost vertical arcs of length at most $2K+1$ using all possible sub-branched surfaces of $B$, we can construct a finite set of surfaces.  Using Haken's algorithm \cite{Ha} and the algorithm to recognize a 3-ball \cite{Th}, we can determine whether or not each side of a surface is a handlebody.  So we can determine which surfaces in our list are Heegaard surfaces. Thus we get a finite set of Heegaard surfaces, one of which is isotopic to $S$.  
\end{proof}

\section{Sum of surfaces carried by a branched surface}\label{Ssum}

As we describe in section~\ref{Spre}, given a finite set of closed surfaces $T_1,\dots, T_n$ carried by $N(B)$, the sum of these surfaces $T=\sum_{i=1}^nn_iT_i$ is a surface obtained by canonical cutting and pasting on copies of the $T_i$'s along the intersection curves.  In this section, we will obtain information of $T$ from certain intersection patterns of the $T_i$'s.  Our main goal in this section is to show that if an almost strongly irreducible Heegaard surface $S$ has no non-trivial $D^2\times I$ region, then either the coefficient of some torus summand for $S$ is bounded or the branched surface has some nice property, see Lemma~\ref{Lflare}.

\begin{definition}\label{Dproduct}
Let $T_1$ and $T_2$ be closed and orientable surfaces carried by $N(B)$ and suppose $T_1$ is transverse to $T_2$.  Suppose $T_i$ ($i=1,2$) has a subsurface $F_i$ such that $\partial F_1=\partial F_2=F_1\cap F_2\subset T_1\cap T_2$.  We say that $F_1$ and $F_2$ bound a (pinched) \textbf{product region} if 
\begin{enumerate}
\item $F_1\cup F_2$ bounds a handlebody $X=F\times [1,2]/\mathord{\sim}$ with $(x,s)\sim (x,t)$ for any  $x\in\partial F$ and $s,t\in [1,2]$, where $F_i$ is (the image of) $F\times\{i\}$ ($i=1,2$) in $\partial X$; and 
\item there is a small neighborhood $A_F$ of $\partial F$ in $F$ such that for any $x\in\Int(F)\cap A_F$,
  $\{x\}\times [1,2]$ is a subarc of an $I$-fiber of $N(B)$.
\end{enumerate}
Note that if $F_i$ is a disk, then $X$ is basically a $D^2\times I$ region (for $T_1+T_2$).  If $F_i$ is an annulus, then $X$ is a solid torus of the form $bigon\times S^1$ and in this case we simply call $X$ a $bigon\times S^1$ region.   We say $X$ is a \textbf{trivial product region} if $X\subset N(B)$ and for each $x\in\Int(F)$, $\{x\}\times [1,2]$ is a subarc of an $I$-fiber of $N(B)$.  We say $X$ is \textbf{innermost} if $X\cap T_i=F_i$ for both $i=1,2$.  Note that if $X$ is a trivial product region, since $F_i$ is carried by $N(B)$, there must be an innermost trivial product region inside $X$.  Suppose $X$ is an innermost trivial product region, then we can perform a $B$-isotopy on $T_i$, which we call a \textbf{trivial isotopy}, to eliminate $X$ (as well as the set of double curves $\partial F_i\subset T_1\cap T_2$) by pushing either $F_1$ or $F_2$ across $X$. 
\end{definition}

Before we proceed, we would like to point out a reason why trivial isotopy is important for surfaces carried by $N(B)$.  

\begin{proposition}\label{Pproduct}
Given any two embedded orientable and closed surfaces $F$ and $G$ carried by $N(B)$, if $F\cap G\ne\emptyset$ but $F$ can be isotoped disjoint from $G$ via $B$-isotopy, then $F$ and $G$ must form a trivial product region and one can isotope $F$ disjoint from $G$ using a sequence of trivial isotopies.  
\end{proposition}
\begin{proof}
This proposition is similar to a theorem of Waldhausen \cite[Proposition 5.4]{W}.  Let $h:F\times I\to N(B)$ be the $B$-isotopy that moves $F$ disjoint from $G$.  We may suppose (1) $h(F\times\{0\})=F$, (2) $h(F\times\{1\})\cap G=\emptyset$, (3) $h(\{x\}\times I)$ lies in an $I$-fiber of $N(B)$ for each $x\in F$,  and (4) each $h(F\times\{t\})$ is an embedded surface in $N(B)$ transverse to the $I$-fibers.  Moreover, we may choose $h$ to be transverse to $G$ in the sense that $h^{-1}(G)$ is a collection of surfaces properly embedded in $F\times I$.  Since $h(F\times\{1\})\cap G=\emptyset$, the boundary of $h^{-1}(G)$ lie in $F\times\{0\}$.  Thus by a theorem of Waldhausen \cite{W}, a component $G_1$ of $h^{-1}(G)$ and a subsurface $F_1$ of $F\times\{0\}$ bound an innermost product region $P$ in $F\times I$, in particular, $F_1\cap G_1=\partial F_1=\partial G_1$, $P\cap (F\times\{0\})=F_1$ and $P\cap h^{-1}(G)=G_1$.  As $P$ is innermost, $h(\Int(P))\cap G=\emptyset$.

Let $N_G$ be the compact 3-manifold obtained by cutting $N(B)$ open along $G$.  Let $G_+$ and $G_-$ be the two components of $\partial N_G$ corresponding to the two sides of $G$.  Since $h(F\times\{0\})=F$ and $h(\Int(P))\cap G=\emptyset$, we may view $F_1'=h(F_1)$ as a surface properly embedded in $N_G$ with $\partial F_1'$ lying $G_+$ or $G_-$.  Without loss of generality, we may assume $\partial F_1'\subset G_+$.  The manifold $P$ has an induced product structure from $F\times I$.  For any vertical arc $\alpha$ properly embedded in $P$ ($\alpha$ is a subarc of an $I$-fiber of $F\times I$), we denote the two endpoints of $\alpha$ by $x_\alpha=\alpha\cap F_1$ and $y_\alpha=\alpha\cap G_1$.  Then the product structure of $P$ gives rise to a projection $p:F_1'\to G_+$ with $p(h(x_\alpha))=h(y_\alpha)$ for each such arc $\alpha$.  So $F_1'$ can be projected into $G_+$ in $N_G$ and this means that $F_1'$ is null-homologous in $H_2(N_G, G_+)$ and hence $F_1'$ is separating in $N_G$.

Note that $N_G=\overline{N(B)-G}$ has an induced $I$-fiber structure from $N(B)$.  Let $\alpha$ be a vertical arc in $N_G$ and we call $\alpha$ a vertical arc between $F_1'$ and $G_+$ if (1) $\Int(\alpha)\cap F_1'=\emptyset$, (2) one endpoint $x$ of $\alpha$ lies in $F_1'$ and the other endpoint of $\alpha$ is $p(x)\in G_+$. 
 Let $F_1''$ be the union of the points $x\in F_1'$ that is an endpoint of a vertical arc between $F_1'$ and $G_+$ described above.  We may view $\partial F_1'\subset F_1''$.   As $F_1'$ is a subsurface of $F$ and is transverse to the $I$-fibers, by the product structure of $P$, a small neighborhood of $\partial F_1'$ in $F_1'$ clearly lies in $F_1''$.  If $F_1''=F_1'$, then by our construction, $F_1'$ and a subsurface of $G_+$ bounded by $\partial F_1'$ bound a trivial product region in $N_G$.  This implies that $F_1'$ ($F_1'\subset F$) and a subsurface of $G$ bound a trivial product region $P'$ in $N(B)$.  If $F\cap\Int(P')\ne\emptyset$, then one can find an innermost trivial product region inside $P'$.  Hence Proposition~\ref{Pproduct} holds. 

Suppose $F_1''\ne F_1'$.  By our construction, $F_1''$ is a closed set in $F_1'$ containing a small neighborhood of $\partial F_1'$. Let $x\in\Int(F_1')$ be a boundary point of $F_1''$.  Let $\alpha_x$ be the vertical arc in $N_G$ between $F_1'$ and $G_+$ with $x\in\partial\alpha$.  As $x$ is a boundary point of $F_1''$, we have the following two cases.  The first case is that $\alpha$ contains a vertical arc of $\partial_vN(B)$.  This case is impossible because the homotopy pushing $F_1'$ to $G_+$ above is totally inside $G_+\cup\Int(N_G)$.  The second case is that the other endpoint of $\alpha_x$ is a point in $\partial F_1'$ and if one slightly move $x$ to a point $x' \in F_1'-F_1''$, then the interior of the corresponding vertical arc (connecting $x'$ to $p(x')$) intersects $F_1'$.  However, if this case happens, since a neighborhood of $\partial F_1'$ lies in $F_1''$, a vertical arc of $N(G)$ (one can use a subarc of the vertical arc connecting $x'$ to $p(x')$ above) connects one side of $F_1'$ to the other side.  This means that $F_1'$ is non-separating in $N_G$, contradicting our conclusion above on $F_1'$.
\end{proof}

Let $T_1,\dots, T_m$ be a collection of embedded closed surfaces in general position and carried by $N(B)$.  We can define a \textbf{complexity} of the intersection to be $(t,d)$ in lexicographic order, where $t$ is the number of triple points and $d$ is the number of double curves in the intersection.  Suppose two surfaces, say $T_1$ and $T_2$ in this collection form an innermost trivial product region $P$, i.e., a collection of double curves in $T_1\cap T_2$ bound surfaces $F_1\subset T_1$ and $F_2\subset T_2$ such that $F_1\cup F_2$ bounds the trivial product region $P$ and $\Int(P)\cap (T_1\cup T_2)=\emptyset$.  Note that other surfaces $T_i$ ($i\ne 1,2$) may intersect $P$.  Let $t_1$ and $t_2$ be the numbers of triple points (in the intersection of the $T_i$'s) lying in $F_1$ and $F_2$ respectively.
Without loss of generality, we may assume $t_1\le t_2$, and if $t_1=t_2$, we assume the number of double curves in $F_1$ is no larger than the number of double curves in $F_2$.  Then we can perform a trivial isotopy as above on $T_2$, pushing $T_2$ across $P$ to eliminate the product region $P$.  This operation is basically replacing $T_2$ by $(T_2-F_2)\cup F_1$. After a small perturbation, this operation eliminates the triple points in $F_2$ but gains copies of triple points in $\Int(F_1)$.  Since $t_1\le t_2$, this $B$-isotopy on $T_2$ does not increase the total number of triple points.  In fact the isotopy eliminates all the triple points (if any) that lie on $\partial F_i$.  As the double curves $\partial F_1=\partial F_2$ are eliminated, by our assumption above on the number of double curves in $F_1$ and $F_2$, this operation reduces the complexity.  Using this complexity, we can conclude that after some trivial isotopies, no pair of surfaces in this collection form any trivial product region.

Let $T_1,\dots, T_m$ be closed surfaces carried by $N(B)$ and in general position.  Let $\Gamma=\bigcup_{i=1}^mT_i$.  Suppose each $T_i$ is a separating surface in $M$. This implies that for any component $N$ of $M-\Gamma$, the inclusion map $i:N\to M$ naturally extends to an embedding/inclusion $i:\overline{N}\to M$ where $\overline{N}$ is the closure of $N$ under path metric. In other words, no two points in $\partial\overline{N}$ correspond to the same point in $\Gamma$, and $\overline{N}$ is an embedded compact submanifold of $M$. 

As illustrated in Figure~\ref{deform}(a), the 2-complex $\Gamma$ can be naturally deformed into a branched surface $B^\Gamma$ lying in $N(B)$ and transverse to the $I$-fibers.  The intersection curves of these $T_i$'s correspond to the branch locus of $B^\Gamma$.  We may identify each component of $M-\Gamma$ to a corresponding component of $M-B^\Gamma$ and the only difference is that, when viewed as a component of $M-B^\Gamma$, some corners of a component of $M-\Gamma$ are smoothed out and some corners become cusps.  

Let $X$ be the closure of a component of $M-\Gamma$ (or $M-B^\Gamma$) and let $\alpha$ be a simple closed curve in $\partial X$.  Since each $T_i$ is separating and by the discussion above, $\alpha$ is a simple closed curve in the 2-complex $\Gamma$.  We say $\alpha\subset\partial X$ is a \textbf{good curve} if $\alpha$ becomes a smooth curve once we deform $\Gamma$ into the branched surface $B^\Gamma$ as above, i.e. $\alpha$ does not cross the cusps.  If $\alpha\subset\partial X$ is a good curve, then the intersection of $B^\Gamma$ with a small/thin vertical annulus in $N(B)$ that contains $\alpha$ is a train track $\tau_\alpha$ and $\tau_\alpha$ consists of the smooth circle $\alpha$ and possibly some ``tails" all lying on the same side of $\alpha$.

\begin{lemma}\label{Lgood}
Let $T_1,\dots, T_m$, $\Gamma$ and $X$ be as above and let $\gamma\subset\partial X$ be a good curve.
Let $S=\sum_{i=1}^mn_iT_i$.  Suppose there is some number $k$ such that $n_i\ge k$ for each $n_i$, then $S$ contains at least $k$ disjoint simple closed curves $B$-isotopic to $\gamma$.
\end{lemma}
\begin{proof}
Let $A$ be a small vertical annulus in $N(B)$ that contains $\gamma$.  Let $\tau=A\cap \Gamma$. By our definition of good curve, after deforming $\Gamma$ into a branched surface, we may view $\tau$ as a train track which consists of $\gamma$ and some ``tails" all on the same side of $\gamma$.  We may assume $S$ lies in a small neighborhood of $\Gamma$.  This implies that $S\cap A$ consists of some simple closed curves $B$-isotopic to $\gamma$ and some arcs whose endpoints all lie in the same component of $\partial A$ (in particular, there is no spirals in $A\cap S$).  Since $S$ is transverse to the $I$-fibers, it is easy to see from the above property of $S\cap A$ that there must be a vertical arc of $A$ that does not intersect any of these $\partial$-parallel arcs in $S\cap A$, i.e. the vertical arc only intersects the closed curves in $S\cap A$.  Since $n_i\ge k$ for each $n_i$, $S\cap A$ contains at least $k$ simple closed curves $B$-isotopic to $\gamma$. 
\end{proof}

\begin{lemma}\label{LgoodH}
Let $T_1,\dots, T_m$ and $\Gamma$ be as above.  Suppose  $N(B)$ fully carries $\Gamma$, i.e., every $I$-fiber of $N(B)$ intersects some $T_i$.  Then for any simple closed curve $C$ in $\partial_hN(B)$, there is a good curve in $\Gamma=\cup_{i=1}^mT_i$ that is $B$-isotopic to $C$.
\end{lemma}
\begin{proof}
Let $S=\sum_{i=1}^mT_i$.  Our hypotheses imply that $S$ is fully carried by $B$.  Thus there is a simple closed curve $C'$ in $S$ such that $C'\cup C$ bounds a vertical annulus $A$ in $N(B)$ and $A\cap S=C'$.  This means that, before the canonical cutting and pasting for $S=\sum_{i=1}^mT_i$, $C'$ corresponds to a good curve in $\Gamma=\cup_{i=1}^mT_i$.
\end{proof}

\begin{figure}
\begin{center}
\psfrag{(a)}{(a)}
\psfrag{(b)}{(b)}
\psfrag{(c)}{(c)}
\psfrag{d}{deform}
\psfrag{al}{$\alpha$}
\psfrag{be}{$\beta$}
\includegraphics[width=5.0in]{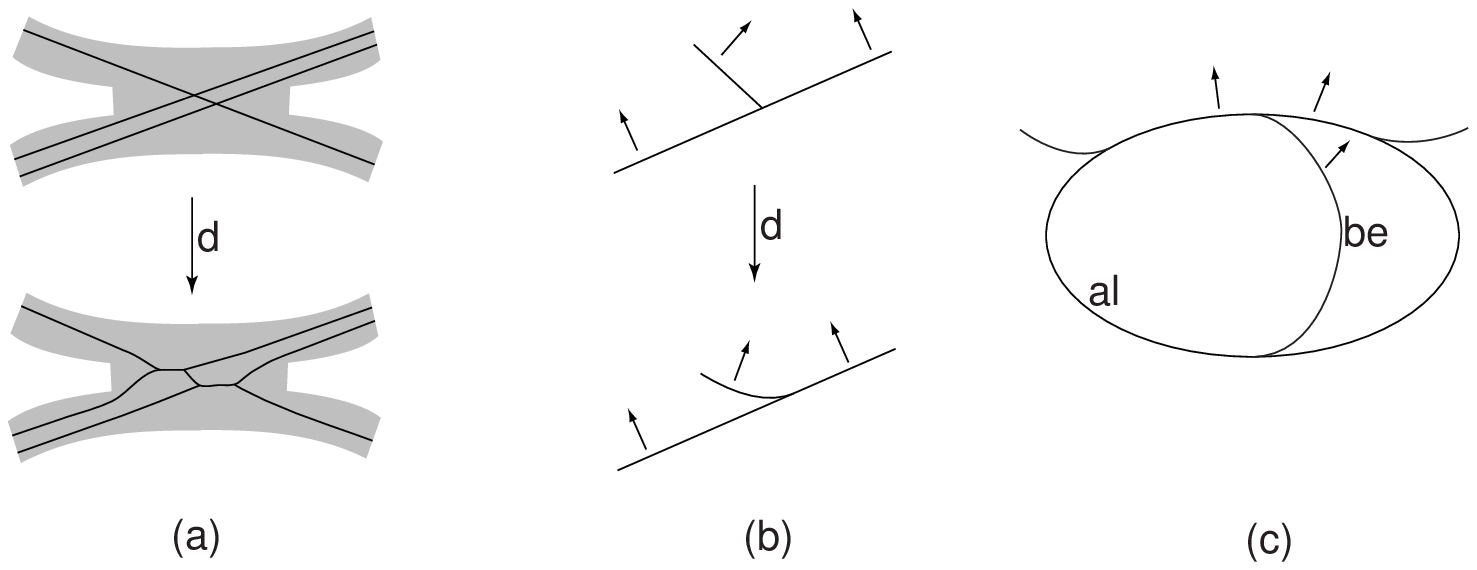}
\caption{}\label{deform}
\end{center}
\end{figure}

\begin{definition}\label{Dflare}
Let $D$ be a disk and $F$ a compact planar surface all carried by $N(B)$ and with either $D\cap F=\emptyset$ or $D\cap F=\partial D\subset\partial F$. 
Let $\gamma_0,\dots,\gamma_k$ be the boundary curves of $F$. Suppose $\gamma_0$ and a small annular neighborhood $A_0$ of $\gamma_0$ in $F$ are $B$-parallel to $\partial D$ and a small annular neighborhood of $\partial D$ in $D$ respectively.  Moreover, if $D\cap F\ne\emptyset$, we require $\gamma_0=\partial D=D\cap F$.  Let $F'$ be a maximal subsurface of $F$ that contains $A_0$ and is $B$-parallel to a subsurface of $D$.  We say $F$ is a \textbf{flare} based at $D$ if $F-F'$ consists of annular neighborhoods of $\gamma_1,\dots,\gamma_k$ in $F$, see Figure~\ref{flare} for a schematic picture of a flare (in this picture, $F$ is an annulus and $D\cap F=\partial D\subset\partial F$).  So a flare surface $F$ is a piece of surface that is not $B$-parallel to a subsurface of $D$, but a large subsurface $F'$ of $F$ is $B$-parallel to a subsurface of $D$ containing $\partial D$.  Let $\gamma_1',\dots,\gamma_k'$ be the components of $\partial F'-\gamma_0$ and let $\alpha_i$ ($i=1,\dots,k$) be the curve in $D$ that is $B$-parallel to $\gamma_i'$.  So the planar subsurface of $D$ bounded by $\partial D$ and the $\alpha_i$'s is $B$-parallel to $F'$.  We call the $\alpha_i$'s the \textbf{flare locus}.  
Note that $\alpha_i\cup\gamma_i'$ bounds a vertical annulus $A_i$ in $N(B)$.  If a vertical arc of $A_i$ does not contain a vertical arc of $\partial_vN(B)$, then one can enlarge $F'$ a little to get a slightly larger subsurface of $F$ $B$-parallel to a subsurface of $D$, which contradicts our assumption that $F'$ is maximal. So each vertical arc of $A_i$ must contain a vertical arc of $\partial_vN(B)$.  In particular, $\alpha_i\subset\pi^{-1}(L)$ where $L$ is the branch locus of $B$ and $\pi:N(B)\to B$ is the projection.  Moreover, the normal direction of $\alpha_i$ in $D$ induced from the branch direction (of the corresponding arcs in the branch locus)  points out of the subdisk of $D$ bounded by $\alpha_i$, see Figure~\ref{flare}.      
Note that if $k\ge 2$, then one can find a simple closed curve in $F'$ that cuts off an annular neighborhood of some $\gamma_i$ ($i\ge 1$) in $F$ such that this annular subsurface of $F$ is a flare based at a subdisk of $D$.   
Let $D_i$ ($i=1,\dots,k$) be the subdisk of $D$ bounded by the component $\alpha_i$ of the flare locus.  
We say a flare $F$ and its base $D$ are \textbf{innermost} if its flare locus is innermost in $D$ in the sense that for any flare surface $G$ based at a subdisk of $D$ and with the flare locus of $G$ in $\bigcup_{i=1}^kD_i$, then the flare locus of $G$ and the flare locus of $F$ are the same.  If the flare $F$ is innermost, then $F$ must be an annulus, since otherwise an annular neighborhood of $\gamma_i$ ($i\ge 1$) in $F$ described above is such a flare $G$ based at a subdisk of $D$.  We say a branched surface $B$ has a flare if there are flare surface $F$ and base disk $D$ carried by $B$ as above.  Clearly if there is a flare based at $D$, then there must be an innermost flare based at a subdisk of $D$.
\end{definition}

\begin{figure}
\begin{center}
\psfrag{F}{$F$}
\psfrag{D}{$D$}
\includegraphics[width=1in]{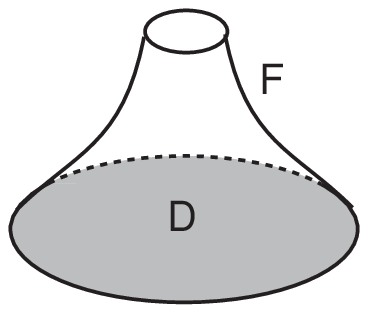}
\caption{}\label{flare}
\end{center}
\end{figure}

The notion of flare can be viewed as a generalization of non-trivial $D^2\times I$ region, see Lemma~\ref{LflareD2}.

\begin{lemma}\label{LflareD2}
Let $S$ be any surface carried by $N(B)$.  If there is a non-trivial $D^2\times I$ region for $S$ and $B$, then $B$  has a flare.
\end{lemma}
\begin{proof}
This lemma is fairly obvious. Let $E$ be a non-trivial $D^2\times I$ region with $\partial_hE=D_0\cup D_1$ and $\partial_vE=A$.  After enlarge $E$ a little if necessary, we may assume a small annular neighborhood of $\partial D_0$ in $D_0$ is $B$-parallel to   an annular neighborhood of $\partial D_1$ in $D_1$.  Since $E$ is non-trivial, $D_0$ is not $B$-parallel to $D_1$.  So one can find a sufficiently large neighborhood of $\partial D_0$ in $D_0$ that is not $B$-isotopic to a subsurface of $D_1$.  This gives a flare based at $D_1$.
\end{proof}

\begin{corollary}\label{Cdisk}
Let $T_1,\dots, T_n$ be a collection of normal tori carried by $N(B)$.  Let $B_T$ be the sub-branched surface of $B$ that fully carries $\bigcup_{i=1}^nT_i$.  Suppose $B_T$ has no flare.  Then, after some $B_T$-isotopy on the $T_i$'s, every double curve in $T_i\cap T_j$ is essential in both $T_i$ and $T_j$, for any $i,j$.
\end{corollary}
\begin{proof} Using the complexity of the intersection of the $T_i$'s defined after Proposition~\ref{Pproduct}, we can conclude that, after some $B_T$-isotopies, the $T_i$'s do not form any trivial product region.  Thus we may assume $T_i$ and $T_j$ do not form any trivial product region for any $i, j$.  

Let $\alpha\subset T_i\cap T_j$ be a double curve and suppose $\alpha$ bounds a disk $D$ in $T_i$.  As both $T_i$ and $T_j$ are transverse to the $I$-fibers of $N(B_T)$, there is a collar annulus $A_\alpha$ of $\alpha$ in $T_j$ (with $\alpha\subset\partial A_\alpha$) that is $B_T$-parallel to an annular neighborhood of $\alpha$ in $D$.  If $\alpha$ is an essential curve in $T_j$, then as in the proof of Lemma~\ref{LflareD2} and since $\alpha$ is essential in $T_j$, we can extend $A_\alpha$ to a subsurface $P$ of $T_j$ such that $\alpha\subset\partial P$ and $P$ is a flare based at $D$, a contradiction to our hypotheses.

Suppose $\alpha$ also bounds a disk $D'$ in $T_j$.  By choosing $\alpha$ to be an innermost such curve, we may suppose $D\cup D'$ is an embedded 2-sphere.  Since $B$ does not carry any normal 2-sphere, $D$ and $D'$ must bound a product region which determines a $D^2\times I$ region.  Since $B_T$ has no flare, by Lemma~\ref{LflareD2}, such a $D^2\times I$ region must be trivial, which means that $D\cup D'$ bounds a trivial product region, a contradiction to the assumption at the beginning that there is no trivial product region. 
\end{proof}

\begin{lemma}\label{Linnermost}
Let $F$ be a flare based at $D$ as above and suppose $F$ and $D$ are innermost.  Let $\alpha\subset D$ be the flare locus, let $\gamma\subset F$ be the corresponding curve $B$-isotopic to $\alpha$, and let $A$ be the vertical annulus in $N(B)$ bounded by $\alpha\cup\gamma$.  Then there must be a component $V$ of $\partial_vN(B)$ such that $V\subset A$.
\end{lemma}
\begin{proof}
Since $F$ and $D$ are innermost, by our discussion in Definition~\ref{Dproduct}, $F$ must be an annulus.  For each $x\in D$, let $I_x$ be the $I$-fiber of $N(B)$ containing $x$. By fixing a normal direction for $D$, we can say one component of $I_x-x$, denoted by $I_x^+$, is on the positive side of $D$ and the other component of $I_x-x$, denoted by $I_x^-$, is on the negative side.  Note that it is possible that $I_x^\pm$ intersects other part of $D$.  As $\gamma\cup\alpha$ bounds an embedded vertical annulus $A$ in $N(B)$, we may suppose $\gamma$ is on the positive side of $D$ in this sense.  

Let $\Gamma$ be the union of all the points in $\Int(D)$ with the property that for each $x\in\Gamma$, $I_x^+$ (not the whole $I_x$) contains a vertical arc of $\partial_vN(B)$.  By the local picture of a branched surface (see Figure~\ref{branch}), it is clear that $\Gamma$ is a trivalent graph in $\Int(D)$ with each vertex corresponding to a double point of the branch locus of $B$.  By the definition of flare locus, $\alpha\subset\Gamma$.  Each arc in $\Gamma$ has a normal direction in $D$ inherited from the branch direction at the corresponding arc in the branch locus of $B$ (or $\partial_vN(B)$).
As illustrated in Figure~\ref{deform}(b) (also see \cite[Figure 2.3]{L1}), $\Gamma$ can be naturally deformed into a transversely oriented train track $\tau_\Gamma$ in $D$.  As we mentioned in the definition of flare locus, $\alpha$ is a smooth circle in the train track with induced normal direction pointing out of the subdisk $D_\alpha$ of $D$ bounded by $\alpha$.  Note that since the train track $\tau_\Gamma$ is transversely orientable, $\tau_\Gamma$ does not form any monogon. As every simple closed curve in $D$ bounds a disk in $D$, the no-monogon property implies that any simple closed curve $c$ carried by $\tau_\Gamma$ must correspond to a smooth simple closed curve in $\tau_\Gamma$ (in other words, in a fibered neighborhood of $\tau_\Gamma$ in $D$, the curve $c$ intersects each $I$-fiber at most once).

Let $D_\alpha$ be the subdisk of $D$ bounded by $\alpha$. 
We say $\alpha$ is \emph{innermost} if $\alpha$ is the only smooth simple closed curve in $\tau_\Gamma\cap D_\alpha$ with branch direction pointing out of the subdisk bounded by this curve.   
We say $\alpha$ is \emph{simple} if there is an annular collar neighborhood $A_\alpha$ of $\alpha$ in $D_\alpha$ such that $A_\alpha\cap\tau_\Gamma=\alpha$.  

\vspace{10pt}
\noindent
Claim.  If $\alpha$ is innermost then $\alpha$ must be simple.
\begin{proof}[Proof of the Claim]
The reason why this claim is true is that $B$ fully carries a closed surface.  Let $S$ be a closed surface fully carried by $N(B)$.  For any point $x\in\Gamma$, by our construction of $\Gamma$, $I_x^+$ contains one or two vertical arcs of $\partial_vN(B)$ and $I_x^+$ contains two vertical arcs of $\partial_vN(B)$ if and only if $x$ is a vertex of $\Gamma$.  Let $p_x$ be the component of $I_x^+-\partial_vN(B)$ that contains $x$ and let $V_x=I_x^+-p_x$.  By our construction, (the closure of) each component of $V_x-\partial_vN(B)$ can be horizontally pushed slightly into an $I$-fiber of $N(B)$, see Figure~\ref{iso}(c) for a schematic picture.  Since $S$ is fully carried by $B$, this implies that, for any point $x\in\Gamma$, we have $S\cap V_x\ne\emptyset$.  From the local pictures of $B$, it is easy to see that the intersection of $S$ with the union of the $V_x$'s ($x\in\Gamma$) is a collection of compact curves that can be fully carried by the train track $\tau_\Gamma$ after projecting these curves to $D$.  The fact that $\tau_\Gamma$ fully carries some compact curves implies that, if $\alpha$ is not simple (i.e., $\Int(A_\alpha)\cap\tau_\Gamma\ne\emptyset$), then there must be a smooth arc $\beta$ in $\tau_\Gamma$ such that $\beta$ is properly embedded in $D_\alpha$.  Since the train track $\tau_\Gamma$ is transversely orientable, as shown in Figure~\ref{deform}(c), $\beta$ and a subarc of $\alpha$ form a smooth circle with branch direction pointing outwards.  This contradicts the hypothesis that $\alpha$ is innermost.
\end{proof}

Let $A_\alpha$ be a small annular neighborhood of $\alpha$ in $D_\alpha$. 
If $\alpha$ is innermost, by the claim above, $A_\alpha\cap\Gamma=\alpha$.  The construction of $\Gamma$ plus $A_\alpha\cap\Gamma=\alpha$ implies that the small annulus $A_\alpha$ is $B$-isotopic to an annulus $A_h$ in $\partial_hN(B)$, where $A_h$ is an annular collar neighborhood (in $\partial_hN(B)$) of a boundary component $\alpha_v$ of $\partial_hN(B)$.  In particular, $\alpha_v$ is $B$-isotopic to $\alpha$ and is on the positive side of $D$.  Let $V$ be the component of $\partial_vN(B)$ that contains $\alpha_v$.  Since $\alpha$ is a flare locus, each vertical arc of $A$ (recall that $A$ is the vertical annulus bounded by $\alpha\cup\gamma$) must contain a vertical arc of $\partial_vN(B)$ (see Definition~\ref{Dflare}).  This implies that $V\subset A$ and the lemma holds.

It remains to prove that $\alpha$ is an innermost such smooth circle in $\tau_\Gamma$.  Suppose $\alpha$ is not innermost, then $D_\alpha$ contains an innermost such smooth simple closed curve $\alpha'$ in $\tau_\Gamma\cap D_\alpha$ with induced normal direction pointing out of the subdisk of $D_\alpha$ bounded by $\alpha'$.  Then the argument above implies that there is a boundary component $\alpha_v'$ of $\partial_hN(B)$ such that $\alpha'\cup\alpha_v'$ bounds a vertical annulus $P$ in $N(B)$.  Let $V'$ be the component of $\partial_vN(B)$ that contains $\alpha_v'$.  So $A'=P\cup V'$ is a vertical annulus in $N(B)$. Let $\gamma'=\partial A'-\alpha'$ be the other boundary circle of $A'$.  Then one can easily construct a small annulus $F'$ containing $\gamma'$ and transverse to the $I$-fibers of $N(B)$.  Since $A'$ contains a component of $\partial_vN(B)$, $F'$ is clearly not totally $B$-parallel to an annulus in $D_a$ and $F'$ is a flare based at a subdisk of $D_\alpha$ whose flare locus is $\alpha'$.  This contradicts the hypothesis that the flare $F$ is innermost.
\end{proof}

\begin{definition}\label{Dbalanced}
Let $A=S^1\times I$ be an annulus and $\alpha=S^1\times\{x\}$ ($x\in\Int(I)$) be a core curve of $A$.  We first fix a direction along $\alpha$.  Let $\lambda$ be a properly embedded essential arc in $A$ transverse to the $I$-fibers.  As $\lambda$ is an essential arc of $A$, we may suppose $\lambda\cap\alpha$ is a single point.  We assign $\lambda$ the direction along $\lambda$ pointing from the endpoint $\lambda\cap (S^1\times\{0\})$ to the endpoint $\lambda\cap (S^1\times\{1\})$.  Since $\lambda$ is transverse to the $I$-fibers, the projection $\pi:A\to\alpha$ maps $\lambda$ to a subarc of $\alpha$.  As shown in Figure~\ref{balance}(a), we say $\lambda$ is a positive arc if the induced direction of $\pi(\lambda)$ agrees with the direction of $\alpha$, and we call $\lambda\cap\alpha$ a positive intersection point.  Otherwise we say $\lambda$ and $\lambda\cap\alpha$ are negative.  Let $\Lambda$ be a collection of disjoint properly embedded essential arcs in $A$ and transverse to the $I$-fibers.  We say $\Lambda$ is \textbf{balanced} if the number of positive arcs  equals the number of negative arcs in $\Lambda$.
\end{definition}

\begin{figure}
\begin{center}
\psfrag{(a)}{(a)}
\psfrag{(b)}{(b)}
\psfrag{al}{$\alpha$}
\psfrag{cp}{cutting and pasting}
\psfrag{+}{$+$}
\psfrag{-}{$-$}
\includegraphics[width=5.0in]{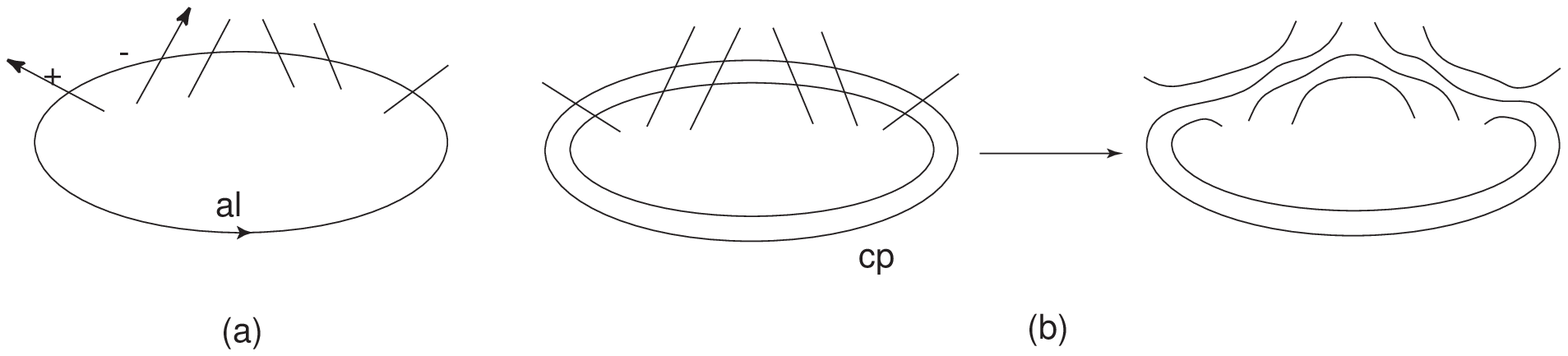}
\caption{}\label{balance}
\end{center}
\end{figure}

\begin{lemma}\label{Lbalanced}
Let $A=S^1\times I$ be a vertical annulus in $N(B)$ and let $\alpha$ be a core curve in $A$. Let $\Lambda$ be a collection of disjoint essential arcs in $A$ transverse to the $I$-fibers.  Suppose $\Lambda$ is balanced.  Then there is a number $k\le \frac{1}{2}|\Lambda|$ such that, if $m\ge k$, $\Lambda+m\alpha$ (the curve obtained by canonical cutting and pasting of $\Lambda$ and $m$ parallel copies of $\alpha$) consists of some $\partial$-parallel arcs in $A$ and $m-k$ circles parallel to $\alpha$.
\end{lemma}
\begin{proof}
The lemma is fairly obvious, see Figure~\ref{balance}(b).  The argument below also gives a way to determine $k$.

We use the notation in Definition~\ref{Dbalanced}.  In particular $\alpha$ has a fixed direction along $\alpha$ and each arc in $\Lambda$ intersects $\alpha$ in a single point.  Let $\beta$ be a subarc of $\alpha$ with $\partial\beta\subset\alpha\cap\Lambda$.  Using the direction along $\alpha$ we can call one endpoint of $\beta$ the starting point and the other endpoint of $\beta$ the ending point.  We say $\beta$ is a bottom arc if the starting point of $\beta$ (along this direction of $\alpha$) is a positive intersection point and the ending point of $\beta$ is a negative intersection point. 

We will inductively construct a sequence of sets of subarcs of $\alpha$.  Since $\Lambda$ is balanced, there is a maximal collection of disjoint subarcs of $\alpha$, denoted by $\Lambda_1$, such that  each arc $\beta$ in $\Lambda_1$ is a bottom arc and $\beta\cap \Lambda=\partial\beta$.  Let $\lambda_1$ be the union of the endpoints of the arcs in $\Lambda_1$.  Since arcs in $\Lambda_1$ are disjoint, the remaining intersection points $(\alpha\cap\Lambda)-\lambda_1$ is balanced in the sense that the number of positive intersection points equals the number of negative intersection points. 

Suppose we have inductively constructed sets of arcs $\Lambda_1,\dots,\Lambda_n$.  Let $\lambda_i$ be the endpoints of the arcs in $\Lambda_i$.  Suppose the remaining intersection points $\sigma_n=\alpha\cap \Lambda-\bigcup_{i=1}^n\lambda_i$ is balanced in the above sense.  If $\sigma_n\ne\emptyset$, since $\sigma_n$ is balanced, there is a maximal collection of disjoint subarcs of $\alpha$, denoted by $\Lambda_{n+1}$, such that each arc $\beta$ in $\Lambda_{n+1}$ is a bottom arc and $\beta\cap\sigma_n=\partial\beta$.  Since the collection is maximal at each step of the induction, the interior of each arc in $\Lambda_{i+1}$ must contain at least one arc of $\Lambda_i$.  So we can inductively construct a sequence of sets of arcs $\Lambda_1,\dots,\Lambda_k$ such that $\bigcup_{i=1}^k\lambda_i=\alpha\cap \Lambda$. The number $k$ depends on the intersection pattern of $\alpha$ with $\Lambda$ and can be easily determined.  In particular $k\le\frac{1}{2} |\alpha\cap\Lambda|$.

Now we consider $\Lambda+k\alpha$.  In $\Lambda+\alpha$, arcs in $\Lambda_1$ above are connected by subarcs of $\Lambda$ to form a collection of $\partial$-parallel arcs in $A$.  Moreover, we may view $\Lambda+2\alpha$ as $(\Lambda+\alpha)+\alpha$ and assume the $\partial$-parallel components of $\Lambda+\alpha$ do not intersect $\alpha$.  So inductively, as shown in Figure~\ref{balance}(b), the arcs $\Lambda_i$ in the $i$-th copy of $\alpha$ are connected by subarc arcs of $\Lambda$ forming a collection of $\partial$-parallel arcs in $A$ with endpoints in the same circle $S^1\times\{0\}$ (using the notation in Definition~\ref{Dbalanced}).  It follows from our construction of $\Lambda_i$, after cutting and pasting, $\Lambda+k\alpha$ consists of $\partial$-parallel arcs in $A$.  Since any additional copy of $\alpha$ can be isotoped to be disjoint from the $\partial$-parallel arcs, the lemma follows.
\end{proof}

The following lemma is similar in spirit to Lemma~\ref{Lbalanced}.  It deals with surfaces instead of curves.

\begin{lemma}\label{LbalancedD}
Let $F$ and $\Lambda$ be compact orientable surfaces carried by $N(B)$.  Suppose $F\cap \Lambda$ is a collection of simple closed curves in $F$. Let $k$ be the number of components of $F\cap\Lambda$ that are trivial in $F$. If $m\ge k$, then after $B$-isotopy, $(\Lambda+mF)\cap F$ contains no trivial curve in $F$.  In particular, if $F\cap\Lambda$ consists of trivial curves in $F$ and $m\ge k$, then $\Lambda+mF$ is disjoint from $F$ after $B$-isotopy.
\end{lemma}
\begin{proof}
Let $F\times I\subset N(B)$ be a product of $F$ and an interval with each $\{x\}\times I$ a subarc of an $I$-fiber of $N(B)$.  We may view $F=F\times\{1/2\}$ and by the hypothesis, we may assume $\Lambda\cap(F\times I)$ is a collection of annuli $A_1,\dots, A_n$ with each $A_i$ transverse to the $I$-fibers of $F\times I$ and $F\cap A_i$ is a simple closed curve in $F$.  Without loss of generality, we may assume $A_i\cap F$ is trivial in $F$ if $i\le k$.

For any $i\le k$, as $A_i\cap F$ is a trivial circle in $F$, $A_i+F$ is a pair of $\partial$-parallel surfaces in $F\times I$ and hence $A_i+F$ is disjoint from $F$ after $B$-isotopy.  Although the surface $A_i+F$ may intersect other annuli $A_j$'s ($j\ne i$), the intersection is disjoint from $F$ after $B$-isotopy. 
This means that after $B$-isotopy, the number of circles in $(\Lambda+F)\cap F=(F+\sum_{i=1}^nA_i)\cap F$ that are trivial in $F$ is at most $k-1$.  Thus inductively, $(\Lambda+kF)\cap F$ contains no trivial circles after $B$-isotopy and the lemma holds.
\end{proof}

In Lemma~\ref{LflareD2}, we showed that if there is a non-trivial $D^2\times I$ region, then $B$ has a flare.  The following Lemma can be viewed as a certain converse of Lemma~\ref{LflareD2}.

\begin{lemma}\label{Lflare}
Let $T_i$ ($i=0,1,\dots n$) and $H$ be closed surfaces carried by $N(B)$.  Suppose each $T_i$ is a normal torus and $S=H+\sum_{i=0}^nn_iT_i$ is either a strongly irreducible Heegaard surface or an almost strongly irreducible Heegaard surface.  If $S$ is an almost Heegaard surface, we suppose the total length of the almost vertical arcs associated to $S$ is bounded by a fixed number $K$ and suppose $S$ is $K$-minimal, see Definition~\ref{Dkmin}.  Let $B_T$ be the sub-branched surface of $B$ that fully carries $\bigcup_{i=0}^nT_i$. 
If $B_T$ contains a flare, then either there is a non-trivial $D^2\times I$ region for $S$ and $B$, or there is a number $k$ depending on $K$ and the intersection pattern of $H$ with the $T_i$'s such that some coefficient $n_i$ is smaller than $k$.  Furthermore, $k$ can be algorithmically determined.
\end{lemma}
\begin{proof}
By the proof of Lemma~\ref{Linnermost}, if $N(B_T)$ contains a flare, then using an innermost flare,  we have a component $V$ of $\partial_vN(B_T)$ and a boundary circle $\alpha$ of $V$ with the following properties
\begin{enumerate}
\item  $\alpha$ bounds an embedded disk $\Delta$ carried by $N(B_T)$, ($\alpha$ corresponds to the flare locus for the innermost flare and $\Delta$ is the subdisk of the base disk bounded by the flare locus.)
\item let $P_\alpha$ be the component of $\partial_hN(B_T)$ containing $\alpha$, then a small annular neighborhood of $\alpha$ in $P_\alpha$ is $B_T$-parallel to an annular neighborhood of $\alpha$ in $\Delta$.  In other words, the normal direction at $\partial\Delta$ induced from the branch direction points out of $\Delta$.
\end{enumerate}
Moreover, since the flare in Lemma~\ref{Linnermost} is innermost, we may assume that no proper subdisk of $\Delta$ is a base of a flare.

Let $\beta$ be the other boundary curve of $V$ and let $P_\beta$ be the component of $\partial_hN(B_T)$ that contains $\beta$.

By Lemma~\ref{LgoodH}, there are good curves $\gamma_\alpha$ and $\gamma_\beta$ in the 2-complex $\cup_{i=0}^nT_i$ that are $B_T$-isotopic to $\alpha$ and $\beta$ respectively. We may assume $\gamma_\alpha$ and $\gamma_\beta$ are close to $\alpha$ and $\beta$ respectively and $\gamma_\alpha\cup\gamma_\beta$ bounds a vertical annulus containing $V$. Let $\hat{T}=\sum_{i=0}^nn_iT_i$, so the surface $S$ in our lemma can be expressed as $S=H+\hat{T}$.  Suppose each $n_i\ge k$ for some number $k$.  Next we will use $\alpha$ and $\beta$ to show that if $k$ is large, then there is a  non-trivial $D^2\times I$ region for $S$ and this implies the lemma.  

By Lemma~\ref{Lgood}, there are at least $k$ curves in $\hat{T}$ that are $B_T$-parallel to $\gamma_\alpha$ and $\alpha$, and we denoted these $k$ curves by $\alpha_1,\dots,\alpha_k$.  Similarly by Lemma~\ref{Lgood}, there are at least $k$ curves in $\hat{T}$ that are $B_T$-parallel to $\gamma_\beta$ and $\beta$, and we denoted these $k$ curves by $\beta_1,\dots,\beta_k$. Note that, since each $T_i$ is separating, $\gamma_\alpha$ and $\gamma_\beta$ must be disjoint good curves in the 2-complex $\cup_{i=0}^nT_i$.  So the $\alpha_i$'s and $\beta_i$'s are disjoint curves in $\hat{T}$.  Moreover we may assume the $\alpha_i$'s and $\beta_i$'s are very close to $\alpha$ and $\beta$ respectively. 
Since $\alpha$ and $\beta$ are the boundary curves of a component $V$ of $\partial_vN(B_T)$, similar to the assumption on $\gamma_\alpha\cup\gamma_\beta$ above, we may assume $\alpha_i\cup\beta_j$ bounds a vertical annulus containing $V$ for any $i, j$.

\vspace{10pt}
\noindent
Claim.  Each $\alpha_i$ ($i=1,\dots,k$) bounds a disk in $\hat{T}$ which is $B_T$-isotopic to $\Delta$.
\begin{proof}[Proof of the Claim]
We first consider $\Delta\cap\hat{T}$.  Since $\alpha=\partial\Delta\subset\partial_vN(B)$, we may assume $\Delta\cap\hat{T}\subset\Int(\Delta)$.  Let $\gamma$ be a component of $\Delta\cap\hat{T}$ which is innermost in $\Delta$ and let $\Delta_\gamma\subset\Delta$ be the subdisk of $\Delta$ bounded by $\gamma$.  Since both $\Delta$ and $\hat{T}$ are transverse to the $I$-fibers, there is a small collar annulus $T_\gamma$ of $\gamma$ in $\hat{T}$ such that $\gamma$ is a component of $\partial T_\gamma$ and $T_\gamma$ is $B_T$-isotopic to an annular neighborhood of $\gamma$ in $\Delta_\gamma$.  If $\gamma$ bounds a disk in $\hat{T}$ that contains $T_\gamma$ and is $B_T$-parallel to $\Delta_\gamma$, then we can perform a $B_T$-isotopy on $\hat{T}$ to eliminate $\gamma$.  If $\gamma$ does not bound such a disk in $\hat{T}$, then as in the proofs of Lemma~\ref{LflareD2} and Corollary~\ref{Cdisk}, we can extend $T_\gamma$ to a subsurface $A_\gamma\subset\hat{T}$ that contains $T_\gamma$ and is not $B_T$-isotopic to a subsurface of $\Delta_\gamma$. This gives a flare based at $\Delta_\gamma$ and contradicts our assumption at the beginning that no subdisk of $\Delta$ is the base of a flare.  Therefore, after some $B_T$-isotopies on $\hat{T}$, we may assume $\Delta\cap\hat{T}=\emptyset$.

Similarly, since $\alpha_i$ is $B_T$-isotopic to $\partial\Delta=\alpha$, there is a small collar annulus $A_i$ of $\alpha_i$ in $\hat{T}$ such that $\alpha_i$ is a component of $\partial A_i$ and $A_i$ is $B_T$-isotopic to an annular neighborhood of $\partial\Delta$ in $\Delta$.  As above, either $\alpha_i$ bounds a disk in $\hat{T}$ which is $B_T$-isotopic to $\Delta$, or we can extend $A_i$ in $\hat{T}$ to get a flare based at a subdisk of $\Delta$, which contradicts our assumption on $\Delta$.  Hence the claim holds
\end{proof}

Now we consider the $\beta_i$'s.  Although $\beta_i$ may not bound a disk carried by $N(B_T)$, there is a curve $\beta_i'\subset\hat{T}$ parallel and close to $\beta_i$ in $\hat{T}$ such that $\beta_i'$ bounds a disk carried by $N(B_T)$ and containing the disk $\Delta$ above.  To see this, let $C_i$ be a small annular neighborhood of $\beta_i$ in $\hat{T}$.  Let $P_\beta$ be the component of $\partial_hN(B_T)$ that contains $\beta$, then one boundary circle of $C_i$ is always $B_T$-parallel to a circle in $P_\beta$ (which is parallel to $\beta$ in $P_\beta$), and we choose $\beta_i'$ to be the other boundary circle of $C_i$.  No matter how small $C_i$ is, $\beta_i'\cup\alpha$ always bounds an annulus $A_i'$ transverse to the $I$-fibers ($A_i'$ can be obtained by slightly pushing the vertical annulus $Y_i$ between $\alpha$ and $\beta_i$ so that $\beta_i$ is pushed to $\beta_i'$ while $\alpha$ is fixed.  There is only one direction to tilt $Y_i$ since $Y_i$ contains the component $V$ of $\partial_vN(B_T)$).    As $\alpha$ bounds the disk $\Delta$, $A_i'\cup\Delta$ is a disk bounded by $\beta_i'$ and carried by $N(B_T)$.

We may choose all the $\beta_i'\subset\hat{T}$ ($i=1,\dots, k$) to be $B_T$-isotopic and suppose they lie in a vertical annulus $V_\beta\subset N(B_T)\subset N(B)$.  Now we consider the surface $H$ in the lemma ($S=H+\hat{T}$).  After some $B$-isotopy on $H$, we may assume each arc of $H\cap V_\beta$ is essential in $V_\beta$.  Moreover, we claim that after $B$-isotopy on $H$ if necessary, $H\cap V_\beta$ is balanced.  To see this, we consider the disk $\Delta_i'=A_i'\cup\Delta$ bounded by $\beta_i'$ above. 
  We can fix a direction along $\beta_i'$ and define positive and negative intersection points as in Definition~\ref{Dbalanced}.  Since both $H$ and $\Delta_i'$ are transverse to the $I$-fibers, the two endpoints of any arc in $H\cap\Delta_i'$ have opposite signs.  This implies that $H\cap\beta_i'$ is balanced.  
We may view $\beta_i'$ as a core curve of $V_\beta$.  Hence $H\cap V_\beta$ is balanced.

Let $h=\frac{1}{2}|H\cap V_\beta|$.  We may view $\beta_1',\dots,\beta_k'$ as $k$ copies of the core curve of $V_\beta$ and assume $\hat{T}\cap V_\beta=\cup_{i=1}^k\beta_i'$, so by Lemma~\ref{Lbalanced}, if $k>h$, $S\cap V_\beta=(H+\hat{T})\cap V_\beta$ contains at least $k-h$ circles and each circle is $B$-isotopic to $\beta_i'$.  Let $\gamma_1,\dots,\gamma_{k-h}$ be the $k-h$ circles in $S\cap V_\beta$.  Recall that $\beta_i'$ bounds a disk $\Delta_i'$ carried by $N(B_T)\subset N(B)$ and we may view $\Delta_i'$ as a disk bounded by a core curve of $V_\beta$.  Since $B_T$ fully carries a collection of normal tori, $B_T$ does not contain any almost normal piece and the disk $\Delta_i'$ does not contain any almost normal piece.  Thus by Lemma~\ref{Lfinite} (setting $V_\beta=A$ in Lemma~\ref{Lfinite}), if $k-h$ is sufficiently large, then the $\gamma_j$'s cannot all be essential in $S$, in other words, some $\gamma_j$ must bound a disk in the surface $S$.  We denote this disk by $D_\gamma$.  
Since $\beta_j'$ is close to $\beta_j$ and $\beta$, by slightly shrinking or enlarging $D_\gamma$ in $S$, we can find a disk $D_\gamma'$ in $S$ with $\partial D_\gamma'=\gamma'$ $B$-isotopic to the curve $\beta$. 
Next we will use $\alpha$ to find another disk in $S$ that together with $D_\gamma'$ bounds a non-trivial $D^2\times I$ region for $S$.

Let $D_i$ ($i=1,\dots, k$) be the disk in $\hat{T}$ bounded by $\alpha_i$ as in the Claim.   Since these disks $D_i$'s are all $B_T$-isotopic to $\Delta$, we may view the $D_i$'s as $k$ parallel copies of the same disk and assume the intersection patterns of $H$ with the $D_i$'s are all the same.    Each component of $H\cap D_i$ is either a trivial circle in $D_i$ or an arc properly embedded in $D_i$.  The intersection of $H$ with a small vertical annulus in $N(B)$ that contains $\partial D_i$ is always balanced, since the two endpoints of each arc in $H\cap D_i$ have opposite signs.  Let $h'$ be the number of components of $H\cap D_i$.  
As illustrated in Figure~\ref{balance}(b) and similar to the arguments in Lemma~\ref{Lbalanced} and Lemma~\ref{LbalancedD}, if $k>h'$, there are at least $k-h'$ disjoint disks in $S=H+\hat{T}$ that are $B$-isotopic to $D_i$.  By assuming $k>h'$, we know that there is a disk $D_a\subset S$ $B$-isotopic to $D_i$ and $\Delta$.  

By our construction, $\partial D_a\cup\partial D_\gamma'$ bounds a vertical annulus $A'\subset N(B)$ that contains the component $V$ of $\partial_vN(B_T)$.  We may view $N(B_T)\subset N(B)$.  By viewing $B_T$ as a branched surface obtained by deleting some branch sectors from $B$, we see that part of $V$ (now viewed as a vertical annulus in $N(B)$) is incident to a component of $M-\Int(N(B))$.
So if a core curve of $\Int(V)$ bounds a disk $D_V$ carried by $N(B)$ and $B$-parallel to a subdisk of $S$, then we can obtain a possibly immersed normal or an almost normal 2-sphere by capping off $\partial\Delta$ using a disk $B$-parallel to $D_V$, a contradiction to Lemma~\ref{LimmersedS2}.  Thus by Lemma~\ref{Last} and its proof, $D_a$ and $D_\gamma'$ are non-nested in $S$ and hence $D_a\cup A'\cup D_\gamma'$ bounds a $D^2\times I$ region $E'$ for $S$ and $B$.    Since $V\subset A'$ is incident to a component of $M-\Int(N(B))$ and $D_a$ is $B$-isotopic to the disk $\Delta$ bounded by $\alpha$, $E'$ must contain a component of $M-N(B)$.  So $E'$ is a non-trivial  $D^2\times I$ region for $S$ and $B$.

Next we show that the number $k$ can be algorithmically determined.  First, by solving branch equations, similar to \cite{AL}, we can check each component of $\partial_vN(B_T)$ to find a disk $\Delta$ that satisfies properties (1) and (2) at the beginning of the proof.  Moreover, using the argument in Lemma~\ref{Linnermost}, we can find a subdisk of $\Delta$ corresponding to an innermost flare.  Thus we can algorithmically find a disk $\Delta$ in the proof above.  In the proof above, since $\alpha\cup\beta$ bounds a component of $\partial_vN(B_T)$, the numbers $h$ and $h'$ above depend only on $H\cap\Delta$ and the intersection pattern of $H$ with the normal tori $T_i$'s near the component $V$ of $\partial_vN(B_T)$.  Moreover, the constant in Lemma~\ref{Lfinite} depends only on $B$, $M$ and $K$.  Thus we can algorithmically find a number $k$ which depends on $B$, $M$, $K$, $H\cap\Delta$ and the intersection pattern of $H$ and the $T_i$'s near $V$, such that if every coefficient $n_i$ in $S=H+\sum_{i=0}^nn_iT_i$ is larger than $k$, there must be a non-trivial $D^2\times I$ region for $S$ and $B$.
\end{proof}

\section{Intersection of normal tori}\label{Sinter}

\begin{notation}\label{N1}
Let $B$ be a branched surface as in Notation~\ref{NB}. 
In this section, we fix a set of normal tori $\mathcal{T}=\{T_1,\dots, T_n\}$ carried by $B$.  Let $B_T$ be the sub-branched surface of $B$ fully carrying $\bigcup_{i=1}^nT_i$.  Suppose $B_T$ does not contain any flare.  Note that, using the complexity defined after Proposition~\ref{Pproduct} and after some trivial isotopies, we may suppose no pair of tori form any trivial product region and by Corollary~\ref{Cdisk}, every curve of $T_i\cap T_j$ must be essential in both $T_i$ and $T_j$.    
We say a torus $T$ can be generated by the set of tori $\mathcal{T}$ if $T$ is a component of $F=\sum_{i=1}^nn_iT_i$ for some non-negative integers $n_i$'s, where $T_i\in\mathcal{T}$.  In this paper we use $\mathcal{G}(\mathcal{T})$ to denote the set of tori that can be generated by $\mathcal{T}$.  Clearly if $\mathcal{T}$ consists of disjoint tori, then $\mathcal{G}(\mathcal{T})=\mathcal{T}$.  Note that  $\mathcal{G}(\mathcal{T})$ is not the same as the solution space $\mathcal{S}(\mathcal{T})=\{\sum_{i=1}^nn_iT_i\}$ mentioned in section~\ref{Spre}, as every surface in $\mathcal{G}(\mathcal{T})$ is connected.  Each component of a surface in $\mathcal{S}(\mathcal{T})$ is a surface in $\mathcal{G}(\mathcal{T})$.  Moreover, since $B$ does not carry any normal 2-sphere, every surface in $\mathcal{G}(\mathcal{T})$ is a normal torus carried by $B$.  So by Lemma~\ref{Ltorus}, every torus in $\mathcal{G}(\mathcal{T})$ bounds a solid torus in $M$.
\end{notation}
 
%Each surface in $\mathcal{S}(\mathcal{T})$ consists of parallel copies of some disjoint subset of tori in $\mathcal{G}(\mathcal{T})$.  In general $\mathcal{G}(\mathcal{T})$ may be an infinite set, but if $\mathcal{G}(\mathcal{T})$ is a finite set, then by enumerating all possible disjoint subsets of tori in $\mathcal{G}(\mathcal{T})$, one can have a concrete understanding of $\mathcal{S}(\mathcal{T})$.  

\begin{lemma}\label{Lproduct}
Let $T_1$ and $T_2$ be tori in the set $\mathcal{T}$ and let $B_T$ be the branched surface as above.  Suppose $B_T$ does not contain any flare.  Suppose $F_i$ ($i=1,2$) is an annulus in $T_i$ with $\partial F_1=\partial F_2\subset T_1\cap T_2$, and suppose $F_1\cup F_2$ bounds a (pinched) product region $X$, see Definition~\ref{Dproduct}.  We view the solid torus $X$ as $bigon\times S^1$.  Let $D$ be a meridional disk of $X$.  Suppose $D\subset N(B_T)$ and $D$ is either vertical in $N(B_T)$ or transverse to the $I$-fibers of $N(B_T)$.  Then $X$ must be a trivial product region in $N(B_T)$.
\end{lemma}
\begin{proof}
The disk $D$ cuts the solid torus $X$ into a 3-ball $X_D$.  If $D$ is vertical in $N(B_T)$, a slight perturbation on $D$ can changed $D$ into a disk transverse to the $I$-fibers of $N(B_T)$.  Thus we may assume next that $D$ is transverse to the $I$-fibers of $N(B_T)$.

Since $F_1$, $F_2$ and $D$ are all carried by $N(B_T)$, similar to Figure~\ref{deform}(a), we can deform $F_1\cup F_2\cup D$ into (part of a) branched surface. By Definition~\ref{Dproduct}, $X$ is deformed into a $bigon\times S^1$ region and the 3-ball $X_D$ is naturally deformed into a $D^2\times I$ region (with its vertical boundary annulus pinched into a cusp circle) for a surface obtained by cutting and pasting of copies of $F_1$, $F_2$ and $D$.  Since $B_T$ has no flare, by Lemma~\ref{LflareD2}, $X_D$ deforms into a trivial $D^2\times I$ region in $N(B_T)$.  This means that $X\subset N(B_T)$ and $X$ must be a trivial product region.
\end{proof}

\begin{definition}\label{Dgood}
Let $T$ be a normal torus carried by $B$ and $\hat{T}$ the solid torus in $M$ bounded by $T$. 
Let $A\subset N(B)$ be a vertical annulus properly embedded in either $\hat{T}$ or $M-\Int(\hat{T})$.  Suppose $A$ is isotopic relative to $\partial A$ to a subannulus $A_T$ of $T$, i.e. $A\cup A_T$ bounds a solid torus $X$ and a meridional curve of $\partial X$ consists of a vertical arc of $A$ and an essential arc of $A_T$.  We call $X$ a $monogon\times S^1$ region.  Note that if we collapse (using $\pi:N(B)\to B$) the vertical annulus $A$ into a circle, then the meridional disk of $X$ becomes a monogon.  If there exists such an annulus $A$, then we say $T$ bounds a $monogon\times S^1$ region. 
If such an annulus $A$ lies in $\hat{T}$ and $\partial A$ is a pair of essential and non-meridional curves in $T$, then we say $T$ is a \textbf{good torus}.
\end{definition}

\begin{lemma}\label{Lmono}
Let $\mathcal{T}$  be a set of normal tori carried by $N(B)$.  Suppose the intersection of the tori in $\mathcal{T}$ has no triple point and each double curve in the intersection is essential in both corresponding tori.  Then all but finitely many tori in $\mathcal{G}(\mathcal{T})$ are good tori.  
\end{lemma}
\begin{proof}
As in Figure~\ref{deform}(a), the tori in $\mathcal{T}$ naturally deform into a branched surface $B^T$.  Since the intersection of the tori has no triple point, each double curve of the tori in $\mathcal{T}$ corresponds to a thin annular branch sector of $B^T$ and the branch locus of $B^T$ consists of disjoint curves.  So each branch sector of $B^T$ is an annulus and every torus in $\mathcal{G}(\mathcal{T})$ is carried by $B^T$.

Next we show that all but finitely many tori carried by $B^T$ are good tori. 
Let $A$ be an annulus branched sector or $B^T$ and let $T$  be a torus carried by $B^T$.  Clearly $T$ is a normal torus and by Lemma~\ref{Ltorus} $T$ bounds a solid torus $\hat{T}$ in $M$. 

Let $t$ be the weight of $T$ at the branch sector $A$, i.e. $t=|T\cap\pi^{-1}(\Int(A))|$ where $\pi:N(B^T)\to B^T$ is the collapsing map.  
Let $c$ be a core curve of $A$.  If $t\ge 3$,  the intersection of $T$ and the vertical annulus $V=\pi^{-1}(c)$ of $N(B^T)$ contains $t\ge 3$ curves.  By our hypotheses on $\mathcal{T}$ and construction of $B^T$, each curve in $T\cap V$ is an essential curve in $T$.  If $t\ge 3$, 3 consecutive curves in $T\cap V$ give two subannuli $A_1$ and $A_2$ of $V$ properly embedded in $\hat{T}$ and $M-\Int(\hat{T})$ respectively.  After a small isotopy, we may assume $\partial A_1$ and $\partial A_2$ are disjoint in $T$.  As $\partial A_1$ is essential in $T$, the annulus $A_1$ is $\partial$-parallel in the solid torus $\hat{T}$.  
So the vertical annulus $A_1$ and a subannulus of $T$ bound a $monogon\times S^1$ region in $\hat{T}$.  By the definition of good torus, either $T$ is a good torus or $\partial A_i$ consists meridional curves of $T$.  

\vspace{10pt}
\noindent
\emph{Claim}. If $\partial A_i$ consists of meridional curves of $T$, then $A_2$ and a subannulus of $T$ bound a $monogon\times S^1$ region outside the solid torus $\hat{T}$.  
\begin{proof}[Proof of the Claim]
The claim is implicitly proved in \cite[Section 5]{L4}. 
Let $C_1$ and $C_2$ be the two annuli in $T$ with $\partial C_1=\partial C_2=\partial A_2$ and $C_1\cup C_2= T$.  Let $X_1=A_2\cup C_1$ and $X_2= A_2\cup C_2$.  The tori $X_1$ and $X_2$ are not normal tori, but by \cite[Theorem 3.2]{JR} $T\cup A_2$ is a barrier for the normalization process.  So either (1) $X_i$ can be normalized (in $M-(\hat{T}\cup A_2)$) into a normal torus isotopic to $X_i$ or (2) similar to the proofs of Lemma~\ref{LnoS2} and Corollary~\ref{Cklein}, some compression occurs and $X_i$ becomes a 2-sphere during the normalization process, in which case $X_i$ bounds a solid torus outside $\hat{T}\cup A_2$.  Since every normal torus bounds a solid torus, in either case, $X_i$ bounds a solid torus in $M$, which we denote by $\hat{X}_i$.  

By the hypothesis in the claim, $\partial A_i$ consists of meridional curves of $T$. 
We first consider the case that the solid torus $\hat{X}_i$ lies outside $\hat{T}$.   In this case, if the intersection number of a meridional curve of $\partial \hat{X}_i$ and a component of $\partial A_2$ is one, then $\hat{X}_i$ is a $monogon\times S^1$ region and the claim holds.  Otherwise, the union of $\hat{X}_i$ and a neighborhood of a meridional disk of $\hat{T}$ bounded by a component of $\partial A_2$ is a non-trivial punctured lens space, which contradicts our hypothesis that $M$ is irreducible and is not a lens space.    

The remaining case is that both $\hat{X}_1$ and $\hat{X}_2$ contain $\hat{T}$.  By our discussion on normalizing $X_i$ above, this happens only if $X_i$ is parallel to a normal torus.  By Lemma~\ref{Ltorus}, this means that $X_i=\partial\hat{X}_i$ is incompressible in $M-\Int(\hat{X}_i)$.  Since $\partial A_i$ consists of meridional curves of $T$, a core curve of $C_1$ bounds a meridional disk $D_1$ of $\hat{T}$.  So $D_1\cap X_2=\emptyset$.  Since $X_1=\partial\hat{X}_1$ is incompressible in $M-\Int(\hat{X}_1)$ and since $\hat{T}\subset\hat{X}_1$, $D_1$ must be a meridional disk for $\hat{X}_1$.  After compressing $\hat{X}_1$ along $D_1$, we obtain a 3-ball $E\subset\hat{X}_1$.  Since $X_2\subset\hat{T}\cup A_2\subset\hat{X}_1$ and $D_1\cap X_2=\emptyset$, the torus $X_2$ lies in the 3-ball $E$.  We may suppose $X_2\subset\Int(E)$.  Hence $X_2$ is compressible in $E$.  Since the solid torus $\hat{X}_2$ bounded by $X_2$ contains $\hat{T}$ and $D_1$, $X_2$ does not bound a solid torus in the 3-ball $E$.  As $X_2$ is incompressible outside the solid torus $\hat{X}_2$, this implies that the submanifold of $E$ bounded by $X_2\cup\partial E$ is a ball with a knotted hole.  Since $X_2$ bounds a solid torus in $M$, this means that $M$ must be $S^3$, a contradiction to our hypothesis on $M$.
\end{proof}

Suppose $T$ is not a good torus, then as in the discussion before the claim, $\partial A_i$ must be meridional curves of $T$ and 
by the Claim above, $A_2$ and a subannulus of $T$ bound a $monogon\times S^1$ region outside $\hat{T}$. 
Let $P_i\subset T$ ($i=1,2$) be the subannulus of $T$ such that $\partial P_i=\partial A_i$ and $P_i\cup A_i$ bounds a $monogon\times S^1$ region.  Since $A_1$ and $A_2$ are assumed to be disjoint, $\partial P_1$ and $\partial P_2$ are disjoint in $T$.  We assign a normal direction for each component of $\partial P_i$ in $T$ which points out of $P_i$.  It is easy to see that, for any configuration of $\partial A_i$ in $T$, there is always an annulus $Q\subset T$ (note that $\Int(Q)$ is allowed to contain curves in $\partial P_i$) such that one component of $\partial Q$ is a curve in $\partial P_1$, the other component of $\partial Q$ is a curve in $\partial P_2$, and the normal directions defined above at both curves of $\partial Q$ point into $Q$, see Figure~\ref{monogon}(a) for a 1-dimensional schematic picture.  Now, as shown in Figure~\ref{monogon}(b), we can use a copy of $P_1$, a copy of $P_2$ and 2 parallel copies of $Q$ to form a normal torus $T_Q$ carried by $B$.  Moreover, as shown in Figure~\ref{monogon}(b), $T_Q$ bounds a solid torus which is the union of the two $monogon\times S^1$ regions (bounded by copies of $P_1\cup A_1$ and $P_2\cup A_2$) and a product neighborhood of $Q$.  A meridional disk of this solid torus is formed by the union of two monogons and hence the meridian for $T_Q$ is not $\partial P_i$.  By Corollary~\ref{Cmeridion}, this contradicts that $\partial P_i$ ($\partial P_i=\partial A_i$) bounds embedded disks in $\hat{T}$.  Thus curves of $\partial A_i$ are not meridional curves of $T$ and $T$ is a good torus.

\begin{figure}
\begin{center}
\psfrag{(a)}{(a)}
\psfrag{(b)}{(b)}
\psfrag{A1}{$A_2$}
\psfrag{A2}{$A_1$}
\psfrag{mono}{monogon}
\psfrag{Q}{$Q$}
\psfrag{solid}{solid torus}
\includegraphics[width=4.5in]{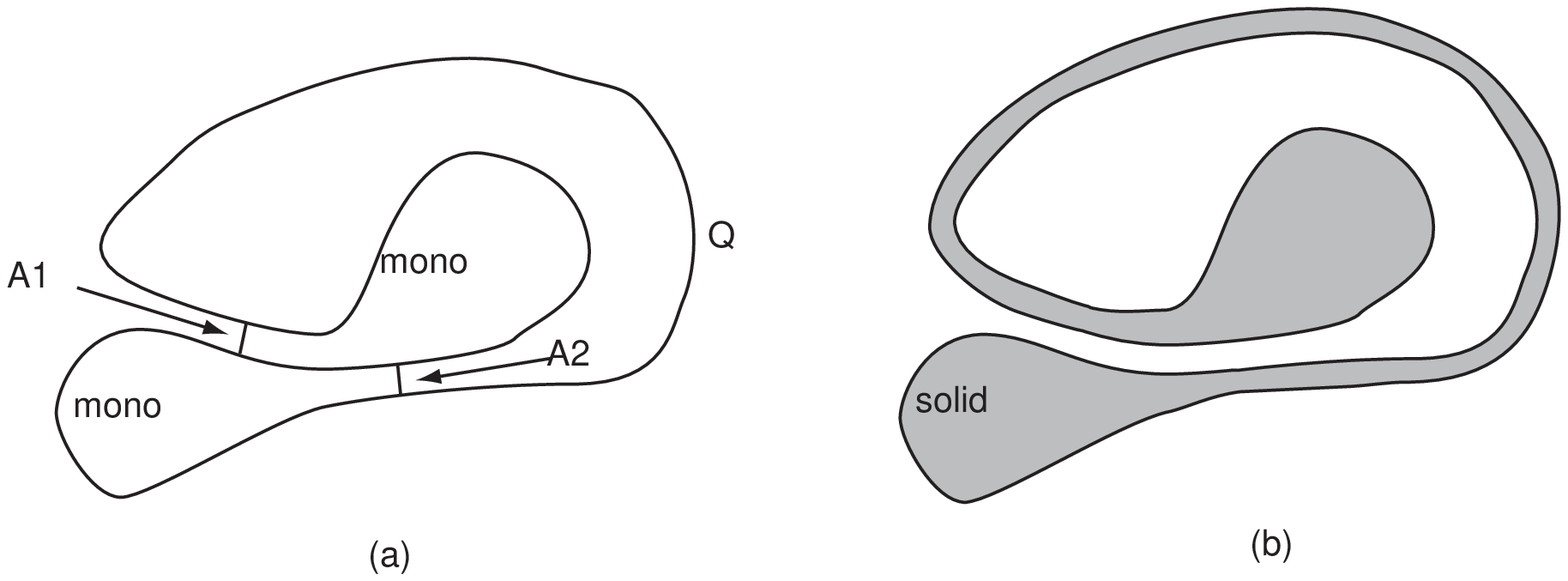}
\caption{}\label{monogon}
\end{center}
\end{figure}

The argument above says that if a torus $T$ carried by $B^T$ is not a good torus, then its weight at each annular branch sector $A$ above (of $B^T$) is at most 2. 
The number of tori carried by $B^T$ and with weight at most 2 at every branch sector is clearly finite, and it is trivial to algorithmically enumerate all such tori.  
\end{proof}

For any two surfaces $T$ and $F$ carried by $N(B)$, we say $T$ non-trivially intersects $F$ if $T\cap F\ne\emptyset$ after any $B$-isotopy.

\begin{lemma}\label{Lgoodone}
Let $\mathcal{T}'$ be any finite set of normal tori carried by $B$.  Suppose the intersection of the tori in $\mathcal{T}'$ contains no triple point and each double curve is essential in both corresponding tori.   Then there is an algorithm that either 
\begin{enumerate}
\item finds a good torus $T$ carried by $B$ such that $T$ non-trivially intersects at least one torus in $\mathcal{T}'$ and the intersection of the tori in $T\cup\mathcal{T}'$ has no triple point, or 
\item lists all possible tori in $\mathcal{G}(\mathcal{T}')$, in which case $\mathcal{G}(\mathcal{T}')$ is a finite set.
\end{enumerate}
\end{lemma}
\begin{proof}
After some trivial isotopies, we may assume there is no trivial product region (see Definition~\ref{Dproduct}) between any two tori in $\mathcal{T}'$. 
As shown in Figure~\ref{deform}(a), the union of the tori in $\mathcal{T}'$ naturally deforms into a branched surface $B^T$.  Since each torus in $\mathcal{T}'$ is carried by $N(B)$, we may view $B^T\subset N(B^T)\subset N(B)$.  
By the proof of Lemma~\ref{Lmono}, given any torus carried by $B^T$, if its weight at an annular branch sector of $B^T$ is at least 3, then it must be a good torus. 

Suppose $T$ is a torus carried by $B^T$ and suppose $T$ can be isotoped disjoint from every torus in $\mathcal{T}'$ via $B$-isotopy.  By Proposition~\ref{Pproduct}, for any torus $T_i$ in $\mathcal{T}'$, one can eliminate the double curves $T\cap T_i$ by a sequence of trivial isotopies that eliminate trivial product regions between $T$ and $T_i$.  
Suppose there is another torus $T_j$ in $\mathcal{T}'$ that is already disjoint from $T$ before these isotopies on $T$ and $T_i$.  For any trivial product region $P$ bounded by subsurfaces of $T$ and $T_i$, if $P\cap T_j\ne\emptyset$, then since $T_j\cap T=\emptyset$, a component of $T_j\cap P$ and a subsurface of $T_i$ in $\partial P$ must bound a trivial product region in $P$, which contradicts our assumption at the beginning that $T_i$ and $T_j$ do not bound any trivial product region.  This means that $T_j$ does not intersect any trivial product region bounded by $T$ and $T_i$.  Thus, after the sequence of trivial isotopies that move $T$ disjoint from $T_i$ as in Proposition~\ref{Pproduct}, $T$ remains disjoint from $T_j$.  As $\mathcal{T}'$ is a finite set, this means that after some $B$-isotopies on $T$, $T$ is disjoint from every torus in $\mathcal{T}'$.  Hence $T$ is disjoint from $B^T$ and $N(B^T)$ after some $B$-isotopies.  

By our construction, we may view $B^T\subset N(B^T)\subset N(B)$, $T\subset N(B)$ and $T\cap N(B^T)=\emptyset$.  However since $T$ is carried by $N(B^T)$ before these isotopies, there is a torus $T'\subset N(B^T)$ that represents the same torus as $T$ in $N(B)$, i.e., $T\cup T'$ bounds a product region $T^2\times I$ in $N(B)$, where each $I$-fiber of the product region is a subarc of an $I$-fiber of $N(B)$.  Note that $T'\subset N(B^T)$ but $T\cap N(B^T)=\emptyset$.  If a component of $\partial_hN(B^T)$ lies in the product region $T^2\times I$, it must be transverse to the $I$-fibers of the product $T^2\times I$.  This implies that the component of $\partial_hN(B^T)\cap(T^2\times I)$ that is closest to $T$ in the product region $T^2\times I$ must be a torus parallel to $T$. So, if $T$ can be isotoped disjoint from every torus in $\mathcal{T}'$ via $B$-isotopy, then $T$ is parallel to a torus component of $\partial_hN(B^T)$.   Since $|\partial_hN(B^T)|$ is finite, there are only finitely many tori carried by $B^T$ that can be isotoped disjoint from all the tori in $\mathcal{T}'$ via $B$-isotopy, and we can algorithmically find all such tori.

Next we inductively construct a sequence of subsets of tori carried by $B^T$.  First we set $\mathcal{T}_1=\mathcal{T}'$.  Suppose we have constructed a subset $\mathcal{T}_n$.  For every pair of tori $T_i$ and $T_j$ in $\mathcal{T}_n$ with $T_i\cap T_j\ne\emptyset$, by adding the components of $T_i+T_j$ that are not already in $\mathcal{T}_n$ to $\mathcal{T}_n$, we obtain a set of tori $\mathcal{T}_{n+1}\supset\mathcal{T}_n$.  If $\mathcal{T}_{n+1}=\mathcal{T}_n$ for some $n$, i.e., we do not get any new torus type from $T_i+T_j$ for any $T_i$ and $T_j$ in $\mathcal{T}_n$, then since $\mathcal{T}'\subset\mathcal{T}_n$, it is easy to see that $\mathcal{G}(\mathcal{T}')\subset \mathcal{T}_n$ and $\mathcal{G}(\mathcal{T}')$ is a finite set.  We can algorithmically build the sequence of sets of tori $\mathcal{T}_n$ carried by $B^T$.  

Recall that we have proved in Lemma~\ref{Lmono} that if the weight of a torus carried by $B^T$ at a branch sector of $B^T$ is at least 3, then it is a good torus.  Moreover, by the argument above, we can check whether or not a torus in $\mathcal{T}_n$ is $B$-parallel to a torus component of $\partial_hN(B^T)$ and determine whether or not it can be isotoped disjoint from every tori in $\mathcal{T}'$.  So eventually either we find a good torus $T$ that non-trivially intersects at least one torus in $\mathcal{T}'$, or $\mathcal{T}_{n+1}=\mathcal{T}_n$ for some $n$ and we obtain a complete (finite) list of tori containing $\mathcal{G}(\mathcal{T}')$.  

Furthermore, all the tori above are carried by $B^T$.  By the no-triple-point hypothesis on $\mathcal{T}'$ and by the construction of $B^T$, the branch locus of $B^T$ has no double point and for any torus $T$ carried by $B^T$, after some $B^T$-isotopy, the intersection of the tori in $T\cup\mathcal{T}'$ has no triple point.  
\end{proof}

\begin{lemma}\label{Lslope}
Let $\mathcal{T}$ and $B_T$ be as in Notation~\ref{N1}. In particular, $B_T$ has no flare. 
Let $T$ be a normal torus in $\mathcal{T}$ and let $\hat{T}$ be  the solid torus bounded by $T$.   Suppose there is a vertical annulus $A$ of $N(B_T)$ properly embedded in $\hat{T}$ with $\partial A$ essential in $T$.  Let $T'$ be any other torus carried by $N(B_T)$ that non-trivially intersects $T$.  Then every  curve in $T\cap T'$ that is essential in $T$ always has the same slope in $T$ as the slope of $\partial A$. 
\end{lemma}
\begin{proof}
As in Notation~\ref{N1}, since $B_T$ has no flare, after some $B_T$-isotopies removing double curves trivial in both tori, we may assume each curve in $T\cap T'$ is essential in both $T$ and $T'$.

As before, $A$ and a subannulus of $T$ bound a $monogon\times S^1$ region in $\hat{T}$. 
Since each curve in $T\cap T'$ is essential in both $T$ and $T'$, the lemma holds trivially if $T'\cap\partial A=\emptyset$.  Next we suppose $T'\cap\partial A\ne\emptyset$.

Since $A$ is vertical in $N(B_T)$ and $T'$ is carried by $N(B_T)$, each arc in $T'\cap A$ is transverse to the induced $I$-fibers of $A=S^1\times I$.  We first consider the case that some arc in $T'\cap A$ is $\partial$-parallel in $A$.  Suppose there is such an arc and let $\alpha\subset T'\cap A$ be an outermost $\partial$-parallel arc in $A$.  Let $\beta\subset\partial A$ be the arc parallel to $\alpha$ and with $\partial\beta=\partial\alpha$.  Let $D_\alpha\subset A$ be the bigon bounded by $\alpha\cup\beta$ in $A$.  The intersection of $T'$ and the solid torus $\hat{T}$ is a collection of annuli.  Let $A_\alpha$ be the annular component of $T'\cap\hat{T}$ that contains $\alpha$.  We first consider the case that $\alpha$ is also a $\partial$-parallel arc in $A_\alpha$. In this case, $\alpha$ and a subarc of $\partial A_\alpha$ bound a subdisk $D_A$ of $A_\alpha$.  After a slight perturbation (or pinching $D_\alpha$ to $\beta$), $D_A\cup D_\alpha$ becomes a disk transverse to the $I$-fibers of $N(B_T)$.  To simplify notation, we still use $D_A\cup D_\alpha$ to denote the disk after the perturbation.  Note that $D_A\cup D_\alpha$ is properly embedded in the solid torus $\hat{T}$.  If the disk $D_A\cup D_\alpha$ is an essential disk in $\hat{T}$, then similar to the proofs of Lemma~\ref{LflareD2} and Corollary~\ref{Cdisk}, a subsurface of $T$ is a flare based at $D_A\cup D_\alpha$, a contradiction to our hypothesis on $B_T$.  So $D_A\cup D_\alpha$ must be a $\partial$-parallel disk in $\hat{T}$. Moreover, as in Lemma~\ref{LflareD2}, the no-flare hypothesis also implies that $D_A\cup D_\alpha$ and the subdisk of $T$ bounded by $\partial(D_A\cup D_\alpha)$ form a trivial $D^2\times I$ region.  Hence we can perform a $B_T$-isotopy on $T'$ to eliminate $\alpha$.  Now we suppose the arc $\alpha$ is  essential in $A_\alpha$ (recall that $\alpha$ is still $\partial$-parallel in $A$). In this case, let $\Sigma_\beta$ be the annulus in $T$ bounded by $\partial A_\alpha$ and containing $\beta$, then $A_\alpha\cup \Sigma_\beta$ bounds a $bigon\times S^1$ region.  The bigon $D_\alpha\subset A$ is a meridional disk for this $bigon\times S^1$ region and $D_\alpha$ is vertical in $N(B_T)$, so by Lemma~\ref{Lproduct}, the $bigon\times S^1$ region is a trivial product region.  Thus a trivial isotopy can remove the two double curves (i.e. $\partial A_\alpha$) in the $bigon\times S^1$ region.   Thus after some $B_T$-isotopies as above, we may assume every arc of $T'\cap A$ is essential in $A$.

Let $A'$ be an annular component of $T'\cap\hat{T}$.  By the assumption above, $A'\cap A$ is a collection of arcs essential in $A$.   Let $\gamma$ be a component of $A'\cap A$.  If we deform $T'\cup T$ into a branched surface $B^{T'}$, $\partial\gamma$ corresponds to two points at the cusp (i.e. branch locus) of the branched surface.  Since $A$ is vertical in $N(B_T)$ and $A'$ is transverse to the $I$-fibers of $N(B_T)$, as shown in Figure~\ref{cusp}(a), the branch directions of $B^{T'}$ at the two endpoints of $\gamma$ must be opposite with respect to $T'$ (one points into the solid torus bounded by $T'$ the other points out).  Thus $A'$ cuts $A$ into a collection of quadrilaterals and after we deform $T'\cup T$ into a branched surface, each quadrilateral is deformed into a bigon.

\begin{figure}
\begin{center}
\psfrag{(a)}{(a)}
\psfrag{(b)}{(b)}
\psfrag{A}{$A$}
\psfrag{g}{$\gamma$}
\psfrag{g'}{$\gamma'$}
\psfrag{bd}{branch direction}
\psfrag{dg}{$D_\gamma$}
\includegraphics[width=4.0in]{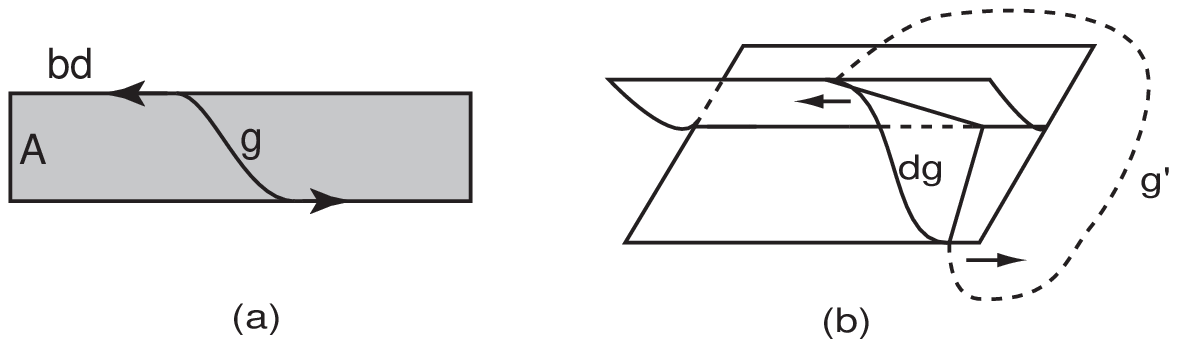}
\caption{}\label{cusp}
\end{center}
\end{figure}

Next we show that each arc in $A'\cap A$ must also be an essential arc in $A'$.  Otherwise, a component $\gamma$ of $A'\cap A$ is an outermost $\partial$-parallel arc in $A'$.  Let $\gamma'$ be the arc in $\partial A'$ such that $\partial\gamma'=\partial\gamma$ and $\gamma\cup\gamma'$ bounds a disk $D_\gamma$ in $A'$.  Note that $\partial A$ cuts $T$ into two annuli and let $\Gamma\subset T$ be the annulus that contains $\gamma'$.  Since $\gamma$ is essential in $A$, the two endpoints in $\partial\gamma=\partial\gamma'$ lie in different components of $\partial A=\partial\Gamma$, which implies that $\gamma'$ must be an essential arc of $\Gamma$.  However, since $A$ is vertical in $N(B_T)$ and $\gamma$ is an essential arc of $A$, if we collapse $A$ into a cusp circle (like the projection $\pi:N(B_T)\to B_T$) and shrink $\gamma$ to a point (along the cusp circle), then $D_\gamma$ becomes a monogon properly embedded in the region bounded by $A\cup\Gamma$.  In fact, since $D_\gamma\subset A'\subset T'$ is carried by $B_T$, we can first collapse $A$ in to a cusp circle, which deforms $T$ into a branched surface, then since $D_\gamma\subset T'$ and as shown in Figure~\ref{cusp}(b), we can add in $D_\gamma$ (as a branch sector) and naturally deform $T\cup D_\gamma$ into a branched surface transverse to the $I$-fibers of $N(B_T)$.   This means that, as a branch sector, $D_\gamma$ is a monogon.  However, as illustrated in Figure~\ref{cusp}(b), such a branched surface cannot have a monogon branch sector because the induced branch directions at $\partial\gamma'$ ($\partial\gamma'=\partial\gamma$) are not compatible along $\gamma'$.  This means that no such disk $D_\gamma$ exists and hence each arc of $A'\cap A$ is also essential in $A'$.

Since $A'$ is a properly embedded annulus in $\hat{T}$ with $\partial A'$ essential in $T$, $A'$ is $\partial$-parallel in $\hat{T}$.  So $\partial A'$ bounds a subannulus $\Sigma_T$ of $T$ such that $A'$ is parallel to $\Sigma_T$ in $\hat{T}$.  So $A'\cup\Sigma_T$ bounds a solid torus $X$ in $\hat{T}$ and there is a meridional disk $\Delta_X$ of $X$ whose boundary $\partial\Delta_X$ consists of an essential arc of $A'$ and an essential arc of $\Sigma_T$.   By our conclusion on the induced branch direction at $\partial\gamma$ (see Figure~\ref{cusp}(a)), if we deform $T\cup A'$ into a branched surface as in Figure~\ref{deform}(a), $\Delta_X$ becomes a monogon and $X$ becomes a solid torus of the form $mongon\times S^1$. 

Recall that the arcs in $A'\cap A$ cut $A$ into a collection of quadrilaterals.  So $A\cap X$ is a collection of such quadrilaterals.  Each quadrilateral becomes a bigon after we deform $T\cup A'$ into a branched surface as above. Since the meridional disk $\Delta_X$ deforms into a monogon and since each quadrilateral of $A\cap X$ deforms into a bigon, every component of $A\cap X$ must be a $\partial$-parallel disk in the solid torus $X$.  For any quadrilateral $Q$ of $A\cap X$, two edges in $\partial Q$ are arcs in $A'\cap A$ and the other two edges of $\partial Q$ are arcs properly embedded in the annulus $\Sigma_T$.  Since $\partial Q$ is trivial in $\partial X$, for every component $Q$ of $A\cap X$, the two arcs of $\partial Q\cap\Sigma_T$ must be $\partial$-parallel arcs in $\Sigma_T$.  This means that one can find a core curve of $\Sigma_T$ that is disjoint from $\partial A$.  Hence the slope of $\partial A$ in $T$ is the same as the slope of $\partial\Sigma_T=\partial A'$ and $T\cap T'$.
\end{proof}

\begin{definition}\label{Dfootball}
Let $T_1$, $T_2$ and $T_3$ be normal tori carried by $N(B)$ and in general position.  Let $p$ and $q$ be triple points in the intersection.  Suppose there are arcs $\alpha_1\subset T_1\cap T_2$, $\alpha_2\subset T_2\cap T_3$ and $\alpha_3\subset T_3\cap T_1$ such that $\partial\alpha_i=p\cup q$ for each $i=1,2,3$, $\alpha_3\cup\alpha_1$ (respectively $\alpha_1\cup\alpha_2$ and $\alpha_2\cup\alpha_3$)  is an embedded trivial circle bounding a disk $D_1$ in $T_1$ (respectively $D_2$ in $T_2$ and $D_3$ in $T_3$).  Suppose $D_1\cup D_2\cup D_3$ is a 2-sphere bounding a 3-ball $X$ in $M$.   Then we say that $X$ is a \textbf{football region}.  In short, $X$ is bounded by 3 bigon disks from the 3 tori.  Note that since $B$ does not carry any normal 2-sphere, if we deform $T_1\cup T_2\cup T_3$ into a branched surface as in Figure~\ref{deform}(a), then $X$ is naturally deformed into a $D^2\times I$ region (with its vertical boundary annulus pinched into a cusp circle) and the cusp circle is $\partial D_i$ for some $i$. 
\end{definition}

\begin{definition}\label{Dregular}
We say a set of normal tori carried by the branched surface $B$ is \textbf{regular} if (1) the intersection of these tori contains no triple point, and (2) the intersection curves are essential and non-meridional curves in the corresponding tori.
\end{definition}

Suppose we have a regular set of normal tori.  The next lemma says sometimes we can add in another torus to enlarge the regular set.

\begin{lemma}\label{Ltriple}
Let $\Sigma$ be a regular set of normal tori carried by $B$ and suppose the union of the tori in $\Sigma$ is a connected 2-complex after any $B$-isotopy that preserves $\Sigma$ as a regular set of tori.  Let $\Gamma$ be another normal torus carried by $B$.  Let $B_T$ be the sub-branched surface of $B$ that fully carries $\Sigma\cup\Gamma$ and suppose $B_T$ does not contain any flare.  Suppose there is a special torus $T\in\Sigma$ such that after any $B_T$-isotopy on $\Gamma$, curves in $T\cap\Gamma$ that are non-trivial in $T$ always have the same slope in $T$ as the intersection curves of $T$ with other tori in $\Sigma$.   
Then after $B_T$-isotopy, $\Sigma\cup\Gamma$ is a regular set of normal tori.
\end{lemma}
\begin{proof} Before we proceed, we would like to remark that in practice, we assume $T$ to be a good torus, which (by Lemma~\ref{Lslope}) guarantees that the conditions on $T$ are satisfied.  Moreover, the special torus $T$ is allowed to be disjoint from $\Gamma$, in which case the lemma holds trivially.

We may assume the intersection of the tori in $\Sigma\cup\Gamma$ is minimal in the sense that the number of triple points is minimal under the conditions that $\Sigma$ is a regular set. 
First note that by this minimality assumption, we may assume
\begin{enumerate}
\item every double curve in the intersection of the tori in $\Sigma\cup\Gamma$ is essential in both corresponding tori, and 
\item  the tori in $\Sigma$ do not form any trivial $bigon\times S^1$ product region.  
\end{enumerate}
To see this, suppose there is a double curve $c$ in $T_i\cap\Gamma$ ($T_i\in\Sigma$) that is trivial in both $T_i$ and $\Gamma$ and let $d_1$ and $d_2$ be the two disks bounded by $c$ in $T_i$ and $\Gamma$ respectively.  By choosing $c$ to be innermost in $T_i$, we may assume $d_1\cap d_2=c$.  As $B$ does not carry any normal $S^2$, $d_1\cup d_2$ forms a $D^2\times I$ region (with its vertical boundary pinched into a cusp circle).  Since $B_T$ has no flare and by Lemma~\ref{LflareD2}, this $D^2\times I$ region is trivial.  Hence we can eliminate the double curve $c$ by performing a trivial isotopy on $\Gamma$ pushing $d_2\subset\Gamma$ across this $D^2\times I$ region.  As $\Sigma$ is a regular set and since $c$ is an innermost trivial curve in $T_i$, $\Int(d_1)$ contains no triple point.  The trivial isotopy on $\Gamma$ above can be viewed as replacing $d_2$ by a parallel copy of $d_1$.  As $\Int(d_1)$ contains no triple points, this means that the number of triple points is not increased by this trivial isotopy.  Moreover, $\Sigma$ is fixed by the isotopy, so the intersection remains minimal in the above sense.  Similar to Corollary~\ref{Cdisk}, the no-flare hypothesis also implies that no double curve in $T_i\cap\Gamma$ is trivial in one torus but essential in the other.  Thus after some trivial isotopy removing trivial double curves, we may assume every double curve in the intersection of the tori in $\Sigma\cup\Gamma$ is essential in both corresponding tori.  Similarly, if two tori in the regular set $\Sigma$ form a trivial $bigon\times S^1$ product region, then as in the discussion after Proposition~\ref{Pproduct}, we can always use a trivial isotopy to eliminate the pair of double curves in an innermost such $bigon\times S^1$ region without increasing the number of triple points.  Note that $\Sigma$ remains a regular set after eliminating the trivial $bigon\times S^1$ region, so the intersection remains minimal in the above sense. Thus, after finitely many such trivial isotopies, we may also assume the tori in $\Sigma$ do not form any trivial $bigon\times S^1$ product region. 

If the intersection of the tori in $\Sigma\cup\Gamma$ has no triple point, then by the hypotheses and Corollary~\ref{Cmeridion}, $\Sigma\cup\Gamma$ is a regular set of normal tori and the lemma holds.  Our goal is to prove that the intersection has no triple point.

Since $\Sigma$ is a regular set of tori, all the triple points of the intersection of $\Sigma\cup\Gamma$ must lie in $\Gamma$. 
If the special torus $T$ does not contain any triple point, then since the union of the tori in $\Sigma$ is connected, there are a sequence of tori $T_0,\dots, T_k$ in $\Sigma$ such that (1) $T_0=T$, (2) $T_i\cap T_{i+1}\ne\emptyset$, (3) $T_k$ contains a triple point and (4) $T_j$ does not contain any triple point if $j<k$.  The curves in $T_{k-1}\cap T_k$ cut $T_k$ into a collection of annuli.  Since $T_{k-1}$ does not contain any triple point, the curves of $T_k\cap\Gamma$ must lie in the interior of these annuli, in other words, the curves of $T_k\cap T_{k-1}$ and $T_k\cap\Gamma$ are disjoint and have the same slope in $T_k$.
So we have a torus $T_k$ such that $T_k$ contains a triple point and $T_k\cap\Gamma$ has the same slope (in $T_k$) as the intersection curves of $T_k$ with other tori in $\Sigma$.  Thus after replacing $T$ by $T_k$ in the proof below if necessary, we may assume that the special torus $T$ contains a triple point.

We use $\Gamma_T$ to denote the union of the double curves of $\Sigma\cup\Gamma$ that lie in $T$.  So $\Gamma_T$ consists of $T\cap\Gamma$ and the intersection of $T$ with other tori in $\Sigma$. By our hypothesis, curves in $\Gamma_T$ have the same slope in $T$.   Next we consider the intersection pattern of curves in $\Gamma_T\subset T$.  Each double point of curves in $\Gamma_T\subset T$ corresponds to a  triple point in the intersection of the tori.   Since the intersection of the tori in $\Sigma$ contains no triple point, all the double points in $\Gamma_T$ lie in $\Gamma\cap T$.  Hence $\Gamma_T$ does not form any triangle (i.e. no 3 curves in $\Gamma_T$ intersecting each other).  This implies that we can find an innermost bigon disk $D\subset T$, such that $D\cap\Gamma_T=\partial D$ and $\partial D$ consists of two (smooth) arcs $\alpha$ and $\beta$ with $\partial\alpha=\partial\beta$ being a pair of double points of $\Gamma_T$.  
Let $T_\alpha$ and $T_\beta$ be the two tori with $\alpha\subset T\cap T_\alpha$ and $\beta\subset T\cap T_\beta$.  Since $\Sigma$ is a regular set, $\Gamma$ is either $T_\alpha$ or $T_\beta$.  Let $A$ and $Z$ be the two points in $\partial\alpha=\partial\beta$.  So $A$ and $Z$ are triple points in the intersection of the tori in $\Sigma\cup\Gamma$ and  $A\cup Z\subset T_\alpha\cap T_\beta$.  Let $\gamma_A$ and $\gamma_Z$ be the double curves of $T_\alpha\cap T_\beta$ that contain $A$ and $Z$ respectively.  

Next we study the properties of $D$ and its nearby regions.  By our construction, no torus intersects $\Int(D)$ and the torus $\Gamma$ is either $T_\alpha$ or $T_\beta$

\vspace{10pt}
\noindent
\emph{Claim 1}.  Let $D$ be the bigon disk as above.  Then $\gamma_A=\gamma_Z$.
\begin{proof}[Proof of the Claim]
Suppose $\gamma_A\ne\gamma_Z$. 
Let $X_\alpha\subset T_\alpha$ and $X_\beta\subset T_\beta$ be the two annuli that are bounded by $\gamma_A\cup\gamma_Z$ and contain $\alpha$ and $\beta$ respectively.  Since $\gamma_A\ne\gamma_Z$, $\partial D$ is an essential curve in the torus $X_\alpha\cup X_\beta$. Since $D$, $T_\alpha$ and $T_\beta$ are all transverse to the $I$-fibers of $N(B_T)$, we can deform $D\cup T_\alpha\cup T_\beta$ into (part of) a branched surface as in Figure~\ref{deform}(a) and view $D$ as a branch sector.  Now $\alpha$ and $\beta$ are viewed as part of the branch locus at the boundary of the branch sector $D$.  There are two cases.
\begin{enumerate}
\item $X_\alpha\cup X_\beta$ is a normal torus transverse to the $I$-fibers of $N(B_T)$, in other words, the two corners of $D$ at $A$ and $Z$ are smoothed out when we deform $X_\alpha\cup X_\beta\cup D$ into a branch surface and $\partial D$ becomes a smooth circle in the branched locus.  Note that $\partial D$ is an essential curve in the normal torus $X_\alpha\cup X_\beta$. 
\item After deforming $X_\alpha\cup X_\beta\cup D$ into branched surface, the two corners of $D$ at $A$ and $Z$ become cusps and $D$ becomes a bigon. 
\end{enumerate}
Note that, as in the proof of Lemma~\ref{Lslope} and illustrated in Figure~\ref{cusp}(b), there is no monogon branch sector and hence $D$ cannot be a monogon and that is why we only have the above two cases to consider. 
 The first case is impossible because, similar to the proofs of Lemma~\ref{LflareD2} and Corollary~\ref{Cdisk}, it implies that a subsurface of the normal torus $X_\alpha\cup X_\beta$ is a flare based at $D$, which contradicts the no-flare hypothesis on $B_T$.   In the second case, the torus $X_\alpha\cup X_\beta$ must bound a $bigon\times S^1$ region $\hat{X}$.  Recall that $\Gamma$ is either $T_\alpha$ or $T_\beta$.  Without loss of generality, we assume $\Gamma=T_\alpha$ in this claim and $X_\alpha\subset\Gamma$.  By Lemma~\ref{Lproduct} and since $B_T$ has no flare, the existence of $D$ implies that $\hat{X}$ is a trivial $bigon\times S^1$ region.   However, if $\hat{X}$ is a trivial $bigon\times S^1$ region, we can perform a trivial isotopy on $\Gamma$ by pushing the annulus $X_\alpha$ across $\hat{X}$ and eliminate the pair of double curves $\gamma_A$ and $\gamma_Z$.  This isotopy on $\Gamma$ can be viewed as replacing $X_\alpha$ by a parallel copy of $X_\beta$.  Since all the triple points lie in $\Gamma$, $\Int(X_\beta)$ contains no triple point and this isotopy reduces the number of triple points (in particular $A$ and $Z$ are eliminated by the isotopy).  This contradicts our assumption that the intersection of $\Sigma\cup\Gamma$ is minimal.  Thus $\gamma_A=\gamma_Z$.
\end{proof}

Since both $T_\alpha$ and $T_\beta$ are separating, the double curves in $T_\alpha\cap T_\beta$ cut $T_\alpha$ into an even number of annuli. Let $Y_\alpha\subset T_\alpha$ be the annulus that contains $\alpha$.  Clearly $\gamma_A$ ($\gamma_A=\gamma_Z$) is a boundary curve of $Y_\alpha$.  Since $\alpha\cap T_\beta=\partial\alpha\subset\gamma_A=\gamma_Z$, $\alpha$ is an arc properly embedded in $Y_\alpha$ with both endpoints in the curve $\gamma_A$ ($\gamma_A=\gamma_Z$).  In particular, there is a subarc $p_\alpha$ of $\gamma_A$ such that $\partial p_\alpha=\partial\alpha=A\cup Z$ and $p_\alpha$ is parallel to $\alpha$ in $Y_\alpha$.  Let $D_\alpha$ be the bigon disk in $Y_\alpha\subset T_\alpha$ bounded by $p_\alpha\cup\alpha$.  Similarly, there is a bigon disk $D_\beta$ in $T_\beta$ such that $\partial D_\beta=\beta\cup p_\beta$ where $p_\beta$ is a subarc of $\gamma_A$ ($\gamma_Z=\gamma_A$) with $\partial p_\beta=\partial\beta=A\cup Z$.

\vspace{10pt}
\noindent
\emph{Claim 2}.  $p_\alpha=p_\beta$ in the curve $\gamma_A$ ($\gamma_A=\gamma_Z$).  
\begin{proof}[Proof of Claim 2]
Suppose the claim is false and $p_\alpha\ne p_\beta$.  Since $\partial p_\alpha=\partial p_\beta=A\cup Z$ and $p_\alpha\cup p_\beta\subset\gamma_A=\gamma_Z$, this implies that $p_\alpha\cup p_\beta$ is the whole curve $\gamma_A$ ($\gamma_A=\gamma_Z$).  However, this means that $\gamma_A$ bounds an embedded disk $E=D_\alpha\cup D\cup D_\beta$ in $M$ and by Lemma~\ref{Ltorus}, $\gamma_A$ must be a meridian of the normal torus $T_\alpha$.  By Corollary~\ref{Cmeridion}, $T_\alpha\cap T_\beta$ consists of meridians in both $T_\alpha$ and $T_\beta$, see Figure~\ref{tori}(a) for a picture.

\begin{figure}
\begin{center}
\psfrag{(a)}{(a)}
\psfrag{(b)}{(b)}
\psfrag{(c)}{(c)}
\psfrag{A}{$A$}
\psfrag{Z}{$Z$}
\psfrag{D}{$D$}
\psfrag{g}{$\gamma'$}
\psfrag{al}{$\alpha$}
\psfrag{be}{$\beta$}
\psfrag{Tb}{$T_\alpha$}
\psfrag{T1}{$T_\beta$}
\psfrag{pa}{$p_\alpha$}
\psfrag{pb}{$p_\beta$}
\includegraphics[width=5.0in]{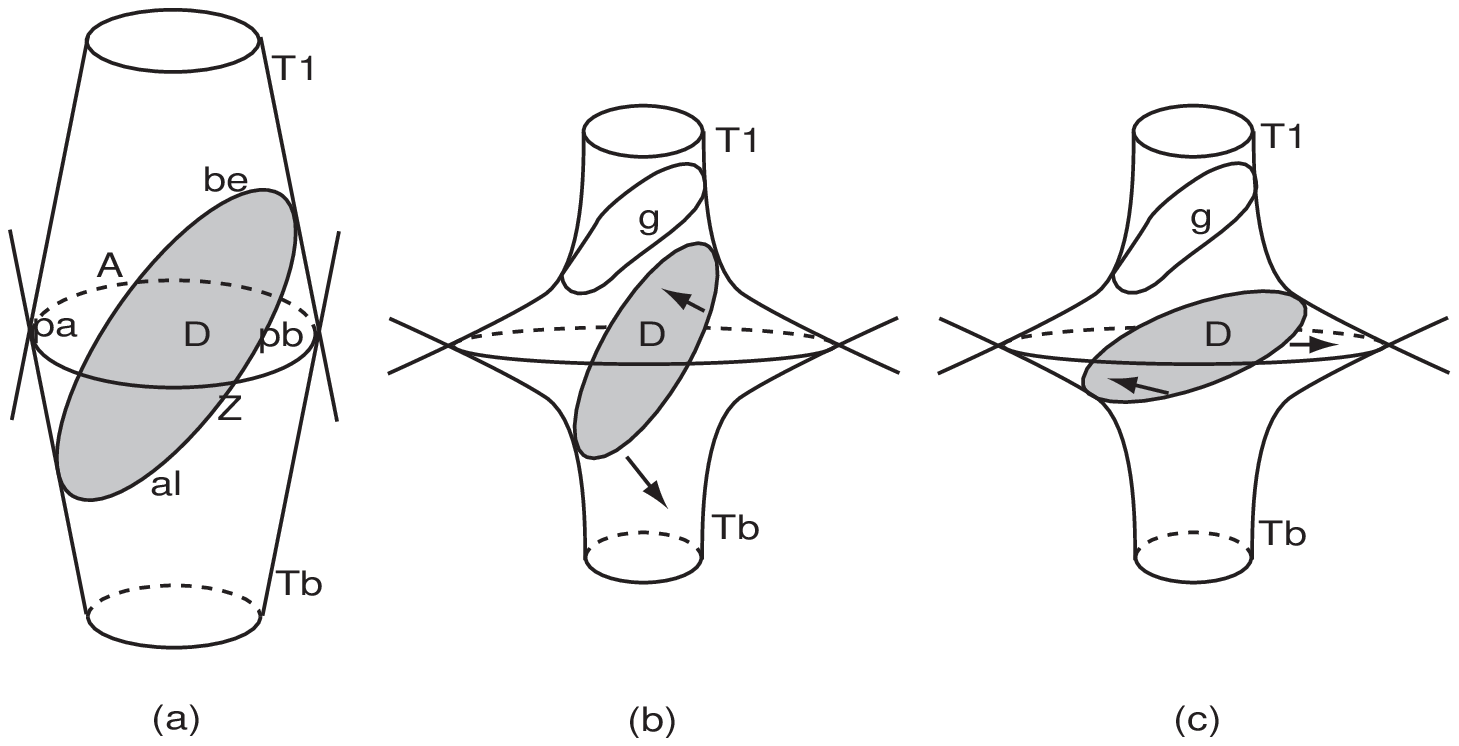}
\caption{}\label{tori}
\end{center}
\end{figure}

As in the proof of Claim 1, if we deform $D\cup T_\alpha\cup T_\beta$ into (part of) a branched surface, $D$ becomes a branch sector that is either a smooth disk or a bigon.  If $D$ becomes a smooth disk, as illustrated in Figure~\ref{tori}(a), we can perform a canonical cutting and pasting along $T_\alpha\cap T_\beta$ and view $\partial D$ as a circle in $T_\alpha+T_\beta$.  Note that $T_\alpha+T_\beta$ is a union of normal tori carried by $B_T$.  By Corollary~\ref{Cmeridion}, the curves of $T_\alpha\cap T_\beta$ are also meridians of the tori in $T_\alpha+T_\beta$. Since $\alpha$ and $\beta$ are parallel to $p_\alpha$ and $p_\beta$ in $T_\alpha$ and $T_\beta$ respectively, $\partial D=\alpha\cup\beta$ must be parallel to $p_\alpha\cup p_\beta=\gamma_A$ in $T_\alpha+T_\beta$, see Figure~\ref{tori}(a).  This means that $D$ can be viewed as a meridional disk of a torus in $T_\alpha+T_\beta$.  Since $D$ is transverse to the $I$-fibers of $N(B_T)$, as in the proofs of Lemma~\ref{LflareD2} and Corollary~\ref{Cdisk}, this implies that a subsurface of $T_\alpha+T_\beta$ is a flare based at $D$, which contradicts our no-flare hypothesis on $B_T$.  

If $D$ becomes a bigon after we deform $D\cup T_\alpha\cup T_\beta$ into (part of) a branched surface, the two corners of $D$ become cusps and the branch direction at $\partial D$ near $\partial\alpha$ is as illustrated in Figure~\ref{cusp}(b) (change $D_\gamma$ in the picture to $D$).  We have two subcases depending on the branch direction at $\partial D$.  

The first subcase is that the branch direction at $\alpha$ points out of $D_\alpha$.  As shown in Figure~\ref{tori}(b) and Figure~\ref{cusp}(b), the branch direction at $\beta$ must point out of $D_\beta$.  After deforming $D\cup T_\alpha\cup T_\beta$ into a branched surface, $D_\alpha$ becomes a branch sector (which we also call it $D_\alpha$) with $\partial D_\alpha=\alpha\cup p_\alpha$ being a cusp circle whose branch direction points out of $D_\alpha$.  Since the branch direction at $\beta$ points out of $D_\beta$, as shown in Figure~\ref{tori}(b), we can isotope the circle $p_\alpha\cup\beta$ in $T_\beta$ to a circle $\gamma'$ which is a meridian of $T_\beta$.  Note that, as shown in Figure~\ref{tori}(b), $\gamma'$ and the cusp circle $\partial D_\alpha=\alpha\cup p_\alpha$ bound a smooth annulus $A'$ (which contains $D$) in the branched surface $D\cup T_\alpha\cup T_\beta$ because of the branch direction at $\beta$.  Since $\gamma'$ is essential in $T_\beta$, similar to the proof of Lemma~\ref{LflareD2} and Corollary~\ref{Cdisk}, we can extend the annulus $A'$ to a flare based at $D_\alpha$, a contradiction.

The second subcase is that the branch direction at $\alpha$ points into $D_\alpha$.  As shown in Figure~\ref{tori}(c) and Figure~\ref{cusp}(b), the branch direction at $\beta$ must point into $D_\beta$.  In this subcase, the disk $D_\alpha\cup D$ becomes a smooth disk once we deform $D\cup T_\alpha\cup T_\beta$ into a branched surface and its boundary $\partial(D_\alpha\cup D)=p_\alpha\cup \beta$ becomes a cusp circle in the branched surface with branch direction pointing out of the disk $D_\alpha\cup D$.  Let $\gamma'$ be the curve as in the previous subcase, see Figure~\ref{tori}(c), and let $A''$ be the smooth annulus in $T_\beta$ between $\gamma'$ and the (cusp) circle $\partial(D_\alpha\cup D)=p_\alpha\cup \beta$.  As in the previous subcase, we can extend the essential annulus $A''$ to a flare based at $D_\alpha\cup D$, a contradiction.
\end{proof}

By Claim 2, we may assume $p_\alpha= p_\beta$.  By the construction of $D$ above, this means that the 3 bigon disks $D$, $D_\alpha$ and $D_\beta$ bound a football region, see Definition~\ref{Dfootball}.  We denote the football region by $\Theta$.  Now we deform the 2-sphere $\partial\Theta=D\cup D_\alpha\cup D_\beta$ into (part of a) branched surface.  Since $B$ does not carry any normal 2-sphere, there must be a cusp in $\partial\Theta$ and the cusp is a circle formed by either $\alpha\cup\beta$, or $\alpha\cup p_\alpha$ or $\beta\cup p_\alpha$.  Since $B_T$ has no flare, in any case, $\Theta$ is deformed into a trivial $D^2\times I$ region in $N(B_T)$ (with its vertical boundary annulus pinched into a cusp circle).  Note that there may be other tori in $\Sigma$ intersecting the 3-ball $\Theta$, though by our assumption on $D$, no torus intersects $\Int(D)$.  Let $\mathcal{T}_1$ be the union of the tori in $\Sigma$ that intersect $\Int(\Theta)$.  

\vspace{10pt}
\noindent
\emph{Claim 3}.  After deforming $D\cup D_\alpha\cup D_\beta$ into (part of a) branched surface as above, the cusp circle cannot be $\alpha\cup\beta$.  In other words, $\partial D$ is not a smooth circle after deforming $D\cup D_\alpha\cup D_\beta$ into branched surface.
\begin{proof}[Proof of Claim 3]
Suppose $\alpha\cup\beta$ is the cusp circle.  As $\Theta$ corresponds to a trivial $D^2\times I$ region, $\Theta\subset N(B_T)$ and every component of $\Int(\Theta)\cap\mathcal{T}_1$ is transverse to the induced $I$-fibers of $\Theta$. Since the torus $\Gamma$ is either $T_\alpha$ or $T_\beta$, without loss of generality, we suppose $D_\alpha\subset\Gamma$ ($\Gamma=T_\alpha$) in this claim.  By our construction of $D$, no torus intersects $\Int(D)$. As $T_\beta\in\Sigma$, this implies that $\Gamma\cap\Theta=D_\alpha$ and $\mathcal{T}_1\cap D_\beta$ is a collection of disjoint arcs with endpoints in $p_\beta$ ($p_\alpha=p_\beta$). 

  We perform two $B_T$-isotopies as illustrated by Figure~\ref{football}(a). We first fix $\Theta$ and perform a $B_T$-isotopy on $\mathcal{T}_1$, by pushing $\Theta\cap\mathcal{T}_1$ along the $I$-fibers, across $D_\alpha\cup D_\beta$ and out of $\Theta$.  The triple points in $\Int(p_\alpha)$ (if any) are eliminated by this $B_T$-isotopy and $\Int(\Theta)\cap\Sigma=\emptyset$ after the isotopy.   Then we can push $D$ across $\Theta$ to eliminate the pair of triple points $A\cup Z$, see Figure~\ref{football}(a).  So the $B_T$-isotopies above reduce the number of triple points in the intersection.  Since $\Gamma\cap\Theta=D_\alpha$, by ignoring $\Gamma$ in the picture, it is easy to see that the above isotopies do not change the intersection pattern of the tori in $\Sigma$ and $\Sigma$ remains a regular set after the isotopies.  This contradicts our assumption that the intersection is minimal. 
\end{proof}

\begin{figure}
\begin{center}
\psfrag{(a)}{(a)}
\psfrag{(b)}{(b)}
\psfrag{(c)}{(c)}
\psfrag{(d)}{(d)}
\psfrag{T0}{$T$}
\psfrag{D1}{$\Delta$}
\psfrag{th}{$\theta$}
\psfrag{A}{$T_\alpha$}
\psfrag{Tb}{$T_\beta$}
\psfrag{Tb2}{$T_\beta=\Gamma$}
\psfrag{Da}{$D_\alpha$}
\psfrag{Db}{$D_\beta$}
\psfrag{tb}{$\theta_\beta$}
\psfrag{td}{$\theta_D$}
\psfrag{b}{$\beta$}
\psfrag{E}{$\Theta$}
\psfrag{D}{$D$}
\psfrag{Z}{$Z$}
\psfrag{d}{$\delta_\alpha$}
\psfrag{V1}{$V_1$}
\psfrag{V2}{$V_2$}
\psfrag{V2p}{$V_2'$}
\psfrag{e1}{$\eta_1$}
\psfrag{e2}{$\eta_2$}
\psfrag{eg}{$\eta_G$}
\psfrag{G}{$G$}
\includegraphics[width=4.75in]{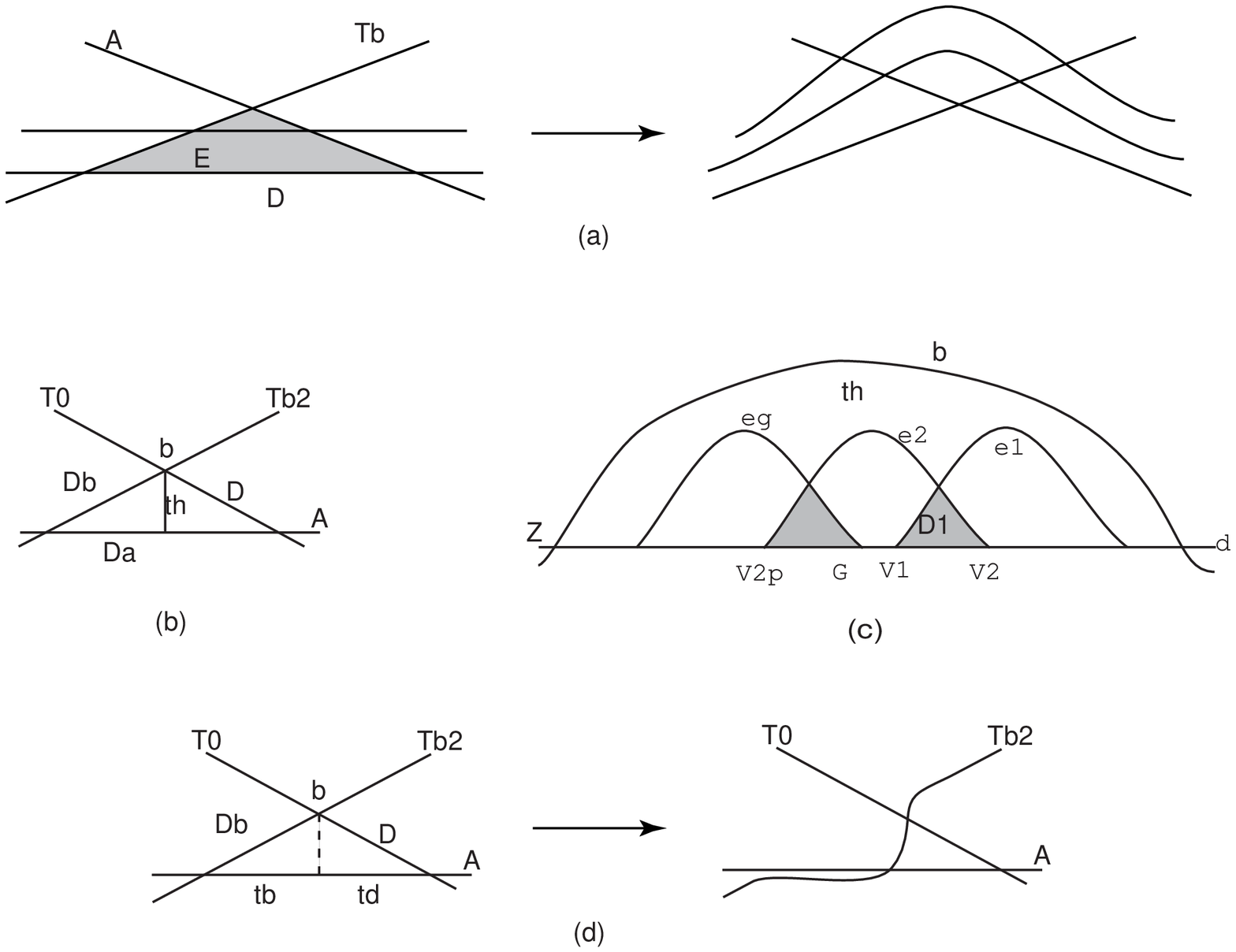}
\caption{}\label{football}
\end{center}
\end{figure}

By Claim 3, the cusp of $\partial\Theta$ must be formed by $p_\alpha\cup\alpha$ or $p_\alpha\cup\beta$. Without loss of generality, we assume $p_\alpha\cup\alpha$ is the cusp circle for $\Theta$.  So $D_\alpha$ is a smooth disk and $D\cup D_\beta$ is the other smooth disk in $\partial\Theta$ (after deforming it into branched surface).  There is a big difference between the cases $\Gamma=T_\alpha$ and $\Gamma=T_\beta$, because $p_\alpha\cup\alpha$ is the cusp circle.

If $\Gamma=T_\alpha$, then $T_\beta\in\Sigma$.  Since the intersection of the tori in $\Sigma$ has no triple point, the intersection arcs $\mathcal{T}_1\cap D_\beta$ has no double point in $\Int(D_\beta)$. By our assumption on $D$, $\Int(\beta)$ contains no triple point.  So $\mathcal{T}_1\cap D_\beta$ consists of disjoint arcs with endpoints in $p_\alpha$ ($p_\alpha=p_\beta$).  Similar to the isotopy in Claim 3, we perform a $B_T$-isotopy on $T_\alpha$ ($T_\alpha=\Gamma$) while fixing every other torus by pushing $D_\alpha$ across $\Theta$ and eliminating the triple points $A$ and $Z$.  This isotopy can be viewed as replacing $D_\alpha$ by a parallel copy of $D\cup D_\beta$.  By our assumption on $D$, the intersection of the tori has no triple points lying in $D-(A\cup Z)$.  Since there is no triple point in $\Int(D_\beta)$, no new triple points are created by this isotopy and all triple points in $p_\alpha=p_\beta$ are also eliminated by this isotopy.  As $\Sigma$ is fixed by the isotopy and the number of triple points is reduced, this contradicts our assumption that the intersection is minimal.  

Next we suppose $\Gamma=T_\beta$, which means $T_\alpha\in\Sigma$.  Recall that $\mathcal{T}_1$ is the union of the tori in $\Sigma$ that intersect $\Int(\Theta)$.  If $\mathcal{T}_1=\emptyset$, then a $B_T$-isotopy pushing $D_\alpha$ across $\Theta$ can eliminate the pair of triple points $A$ and $Z$, which contradicts our assumption that the intersection is minimal.  Thus we may assume $\mathcal{T}_1\ne\emptyset$.  By our assumption on $D$, $\mathcal{T}_1\cap D=\emptyset$.  Hence $\mathcal{T}_1\ne\emptyset$ implies that $\mathcal{T}_1\cap\Int(p_\alpha)\ne\emptyset$.   As $T_\alpha\in\Sigma$, $\mathcal{T}_1\cap D_\alpha$ is a collection of mutually disjoint arcs with endpoints in $p_\alpha$. 

The union of the vertical arcs (of $N(B_T)$) in $\Theta$ with one endpoint in $\beta$ and the other endpoint in $D_\alpha$ is a vertical bigon $\theta\subset\Theta$.  One boundary edge of $\theta$ is $\beta$ and the other boundary edge of $\theta$, denoted by $\delta_\alpha$, is an arc in $D_\alpha$ connecting $A$ to $Z$, see Figure~\ref{football}(b) for a one dimensional schematic picture where the vertical arc denotes $\theta$ and the top vertex denotes $\beta$.   The arcs in $\theta\cap\mathcal{T}_1$ are transverse to the $I$-fibers and with all endpoints in $\delta_\alpha$, since $D\cap\mathcal{T}_1=\emptyset$.

Next we analyze the intersection pattern of $\theta\cap\mathcal{T}_1$.   Our goal is to simplify the intersection of $\theta\cap\mathcal{T}_1$ by isotopies that (1) do not increase the number of triple points and (2) preserve $\Sigma$ as a regular set. 

We first consider the case that $\theta\cap\mathcal{T}_1=\emptyset$.  The arc $\delta_\alpha$ cuts $D_\alpha$ into two subdisks $\theta_\beta$ and $\theta_D$ which are $B_T$-isotopic to $D_\beta$ and $D$ respectively.  As $\theta\cap\mathcal{T}_1=\emptyset$, by the construction of $D$, the arcs of $D_\alpha\cap\mathcal{T}_1$ all lie in $\theta_\beta$.  Now we perform a $B_T$-isotopy on $\Gamma$ ($\Gamma=T_\beta$) as illustrated in Figure~\ref{football}(d), which is basically replacing $D_\beta$ by a copy of the disk $\theta_\beta\cup\theta$ (then perturbing $\theta_\beta\cup\theta$ to be transverse to the $I$-fibers).  Since $T_\alpha$ is a torus in the regular set $\Sigma$, there is no triple point in $\Int(\theta_\beta)$.  As $\theta\cap\mathcal{T}_1=\emptyset$, this isotopy does not gain any new triple point.  Moreover, by our assumption above that $\mathcal{T}_1\cap\Int(p_\alpha)\ne\emptyset$, there are triple points in $\Int(p_\alpha)$ and clearly all triple points in $\Int(p_\alpha)$ are eliminated by this isotopy.  Thus this $B_T$-isotopy on $\Gamma$ reduces the number of triple points.  As $\Sigma$ is fixed by this isotopy, $\Sigma$ remains a regular set.  This contradicts our assumption that the intersection is minimal.  Therefore, we may assume $\theta\cap\mathcal{T}_1\ne\emptyset$.

For any two components $\eta_1$ and $\eta_2$ of $\delta_\alpha\cup(\theta\cap\mathcal{T}_1)$, if $\eta_1\cap\eta_2$ contains more than one point (note that it is possible that $\eta_i=\delta_\alpha$), then $\eta_1$ and $\eta_2$ form a bigon in $\theta$, i.e. there is a subarc $e_i$ of $\eta_i$ ($i=1,2$) with $\partial e_1=\partial e_2\subset\eta_1\cap\eta_2$ such that $e_1\cup e_2$ bounds a disk $\Delta_e\subset\theta$ and $\eta_i\cap\Delta_e=e_i$ ($i=1,2$).  Since $\theta$ is vertical in $N(B_T)$ and $\eta_i$ is transverse to the $I$-fibers, if we deform the tori in $\Sigma$ into a branched surface, then $\Delta_e$ is deformed into a bigon with two cusps at $\partial e_i$.  Let $T_1$ and $T_2$ be the two tori in $\Sigma$ containing $\eta_1$ and $\eta_2$ respectively.  By our assumption on $D$ and $\Theta$, $T_i=T_\alpha$ if and only if $\eta_i=\delta_\alpha$.  

Next we show that the two vertices of the bigon $\Delta_e$ lie in the same double curve of $T_1\cap T_2$.  To prove this, we only need to consider $T_1$ and $T_2$ and can ignore how other tori intersect $\Delta_e$.  We first consider the simplest case that the bigon $\Delta_e$ is innermost with respect to $T_1\cup T_2$, i.e., the case that $\Delta_e\cap(T_1\cup T_2)=\partial\Delta_e=e_1\cup e_2$.  Suppose the two vertices of $\Delta_e$ lie in different double curves of $T_1\cap T_2$.  As $\Delta_e\cap(T_1\cup T_2)=e_1\cup e_2$, this implies that $T_1$ and $T_2$ form a $bigon\times S^1$ product region which contains the bigon $\Delta_e$ as a meridional disk.  As $\theta$ is vertical, by Lemma~\ref{Lproduct}, this $bigon\times S^1$ product region must be trivial, which contradicts our assumption at the beginning of the proof that the tori in $\Sigma$ do not form any trivial $bigon\times S^1$ product region.  Thus the two vertices of any innermost bigon $\Delta_e$ lie in the same double curve of $T_1\cap T_2$.  If $\Int(\Delta_e)\cap(T_1\cup T_2)\ne\emptyset$, then $\Delta_e\cap(T_1\cup T_2)$ cuts $\Delta_e$ into a collection of smaller bigons and there is an innermost bigon $\Delta_e'$ formed by $T_1$ and $T_2$ and with $\Delta_e'\cap(T_1\cup T_2)=\partial\Delta_e'$.  As above, the two vertices of $\Delta_e'$ lie in the same double curve of $T_1\cap T_2$.  Then we can perform an isotopy on $T_1$, similar to a $\partial$-compression, by pushing one edge of $\Delta_e'$ across $\Delta_e'$ to cancel the two vertices (or double points) of $\Delta_e'$. 
Since the two vertices of $\Delta_e'$ lie in the same double curve, this isotopy changes this double curve into a parallel (essential) double curve plus a trivial double curve, and by the no-flare hypothesis, we can eliminate the trivial double curve using a trivial isotopy.  The two isotopies together can be viewed as pulling the double curve containing the two vertices of $\Delta_e'$ across $\Delta_e'$.  In particular, $T_1$ and $T_2$ do not form any trivial $bigon\times S^1$ region after the isotopies.  After finitely many such isotopies, $\Delta_e$ becomes an innermost bigon and we can use the argument above on innermost bigons to conclude that the two vertices of $\Delta_e$ lie in the same double curve.  Moreover, it follows from the isotopies on $\Delta_e'$ above that the two vertices of $\Delta_e$ must lie in the same double curve before all the isotopies on innermost bigons $\Delta_e'$ above.  Note that the only purpose of the isotopies above is to demonstrate that the two vertices of the bigon $\Delta_e$ lie in the same double curve of $T_1\cap T_2$, and we do not actually perform these isotopies in our main proof.
Furthermore, it follows from this argument that $e_i$ is homotopic in $T_i$ (fixing $\partial e_i$) to a subarc of the curve $T_1\cap T_2$ that contains the two vertices of $\Delta_e$.

We say the bigon $\Delta_e$ above is a simple bigon if $\Int(\Delta_e)\cap\Sigma=\emptyset$.  By the conclusion above, the two vertices of $\Delta_e$ lie in the same double curve. 
Suppose $\Delta_e$ above is a simple bigon.  Then we can perform an isotopy on $T_1$, similar to a $\partial$-compression, by pushing $e_1$ across $\Delta_e$ to cancel the two intersection points $\partial e_1$ in $\theta$.  As the bigon $\Delta_e$ is simple, this isotopy does not create any triple point in the torus intersection.  Since $\partial e_i$ lies in the same double curve, the isotopy changes this double curve into a parallel (essential) double curve plus a trivial double curve.  As before, since $B_T$ has no flare, we can eliminate the resulting trivial double curve using a trivial isotopy without increasing the number of triple points.  Moreover, after eliminating this resulting trivial double curve, $\Sigma$ remains a regular set of tori.   Thus after finitely many such isotopies, we may assume the arcs in $\theta\cap\Sigma$ do not form any simple bigon.

\vspace{10pt}
\noindent
\emph{Claim 4}.  Let $\Delta_e$, $e_1$, $e_2$, $T_1$ and $T_2$ be as above.  Let $T'\ne T_1$ be any torus in $\Sigma$ and let $l$ be a double curve of $T'\cap T_1$.   We fix an orientation along $e_1$ and a normal direction for $l$ in $T_1$ and assign positive and negative signs for each point in $\Int(e_1)\cap l$ according to the fixed orientations above.  Then the number of positive intersection points of $\Int(e_1)\cap l$ equals the number of negative intersection points, in particular, $|\Int(e_1)\cap l|$ is an even number. 
\begin{proof}[Proof of Claim 4]
As $\Delta_e\subset\theta$, if we deform $\Sigma$ into a branched surface, then $\Delta_e$ is deformed into a bigon with cusp directions at the two vertices pointing out of $\Delta_e$.  By our discussion above, the two vertices of $\Delta_e$ lie in the same double curve $\gamma_e$ of $T_1\cap T_2$.  Now we view $e_1$, $\gamma_e$ and $l$ as curves in $T_1$. 
Because of the cusp directions at $\partial e_1$, as shown in Figure~\ref{arcs}(a), the two small arc neighborhoods of the two points of $\partial e_1$ in $e_1$ lie on the same side of $\gamma_e$ in $T_1$.  Since $\Sigma$ is a regular set, either $l=\gamma_e$ or $l\cap\gamma_e=\emptyset$ and $l$ is parallel to $\gamma_e$ in $T_1$.  By the conclusion above, $e_1$ is homotopic in $T_1$ (fixing $\partial e_1$) to a subarc of $\gamma_e$ bounded by $\partial e_1$.  By the cusp directions at $\partial e_1$ above and as shown in Figure~\ref{arcs}(a), clearly the number of positive points in $\Int(e_1)\cap l$ equals the number of negative points.  
\end{proof}

\begin{figure}
\begin{center}
\psfrag{(a)}{(a)}
\psfrag{(b)}{(b)}
\psfrag{(c)}{(c)}
\psfrag{(d)}{(d)}
\psfrag{T1}{$T_1$}
\psfrag{T2}{$T_2$}
\psfrag{e1}{$e_1$}
\psfrag{e}{$\gamma_e$}
\psfrag{da}{$\delta_\alpha$}
\psfrag{l}{$l$}
\psfrag{+}{$+$}
\psfrag{-}{$-$}
\psfrag{P}{$P$}
\psfrag{Q}{$Q$}
\psfrag{P1}{$P'$}
\psfrag{Q1}{$Q'$}
\psfrag{g1}{$\gamma_1$}
\psfrag{g}{$\gamma_p'$}
\psfrag{gq}{$\gamma_q'$}
\psfrag{tq}{$\theta_q$}
\includegraphics[width=4.5in]{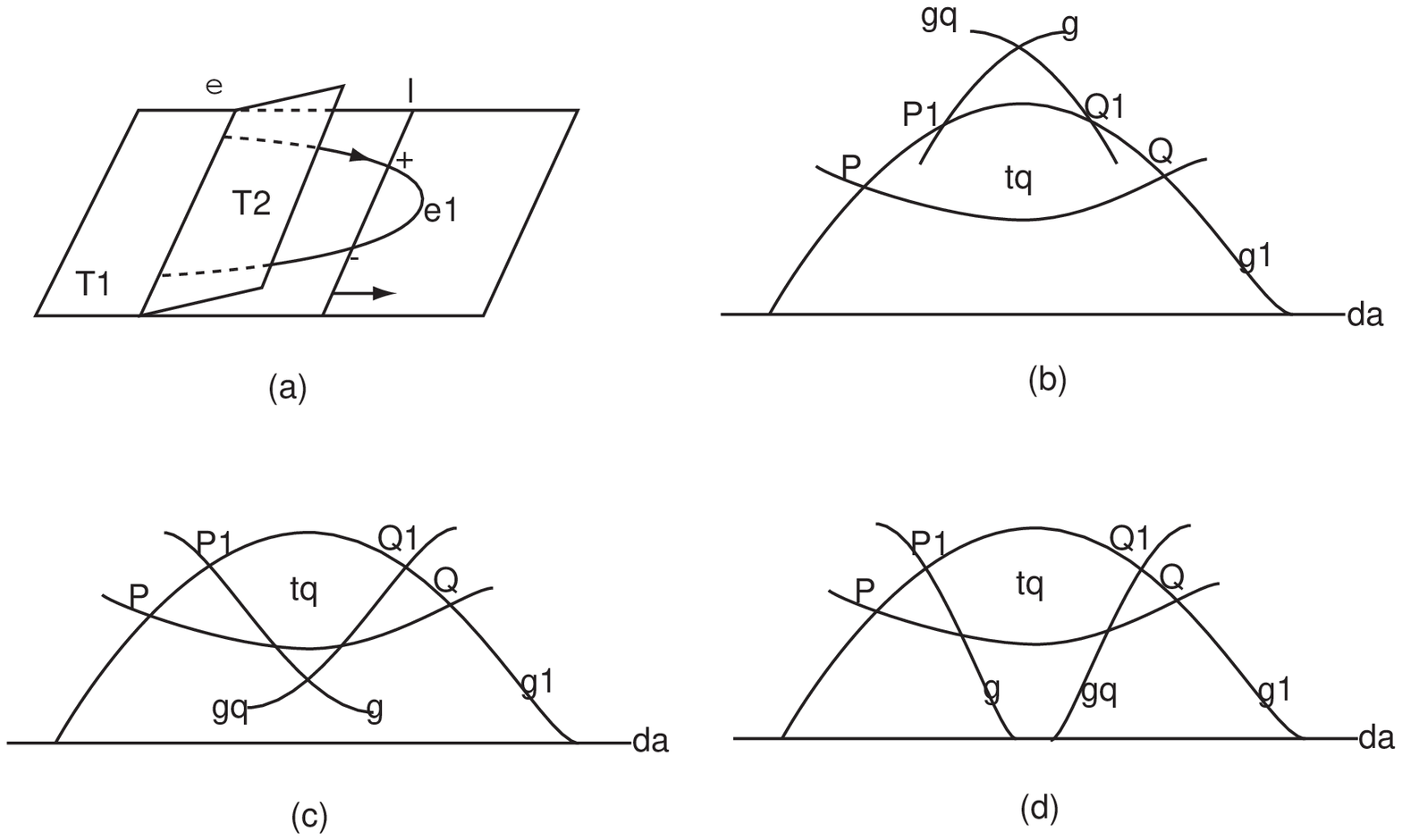}
\caption{}\label{arcs}
\end{center}
\end{figure}

Next we will use the properties of $\theta\cap\Sigma$ to get a contradiction to the no-simple-bigon assumption.  The key reason behind everything is the hypothesis that $B_T$ has no flare.

Let $\eta_1$ and $\eta_2$ be two components of $\theta\cap\mathcal{T}_1$.  So each $\eta_i$ is an arc properly embedded in $\theta$ with $\partial\eta_i\subset\delta_\alpha$.  We say $\eta_1$ and $\eta_2$ alternate along $\delta_\alpha$ if the subarc of $\delta_\alpha$ bounded by $\partial\eta_1$ contains exactly one endpoint of $\eta_2$.  We would like to remind the reader that curves in $\theta\cap\mathcal{T}_1$ are transverse to the $I$-fibers, so the intersection pattern of $\theta\cap\mathcal{T}_1$ cannot be overly complicated.

\vspace{10pt}
\noindent
\emph{Case a}.  There are two arcs $\eta_1$ and $\eta_2$ in $\theta\cap\mathcal{T}_1$ that alternate along $\delta_\alpha$.
\vspace{10pt}

In this case, since both $\eta_1$ and $\eta_2$ are transverse to the $I$-fibers, there must be a triangle $\Delta$ formed by $\eta_1$ and $\eta_2$ and $\delta_\alpha$ (see Figure~\ref{football}(c) for a picture) such that (1) the 3 edges $e_1$, $e_2$ and $e_\alpha$ of $\Delta$ lie in $\eta_1$, $\eta_2$ and $\delta_\alpha$ respectively, and (2) if we deform $\Sigma$ into a branched surface, $\Delta$ becomes a bigon with both cusps at $\partial e_\alpha\subset\delta_\alpha$.  Note that it is possible that $\Int(\Delta)$ intersects other parts of $\eta_i$.  We say the triangle $\Delta$ above is formed by two arcs that alternate along $\delta_\alpha$.   We may assume $\Delta$ is closest to the endpoint $Z$ of $\delta_\alpha$ among all such triangles.  More precisely, let $\delta_z$ be the component of $\delta_\alpha-e_\alpha$ that contains $Z$, we assume there is no triangle (formed by arcs that alternate along $\delta_\alpha$) with an edge in $\delta_z$.

Let $V_i$ ($i=1,2$) be the vertex $e_i\cap e_\alpha$ of $\Delta$, let $T_i$ ($i=1,2$) be the torus containing $\eta_i$, and let $l_i$ ($i=1,2$) be the double curve in $T_i\cap T_\alpha$ that contains the vertex $V_i$, see Figure~\ref{football}(c).  As $\eta_1\cap\eta_2\ne\emptyset$, $T_1$ and $T_2$ are not the same torus and $l_1\ne l_2$.  Without loss of generality, we assume $V_1$ is closer to $Z$ in $\delta_\alpha$ than $V_2$, i.e., $V_2$ lies outside the subarc $\delta_z$ of $\delta_\alpha$ bounded by $Z\cup V_1$. Since $\eta_1$ and $\eta_2$ alternate along $\delta_\alpha$ and are transverse to the (induced) $I$-fibers of $\theta$, as shown in Figure~\ref{football}(c), one endpoint of $\eta_2$ is $V_2$ and the other endpoint of $\eta_2$, which we denote by $V_2'$, lies in $\Int(\delta_z)$.  Let $\delta_2$ be the subarc of $\delta_\alpha$ bounded by $\partial\eta_2=V_2\cup V_2'$.  Recall that if we naturally deform the arcs in $\theta\cap\Sigma$ into a train track, then the subdisk of $\theta$ bounded by $\eta_2\cup\delta_2$ is a bigon with two cusps at $V_2$ and $V_2'$ and the cusp directions pointing out of $\delta_2$.  Moreover, by the argument before Claim 4, $V_2$ and $V_2'$ lie in the same double curve $l_2$ of $T_\alpha\cap T_2$.  

Now we consider $\delta_2\cap l_1$ in $T_\alpha$ (recall that$l_1$ is the double curve of $T_1\cap T_\alpha$ that contains $V_1$). We fix an orientation for $\delta_2$ pointing from $V_2'$ to $V_2$ and assign a sign to each intersection point of $\delta_2\cap l_1$: we call a point $G$ of $\delta_2\cap l_1$ a positive point if the cusp direction  of the arc (in $\theta\cap T_1$ containing $G$) at $G$ agrees with the orientation of $\delta_2$ above, otherwise we call $G$ a negative point.  So by the assumption on $\Delta$, $V_1$ is a negative point in $\delta_2\cap l_1$.  By Claim 4, the number of positive points in $\delta_2\cap l_1$ equals the number of negative points in $\delta_2\cap l_1$.  Note that if an arc of $T_1\cap \theta$ has both endpoints in $\delta_2$, then the signs at the two endpoints are opposite.  As $V_1$ is a negative point and since $\eta_1$ and $\eta_2$ alternate in $\delta_\alpha$, there must be a positive point $G$ in $\delta_2\cap l_1$ such that the other endpoint of the component $\eta_G$ of $T_1\cap\theta$ that contains $G$ lies outside $\delta_2$, see Figure~\ref{football}(c).  By our construction, $G\in l_1\cap\delta_2\subset T_1$, so both $\eta_G$ and $\eta_1$ are arcs in $T_1\cap\theta$.  As $T_1$ is an embedded surface, we have $\eta_G\cap\eta_1=\emptyset$.  This plus the assumption on the positive cusp direction at $G$ implies that (1) $G$ lies in the arc $\delta_2-e_\alpha$ and (2) the other endpoint $\partial\eta_G-G$ of $\eta_G$ lies in the subarc of $\delta_\alpha$ bounded by $Z\cup V_2'$, as shown in Figure~\ref{football}(c). Note that these conclusions hold also because these arcs are all transverse to the (induced) $I$-fibers of $\theta$.  Thus, as shown in Figure~\ref{football}(c), $\eta_2$ and $\eta_G$ alternate in $\delta_\alpha$ and there is a triangle formed by $\eta_2$ and $\eta_G$ that is closer to $Z$ along $\delta_\alpha$ than the triangle $\Delta$, a contradiction to our assumption on $\Delta$ at the beginning.  So Case (a) cannot happen.

\vspace{10pt}
\noindent
\emph{Case b}.  No two arcs in $\theta\cap\mathcal{T}_1$ alternate along $\delta_\alpha$.
\vspace{10pt}

Recall that we have assumed that $\theta\cap\Sigma$ has no simple bigon.  For any arcs $\gamma_1$ and $\gamma_2$ of $\theta\cap\mathcal{T}_1$ that intersect each other, since $\gamma_1$ and $\gamma_2$ do not alternate, $\gamma_1\cap\gamma_2$ contains more than one point.  Hence there is a subarc $q_i$ of $\gamma_i$ ($i=1,2$) with $q_1\cap q_2=\partial q_1=\partial q_2$.  So $q_1\cup q_2$ bounds a bigon $\theta_q\subset\Int(\theta)$.  Let $T_i$ be the torus in $\Sigma$ that contains $q_i$ and $\gamma_i$ ($i=1,2$).    

Let $\pi:\theta\to\delta_\alpha$ be the map collapsing each (induced) $I$-fiber to a point.  Since the arcs in $\theta\cap\Sigma$ are transverse to the $I$-fibers, $\pi(q_1)=\pi(q_2)$ and $\pi(\gamma_i)$ is the subarc of $\delta_\alpha$ bounded by $\partial\gamma_i$ ($i=1,2$).  We may choose $\gamma_1$ to be thinnest in the sense that there is no component $\gamma$ of $\theta\cap\Sigma$ with $\pi(\gamma)\subset\Int(\pi(\gamma_1))$.  Note that since there is no simple bigon and since no arcs of $\theta\cap\mathcal{T}_1$ alternate along $\delta_\alpha$, being thinnest implies that there must be an arc in $\theta\cap\mathcal{T}_1$ intersecting $\Int(\gamma_1)$ and there is a bigon with an edge in $\Int(\gamma_1)$.  After fixing $\gamma_1$, we may assume the bigon $\theta_q$ above is the shortest along $\gamma_1$ in the sense that there is no bigon with an edge totally lying in $\Int(q_1)$.

Let $P$ and $Q$ be the two vertices of the bigon $\theta_q$. 
We first suppose $\Int(q_1)$ contains a double point of $\theta\cap\Sigma$ (i.e. $\Int(q_1)\cap (\Sigma-T_1)\ne\emptyset$).  Let $P'$ be the double point of $\Int(q_1)\cap (\Sigma-T_1)$ that is closest to $P$ in $q_1$, that is, the subarc of $q_1$ bounded by $P\cup P'$ contains no other double point of $\Int(q_1)\cap (\Sigma-T_1)$.  Let $\gamma_p'$ be the component of $\theta\cap\Sigma$ that intersects $q_1$ at $P'$, let $T'$ be the torus in $\Sigma$ containing $\gamma_p'$, and let $l_p'$ be the double curve of $T_1\cap T'$ containing $P'$.    

By deforming the tori in $\Sigma$ into a branched surface as before, we can deform $\theta_q\cap\Sigma$ into a train track and each double point of $\Int(q_1)\cap (\Sigma-T_1)$ becomes a switch of the train track.  We can assign a sign for each double point of $\Int(q_1)\cap (\Sigma-T_1)$ as follows: a double point is positive if the cusp direction at this double point (of the train track $\theta_q\cap\Sigma$) points towards $P$ in $q_1$; otherwise we say this double point is negative.

By Claim 4, there must be another double point $Q'$ in $l_p'\cap\Int(q_1)$ such that $P'$ and $Q'$ have opposite signs.  Let $\gamma_q'$ be the component of $\theta\cap\Sigma$ that intersects $q_1$ at $Q'$.  Since $P'$ and $Q'$ are in the same double curve $l_p'$, both $\gamma_p'$ and $\gamma_q'$ lie in the same torus $T'$.  Hence either $\gamma_p'=\gamma_q'$ or $\gamma_p'\cap\gamma_q'=\emptyset$.  If $\gamma_p'=\gamma_q'$ then $P'$ and $Q'$ are connected by a subarc of $\gamma_p'$ and this means that there is a bigon formed by subarcs of $\gamma_p'$ and $q_1$ with an edge of the bigon lying in the subarc of $q_1$ bounded by $P'\cup Q'$.  This contradicts our assumption that $\theta_q$ is the shortest bigon along $\gamma_1$.  Thus $\gamma_p'\ne\gamma_q'$ and $\gamma_p'\cap\gamma_q'=\emptyset$.  Moreover, by our construction, $\pi(\gamma_1)\cap\pi(\gamma_p')\ne\emptyset$ and $\pi(\gamma_1)\cap\pi(\gamma_q')\ne\emptyset$ in $\delta_\alpha$.
Since no two arcs in $\theta\cap\mathcal{T}_1$ alternate along $\delta_\alpha$, this means that $\pi(\gamma_1)$ and $\pi(\gamma_p')$ are nested in $\delta_\alpha$. 
By our assumption above that $\gamma_1$ is a thinnest component of $\theta\cap\mathcal{T}_1$, we have $\pi(\gamma_1)\subset\pi(\gamma_p')$ and $\pi(\gamma_1)\subset\pi(\gamma_q')$. 

Let $q_1'$ be the subarc of $q_1$ bounded by $P'\cup Q'$.  If $\gamma_p'$ or $\gamma_q'$ intersects $\Int(q_1')$, then there is a bigon formed by $\gamma_p'\cup\gamma_1$ or $\gamma_q'\cup\gamma_1$ with an edge in $q_1'$, which contradicts our above assumption that the bigon $\theta_q$ is the shortest along $\gamma_1$.  Thus we may assume $\Int(q_1')\cap \gamma_p'=\emptyset$ and $\Int(q_1')\cap \gamma_q'=\emptyset$.  This plus the assumption that $P'$ and $Q'$ have opposite signs implies that either (1), as shown in Figure~\ref{arcs}(b) and Figure~\ref{arcs}(c), $\gamma_p'$ and $\gamma_q'$ have to cross each other, contradicting that $\gamma_p'\cap\gamma_q'=\emptyset$ or (2), as shown in Figure~\ref{arcs}(d), an endpoint of $\gamma_p'$ lies in $\pi(q_1')$, contradicting our conclusion above that $\pi(\gamma_1)\subset\pi(\gamma_p')$.  Note that the reason why there is no other configuration is that each $\theta\cap\mathcal{T}_1$ is transverse to the (induced) $I$-fibers of $\theta$.

So the argument above means that $\Int(q_1)$ contains no double point of $\theta\cap\Sigma$.  In this case, we can apply our arguments for $\theta$ above to $\theta_q$ (by viewing $q_1$ as $\beta$ and viewing $q_2$ as $\delta_\alpha$).  Eventually, either we get various contradictions as above, or we can conclude that there is a simple bigon inside $\theta_q$, which contradicts our assumption that there is no simple bigon.

Therefore if the intersection of the tori in $\Sigma\cup\Gamma$ is minimal, then there is no triple point and $\Sigma\cup\Gamma$ is a regular set.
\end{proof}

Although we assume the intersection is minimal at the beginning of the proof of Lemma~\ref{Ltriple}, 
the isotopies that we performed in the proof (in order to reduce the number of triple points) are either trivial isotopies removing trivial product regions, or the isotopies in Figure~\ref{football}(a, d).  All these isotopies can be made algorithmic.  So one can follow the proof of Lemma~\ref{Ltriple} to algorithmically eliminate the triple points and change $\Sigma\cup\Gamma$ into a regular set of tori.

Next we study $\mathcal{S}(\mathcal{T})$, the set of (possibly disconnected) surfaces obtained by cutting and pasting multiple copies of tori in $\mathcal{T}$, see section~\ref{Spre} for definition.  We will also consider the set of tori $\mathcal{G}(\mathcal{T})$ defined in Notation~\ref{N1}.

\begin{lemma}\label{Lset}
Let $\mathcal{T}$ be a finite set of normal tori carried by $B$. Suppose the sub-branched surface $B_T$ of $B$ that fully carries the union of the tori in $\mathcal{T}$ contains no flare.  Then there is a finite set of normal tori $\hat{\mathcal{T}}$ carried by $B$ such that for any $F\in\mathcal{S}(\mathcal{T})$, there is a finite regular subset of normal tori $\mathcal{T}_F\subset\hat{\mathcal{T}}$ such that $F\in\mathcal{S}(\mathcal{T}_F)$.  Moreover, the tori in $\hat{\mathcal{T}}$ and all possible regular subsets $\mathcal{T}_F$ of $\hat{\mathcal{T}}$ can be algorithmically determined.
\end{lemma}
\begin{proof}
The key point of the lemma is that $\mathcal{T}_F$ is a regular set of tori. 
If the tori in $\mathcal{T}$ are disjoint or more generally if $\mathcal{T}$ is a regular set of tori, then $\hat{\mathcal{T}}=\mathcal{T}$ and there is nothing to prove.  Moreover, since each component of a surface in $\mathcal{S}(\mathcal{T})$ is a surface in $\mathcal{G}(\mathcal{T})$, if $\mathcal{G}(\mathcal{T})$ is a finite set and if we can algorithmically find all possible tori in $\mathcal{G}(\mathcal{T})$, then by setting $\hat{\mathcal{T}}=\mathcal{G}(\mathcal{T})$, the lemma also holds trivially (with the subset $\mathcal{T}_F$ in the lemma being a collection of disjoint tori).  For simplicity, we may assume the union of the tori in $\mathcal{T}$ is a connected 2-complex.  

\vspace{10pt}
\noindent
\emph{Claim}.  The lemma holds for $\mathcal{T}$ if we can algorithmically find a good torus (see Definition~\ref{Dgood}) that non-trivially intersects at least one torus in $\mathcal{T}$.  
\begin{proof}[Proof of the Claim]
Suppose $T$ is a good torus as in the claim and let $T_1\in \mathcal{T}$ be a torus that non-trivially intersects $T$.  By Lemma~\ref{Lslope} and by the definition of good torus, after some trivial isotopies removing trivial intersection curves, $T\cap T_1$ consists of essential and non-meridional curves in $T$ and hence $\{T, T_1\}$ is a regular set of tori.  We can repeatedly apply Lemma~\ref{Ltriple}, using $T$ as the special torus in Lemma~\ref{Ltriple},  to add tori from $\mathcal{T}$ to the regular set.  Eventually we can conclude that $T\cup\mathcal{T}$ is a regular set of tori after $B_T$-isotopy.  As above, we can set $\hat{\mathcal{T}}=T\cup\mathcal{T}$.  
Note that the proof of Lemma~\ref{Ltriple} is algorithmic, so we can algorithmically isotope these tori into a regular set of tori, see the remark after the proof of Lemma~\ref{Ltriple}.  
\end{proof}

Now we use the claim to prove Lemma~\ref{Lset}.  We can divide the set $\mathcal{T}$ into a collection of subsets of tori $\mathcal{P}_1,\dots,\mathcal{P}_m$ such that the tori in each $\mathcal{P}_i$ are disjoint and $\mathcal{T}=\bigcup_{i=1}^m\mathcal{P}_i$.  Clearly if $m=1$, i.e., the tori in $\mathcal{T}$ are disjoint, then Lemma~\ref{Lset} holds trivially.  We use induction on $m$ and suppose Lemma~\ref{Lset} is true for any set of tori with smaller such number $m$.

Next we consider $\mathcal{P}_1\cup\mathcal{P}_2$.  Since the tori in each $\mathcal{P}_i$ are disjoint, the intersection of the tori in $\mathcal{P}_1\cup\mathcal{P}_2$ has no triple point.  Since $B_T$ has no flare, as in Notation~\ref{N1}, after some $B_T$-isotopy removing trivial double curves if necessary, we may assume each double curve in the intersection of the tori in $\mathcal{P}_1\cup\mathcal{P}_2$ is essential in both tori.  Thus by Lemma~\ref{Lgoodone} either (1) we can algorithmically find a good torus $T$ that non-trivially intersects at least one torus in $\mathcal{P}_1\cup\mathcal{P}_2$, or (2) $\mathcal{G}(\mathcal{P}_1\cup\mathcal{P}_2)$ is finite and we can algorithmically find a finite list of tori $\mathcal{G}_p$ containing $\mathcal{G}(\mathcal{P}_1\cup\mathcal{P}_2)$.  As $\mathcal{T}=\bigcup_{i=1}^m\mathcal{P}_i$, in the first possibility, the good torus $T$ non-trivially intersects at least one torus in $\mathcal{T}$ and the lemma follows from the claim above.  Suppose we are in possibility (2) above and $\mathcal{G}(\mathcal{P}_1\cup\mathcal{P}_2)$ is finite.  Note that each surface in $\mathcal{S}(\mathcal{P}_1\cup\mathcal{P}_2)$ consists of copies of disjoint tori in $\mathcal{G}(\mathcal{P}_1\cup\mathcal{P}_2)$.  By enumerating all possible subsets of disjoint tori, we can find a collection of finite subsets of $\mathcal{G}_p\supset\mathcal{G}(\mathcal{P}_1\cup\mathcal{P}_2)$, denoted by $\mathcal{Q}_1,\dots,\mathcal{Q}_k$, such that (1) the tori in each $\mathcal{Q}_i$ are disjoint and (2) $\mathcal{S}(\mathcal{P}_1\cup\mathcal{P}_2)\subset\bigcup_{i=1}^k\mathcal{S}(\mathcal{Q}_i)$.  Since the tori in $\mathcal{Q}_i$ are disjoint and by the induction hypothesis, Lemma~\ref{Lset} holds for each set of tori $\mathcal{Q}_i\cup\mathcal{P}_3\cup\cdots\cup\mathcal{P}_m$.  Since $\mathcal{S}(\mathcal{P}_1\cup\mathcal{P}_2)\subset\bigcup_{i=1}^k\mathcal{S}(\mathcal{Q}_i)$, every element in $\mathcal{S}(\mathcal{T})$ (recall that $\mathcal{T}=\bigcup_{i=1}^n\mathcal{P}_i$) is an element in $\mathcal{S}(\mathcal{Q}_i\cup\mathcal{P}_3\cup\cdots\cup\mathcal{P}_m)$ for some $i$.  As the lemma holds for each $\mathcal{Q}_i\cup\mathcal{P}_3\cup\cdots\cup\mathcal{P}_m$, the lemma clearly holds for $\mathcal{T}=\bigcup_{i=1}^n\mathcal{P}_i$.
\end{proof}

\section{Engulfing normal tori by a solid torus}\label{Sengulf}

\begin{lemma}\label{Lsolid}
Let $\mathcal{T}$ be a regular set of normal tori carried by $B$.  Suppose the union of all the tori in $\mathcal{T}$ is a connected 2-complex $\Sigma$.  As $\mathcal{T}$ is a regular set, we may view a neighborhood $N(\Sigma)$ of $\Sigma$ in $M$ as a Seifert fiber space.  Let $\Gamma_1,\dots,\Gamma_k$ be the boundary tori of $N(\Sigma)$.  Then some $\Gamma_i$ must be isotopic (in $M-\Sigma$) to a normal torus and $\Gamma_i$ bounds a solid torus in $M$ that contains $\Sigma$.
\end{lemma}
\begin{proof}
First note that the tori $\Gamma_i$'s may not be normal surfaces.  If one try to normalize $\Gamma_i$ in $M-\Sigma$, by \cite[Theorem 3.2]{JR}, the 2-complex $\Sigma$ is a barrier for the normalization process.  So either (1) $\Gamma_i$ is isotopic (in $M-\Sigma$) to a normal torus, or (2) a compression occurs when normalizing $\Gamma_i$.  If a compression occurs, then after the compression, $\Gamma_i$ becomes a 2-sphere which either can be normalized (in $M-\Sigma$) to a normal 2-sphere or becomes a 2-sphere inside a tetrahedron during the normalization.  Since the only normal 2-sphere is vertex-linking, if a compression occurs during the normalization, then $\Gamma_i$ must bound a solid torus in $M-\Sigma$.  Since each normal torus bounds a solid torus in $M$, in any case, $\Gamma_i$ bounds a solid torus in $M$ for each $i$.

The argument above implies that either Lemma~\ref{Lsolid} holds or each $\Gamma_i$ bounds a solid torus in $M-\Sigma$, i.e., $M-\Int(N(\Sigma))$ is a collection of solid tori.  Suppose the lemma is false and each $\Gamma_i$ bounds a solid torus in $M-\Sigma$.  Since $\mathcal{T}$ is a regular set of tori, the intersection of the tori contains no triple point and the double curves of the intersection are not meridians of the tori in $\mathcal{T}$.  Since $\mathcal{T}$ is a regular set of normal tori, as in the proof of Corollary~\ref{Cmeridion}, a double curve in the intersection does not bound an embedded disk in $M$.  As $M-\Int(N(\Sigma))$ is a collection of solid tori, this means that $M$ is a Seifert fiber space with each double curve of the intersection a regular fiber.  This contradicts our hypothesis at the beginning that $M$ is not a Seifert fiber space.  Therefore some $\Gamma_i$ must bound a solid torus in $M$ that contains $\Sigma$ and the lemma holds.
\end{proof}

\begin{definition}\label{Dbalanced2}
Let $T$ and $F$ be two closed orientable surfaces carried by $N(B)$.  Let $\alpha$ be a simple closed curve in $T$ and we suppose $\alpha$ is transverse to $F$.  Let $A\subset N(B)$ be a thin vertical annulus that contains $\alpha$ and let $\Lambda=A\cap F$. 
We say $\alpha$ is \textbf{balanced} with respect to $F$ if $\Lambda$ is balanced in $A$ as in Definition~\ref{Dbalanced}. 
\end{definition}

\begin{lemma}\label{LAsign}
Let $T$ and $F$ be two closed surfaces carried by $N(B)$.  Let $\alpha_1$ and $\alpha_2$ be disjoint simple closed curves in $T$ transverse to $F$.  Suppose $\alpha_1\cup\alpha_2$ bounds an annulus in $T$.  If $\alpha_1$ is balanced with respect to $F$, then so is $\alpha_2$.
\end{lemma}
\begin{proof}
Let $A$ be the annulus bounded by $\alpha_1\cup\alpha_2$ in $T$.  For any essential arc $\gamma$ properly embedded in $A$ and vertical in $A$, a direction along $\alpha_i$ induces a normal direction for $\gamma$ in $A$.  We say two orientations along $\alpha_1$ and $\alpha_2$ are compatible along $A$ if they induce the same normal direction for $\gamma$.  Otherwise we say the two orientations for $\alpha_1$ and $\alpha_2$ are opposite along $A$.  We choose compatible orientations for $\alpha_1$ and $\alpha_2$ along $A$.  Let $\beta$ be any arc of $F\cap A$.  Since both $F$ and $T$ are transverse to the $I$-fibers of $N(B)$, it is easy to see that if $\beta$ is a trivial arc in $A$, i.e., $\partial\beta$ lies in the same curve $\alpha_i$, then the two signs at $\partial\beta$ are opposite in $\alpha_i$.  Similarly, if the two endpoints of $\beta$ lie in different components of $\partial A$, then the two signs at $\partial\beta$ are the same.  Thus the sum of the signs for $\alpha_1$ is equal to the sum for $\alpha_2$ and the lemma holds.
\end{proof}

\begin{lemma}\label{Lmosign}
Let $T$ be a normal torus carried by $N(B)$ and let $\hat{T}$ be the solid torus bounded by $T$.  Suppose there is a vertical annulus $A$ of $N(B)$ properly embedded in $\hat{T}$ with $\partial A$ essential in $T$.  Let $\alpha$ be a component of $\partial A$.  Then $\alpha$ is  balanced with respect to any closed surface carried by $B$. 
\end{lemma}
\begin{proof} The proof of this lemma is similar in spirit to that of Lemma~\ref{Lslope}. 
Since $A$ is properly embedded in $\hat{T}$ and with $\partial A$ essential in $T$, $A$ is $\partial$-parallel in $\hat{T}$.  Let $A_T$ be the annulus in $T$ parallel (in $\hat{T}$) to $A$ and with $\partial A_T=\partial A$.  So $A_T\cup A$ bounds a $mongogon\times S^1$ region.    Let $\alpha$ and $\alpha'$ be the two components of $\partial A$ and we choose orientations for $\alpha$ and $\alpha'$ so that they are compatible along the annuli $A_T$ and $A$ (see the proof of Lemma~\ref{LAsign}). 
  
Let $F$ be any closed surface carried by $N(B)$.  
For any arc $\gamma$ of $F\cap A$, if $\gamma$ is a $\partial$-parallel arc in $A$, then since $\gamma$ is transverse to the $I$-fibers and $A$ is vertical in $N(B)$, the signs at the two points $\partial\gamma$ are opposite in the corresponding component of $\partial A$.  If $\gamma$ is an essential arc in $A$, since the orientations of $\alpha$ and $\alpha'$ are compatible along $A$ and since $\gamma$ is transverse to the $I$-fibers (see Figure~\ref{cusp}(a)), the signs at the two endpoints $\partial\gamma$ in $\alpha$ and $\alpha'$ are also opposite.  Let $m$ be the sum of the signs of the intersection points in $\alpha\cap F$.  The argument above implies that the sum of the signs of the intersection points in $\alpha'\cap F$ must be $-m$.  However, by considering $\alpha$ and $\alpha'$ as boundary curves of $A_T$, the proof of Lemma~\ref{LAsign} says that the sum of the signs of the intersection points in $\alpha'\cap F$ is the same as the sum for $\alpha\cap F$ which is $m$.  So $m=-m$ and $m=0$.  Thus $\alpha$ is balanced with respective to $F$.
\end{proof}

\begin{lemma}\label{Lsheets}
Let $T$ and $F$ be closed surfaces carried by $N(B)$.  Let $\alpha\subset T$ be a balanced simple closed curve with respect to $F$.  Then there is a number $k$ depending on the intersection of $\alpha\cap F$ such that, after some $B$-isotopy, $F'=F+mT$ is disjoint from $\alpha$ if $m>k$.
\end{lemma}
\begin{proof}
Let $A$ be a vertical annulus containing $\alpha$.  Then the lemma is an immediate corollary of Lemma~\ref{Lbalanced} (by considering the arcs $\Lambda=A\cap F$).
\end{proof}

Let $\{T_1,\dots, T_n\}$ ($n\ge 2$) be a regular set of tori carried by $B$ and let $\Gamma=\bigcup_{i=1}^nT_i$.  Suppose $\Gamma$ is a connected 2-complex.  The intersection curves of the $T_i$'s cut the $T_i$'s and $\Gamma$ into a collection of annuli. By Lemma~\ref{Lsolid}, the union of a subset of these annuli form a torus $T_\Gamma$ that bounds a solid torus containing $\Gamma$.  The torus $T_\Gamma$ is a subcomplex of $\Gamma$ and can be viewed as the frontier of the 2-complex $\Gamma$.  $T_\Gamma$ contains at least one double curve (of the intersection of the tori) in $\Gamma$.

\begin{lemma}\label{Lengulf}
Let $\mathcal{T}=\{T_1,\dots, T_n\}$ ($n\ge 2$), $\Gamma=\bigcup_{i=1}^nT_i$ and $T_\Gamma$ be as above.  In particular, $\{T_1,\dots, T_n\}$ is a regular set of tori, $\Gamma$ is connected, and $T_\Gamma$ is a torus bounding a solid torus containing $\Gamma$ as in Lemma~\ref{Lsolid}. 
 Suppose there is a good torus $T$ that non-trivially intersects at least one torus in $\mathcal{T}$ and $T\cup\mathcal{T}$ is a regular set.   Let $F$ be a surface carried by $B$. Then there is an number $K$ which can be algorithmically determined, such that for any $F'=F+\sum_{i=1}^nc_iT_i$ with each $c_i\ge K$, after $B$-isotopy,  $T_\Gamma\cap F'$ (if not empty) has the same slope in $T_\Gamma$ as the double curves in $\Gamma$.  In particular, $T_\Gamma\cap F'$ is either empty or a collection of essential and non-meridional curves in $T_\Gamma$.  
\end{lemma}
\begin{proof}
First note that since the $T_i$'s are all carried by $N(B)$, for any double curve $\gamma\subset T_p\cap T_q$, an annular neighborhood of $\gamma$ in $T_p$ is $B$-isotopic to an annular neighborhood of $\gamma$ in $T_q$, and this implies that if $\gamma$ is balanced in $T_p$ with respect to $F$, then $\gamma$ is also balanced in $T_q$ with respect to $F$.  

Now we consider the good torus $T$ in the hypotheses.  By Lemma~\ref{Lmosign}, there is a curve $\alpha$ in $T$ that is balanced with respect to $F$ and $\alpha$ corresponds to the cusp of a $monogon\times S^1$ region formed by $T$.  Moreover, by Lemma~\ref{Lslope}, the slope of $\alpha$ in $T$ is the same as the slope of the double curves in $T\cap T_i$. As $T\cup\mathcal{T}$ is a regular set, by Lemma~\ref{LAsign}, this means that a double curve $\alpha'$ in $\Gamma$ is balanced with respect to $F$.   Since the 2-complex $\Gamma$ is connected, the double curves in $\Gamma$ are connected by annuli in the $T_i$'s.  By Lemma~\ref{LAsign}, we can use the curve $\alpha'$ above to successively show that every double curve in $\Gamma$ is balanced with respect to $F$.

Note that $T_\Gamma$ is a subcomplex of $\Gamma$ and it is the union of some annuli  along the double curves of the $T_i$'s.  Since all the double curves are balanced with respect to $F$, by Lemma~\ref{Lsheets}, after some $B$-isotopy, we may assume $F'=F+\sum_{i=1}^nc_iT_i$ is disjoint from those double curves that lie in $T_\Gamma$ if the $c_i$'s are all large.  Moreover, by Lemma~\ref{LbalancedD}, if the $c_i$'s are large, we may assume the intersection of $F'$ with those annuli in $T_\Gamma$ (bounded by the double curves) contains no curve trivial in $T_\Gamma$.  This means that after $B$-isotopies, either $F'\cap T_\Gamma=\emptyset$ or curves in $F'\cap T_\Gamma$ are parallel in $T_\Gamma$ to the double curves.  The bound $K$ in the lemma depends on the intersection patterns of $F$ with the $T_i$'s and $K$ can be easily calculated using the proofs of Lemma~\ref{Lbalanced} and Lemma~\ref{LbalancedD}.

Since $\mathcal{T}$ is a regular set, a double curve in $\Gamma$ does not bound an embedded disk in $M$ (see the proof of Corollary~\ref{Cmeridion}).  By Lemma~\ref{Lsolid}, $T_\Gamma$ is isotopic to a normal torus.  Hence $F'\cap T_\Gamma$ and the double curves of $\Gamma$ that lie in $T_\Gamma$ are essential and non-meridional in $T_\Gamma$.
\end{proof}

We finish this section by quoting a theorem of Scharlemann \cite[Theorem 3.3]{S}.  The theorem is  the main motivation for Lemma~\ref{Lengulf}.

\begin{theorem}[Theorem 3.3 of \cite{S}]\label{Tsch}
Suppose $H_1\cup_S H_2$ is a strongly irreducible Heegaard splitting of a 3-manifold and $V\subset M$ is a solid torus such that $\partial V$ intersects $S$ in parallel essential non-meridional curves.  Then $S$ intersects $V$ in a collection of $\partial$-parallel annuli and possibly one other component, obtained from one or two annuli by attaching a tube along an arc parallel to a subarc of $\partial V$.
\end{theorem}

\section{Proof of the main theorem}\label{Sproof}

In this section, we suppose our manifold $M$ is non-Haken and we discuss the Haken case in the next section.  Note that by \cite{JO}, there is an algorithm to determine whether or not $M$ is Haken.  The goal of this section is to algorithmically list all the Heegaard splittings of any fixed genus $g$ (with possible repetition) in $M$.

Since Heegaard splittings of lens spaces and small Seifert fiber spaces are classified \cite{BO, BCZ, M}, we may assume $M$ is not a Seifert fiber space.  By \cite{JR}, there is an algorithm to find a 0-efficient triangulation for $M$.

As in section~\ref{Spre}, we can algorithmically find a finite collection of branched surfaces such that:
\begin{enumerate}
\item every strongly irreducible Heegaard surface is fully carried by a branched surface in this collection,
\item no branched surface in this collection carries any normal or almost normal 2-sphere.
\end{enumerate}
For each surface, we are interested in the simplest branched surface carrying it, so we assume this collection contains all the sub-branched surfaces of every branched surface in it. 
Let $B$ be a branched surface in this collection.  To simplify notation, we view a Heegaard surface as an almost Heegaard surface with associated (almost vertical) arcs being empty, see Definition~\ref{Dai}. 

By  Lemma~\ref{Lal}, for any strongly irreducible Heegaard surface $S'$ of genus $g$ fully carried by $B$, there is a normal or an almost normal surface $S$ carried by $B$ such that 
\begin{enumerate}
\item $S$ is an almost strongly irreducible Heegaard surface and $S'$ can be derived from $S$,  
\item  the total length of the almost vertical arcs associated to $S'$ is bounded from above by a number $K'$ which depends only on $M$, $B$ and the genus $g$.  Moreover, we can compute an upper bound for $K'$.  
\item there is no non-trivial $D^2\times I$ region for $S$ and $B$.  
\end{enumerate}

Note that we may assume the surface $S'$ above is $K'$-minimal, see Definition~\ref{Dkmin}.  Next we try to find all possible such almost strongly irreducible Heegaard surfaces $S$.

We consider all the normal and almost normal surfaces of genus at most $g$ carried by a branched surface $B$ in this collection.  As explained in section~\ref{Sintro} and section~\ref{Spre}, after solving for the fundamental set of solutions to the branch equations as in the normal surface theory, we can express any normal or almost normal surface carried by $B$ as a linear combination of these fundamental solutions.  Since $B$ does not carry any normal or almost normal 2-sphere or projective plane (see Lemma~\ref{LnoS2}), we may assume the surface corresponding to each fundamental solution has non-positive Euler characteristic.  By Corollary~\ref{Cklein}, there is no Klein bottle in the fundamental solutions.  Since the canonical cutting and pasting preserves Euler characteristic and since the genus is bounded, the coefficient of each non-torus fundamental solution is bounded.  Thus we can find a finite collection  of surfaces $\mathcal{F}_1$ of genus at most $g$ and a finite collection of normal tori $T_1\dots, T_n$ carried by $B$ such that, for every normal or almost normal surface $S$ of genus at most $g$ and carried by $B$, we can express $S$ as $S=H+\sum_{i=1}^n c_i T_i$, where $H\in\mathcal{F}_1$.  Note that if the fundamental solutions contain an almost normal torus, since $S$ contains at most one almost normal piece, the coefficient at the almost normal torus is at most one.  Hence we can add the possible almost normal torus to $H$ and this is the reason that we can assume each torus $T_i$ is normal.

Our goal is to find all possible such surfaces $S$.  Next we show that except for finitely many possibilities, $S=H+\sum_{i=1}^n c_i T_i$ can be expressed into some nicer form.  

First note that we can compute an upper bound for the constant $k$ in Lemma~\ref{Lflare} for any $H\in\mathcal{F}_1$ and any subset of  $\{T_1,\dots, T_n\}$.  Let $\mathcal{F}_2$ be all the surfaces of the form $H+\sum_{i=1}^n c_iT_i$ with $H\in \mathcal{F}_1$ and each $c_i\le k$, where $k$ is the constant in Lemma~\ref{Lflare}.  Clearly $\mathcal{F}_2$ is a finite set.  So the sum of $H$ and all the summands $c_iT_i$ with coefficient $c_i\le k$ is a surface in $\mathcal{F}_2$, and the remaining summands have coefficients larger than $k$.  Let $\mathcal{T}_S$ be the collection of tori $T_i$ with coefficient $c_i>k$, so $\mathcal{T}_S$ is a subset of $\mathcal{T}$. 

For any almost strongly irreducible Heegaard surface $S=H+\sum_{i=1}^n c_i T_i$ as above, since there is no non-trivial $D^2\times I$ region for $S$ and $B$, it follows from Lemma~\ref{Lflare} that, we can express $S$ as $S=F'+\sum_{T_i\in\mathcal{T}_S}c_iT_i$ such that 
\begin{enumerate}
\item $F'\in\mathcal{F}_2$, 
\item $\mathcal{T}_S$ is a subset of $\{T_1\dots, T_n\}$, 
\item $c_i>k$ for each $i$, where $k$ is as in Lemma~\ref{Lflare}, and 
\item there is no flare in the sub-branched surface that fully carries the union of the tori in $\mathcal{T}_S$.
\end{enumerate}

Let $\mathcal{T}$ be any subset of $\{T_1\dots, T_n\}$ and let $B_T$ be the sub-branched surface of $B$ that fully carries the union of the tori in $\mathcal{T}$.  By Lemma~\ref{Linnermost} and its proof, $B_T$ contains a flare if and only if there are a component $A$ of $\partial_vN(B_T)$ and a component $\alpha$ of $\partial A$ such that (1) $\alpha$ bounds an embedded disk $D$ carried by $N(B_T)$ and (2) the normal direction of $\alpha$ induced from the branch direction points out of $D$ (i.e. $D$ is not a disk of contact).  By solving branch equations, similar to \cite{AL}, we can check each component of $\partial_vN(B_T)$ to see whether or not such a disk $D$ exists.  So we can algorithmically determine whether or not $B_{T}$ has a flare.  Next, we list and only consider all the subsets of $\{T_1\dots, T_n\}$ whose corresponding branched surface $B_T$ has no flare.

Let $\mathcal{T}$ and $B_T$ be as above and suppose $B_T$ has no flare.  By Lemma~\ref{Lset}, there is a finite set of tori $\hat{\mathcal{T}}$ such that any surface $S=F'+\sum_{T_i\in \mathcal{T}} n_iT_i$ can be expressed as $S=F'+\sum_{T_i\in \mathcal{T}'} c_iT_i$ for some regular subset of tori $\mathcal{T}'\subset \hat{\mathcal{T}}$.  By Lemma~\ref{Lset}, we can algorithmically find the set of tori $\hat{\mathcal{T}}$ and all possible regular subsets of tori $\mathcal{T}'$ in $\hat{\mathcal{T}}$.

Let $\mathcal{T}'\subset \hat{\mathcal{T}}$ be a regular set of tori.  We consider all the surfaces that can be expressed as $S=F+\sum_{T_i\in\mathcal{T}'}c_i T_i$, where $F\in\mathcal{F}_2$ is as above.  Let $\Gamma_1,\dots\Gamma_k$ be the connected components of $\bigcup_{T_i\in\mathcal{T}'}T_i$ and we  use $\mathcal{T}_i$ to denote the subset of tori whose union is $\Gamma_i$.  So $\mathcal{T}'=\bigcup_{i=1}^k\mathcal{T}_i$.

For each $\mathcal{T}_i$, by Lemma~\ref{Lgoodone}, we can either (1) find a good torus $T$ that intersects at least one torus in $\mathcal{T}_i$ and such that $T\cup\mathcal{T}_i$ is a regular set, or (2) conclude that $\mathcal{G}(\mathcal{T}_i)$ is finite and list all possible tori in $\mathcal{G}(\mathcal{T}_i)$. 
 If $\mathcal{G}(\mathcal{T}_i)$ is finite, then each surface in $\mathcal{S}(\mathcal{T}_i)$ is the union of some parallel copies of disjoint tori in $\mathcal{G}(\mathcal{T}_i)$.  Thus, in possibility (2) above, by adding the finite list of all possible tori in $\mathcal{G}(\mathcal{T}_i)$ to $\hat{\mathcal{T}}$, we may assume that for any connected component $\Gamma_i$ as above, either we have a good torus $T$ as in (1) above, or $\mathcal{T}_i$ is a single torus disjoint from all other tori in $\mathcal{T}'$.

If a torus $T\in\mathcal{T}'$ is disjoint from all other tori in $\mathcal{T}'$.  We consider the surface $F+cT$ where $c$ is a positive integer. 
By Lemma~\ref{LbalancedD}, there is a number $K_{T,F}$, which depends on $T$ and $F\cap T$, such that, if $c>K_{T,F}$, $(F+cT)\cap T$ (if not empty) consists of curves essential in the torus $T$.  Moreover, similar to the proof of Lemma~\ref{LbalancedD}, there is a number $K_{T,F}$ such that if $c>K_{T,F}$, then $(F+cT)+mT$ is the surface obtained by an $m$-fold Dehn twist on $F+cT$ along the torus $T$.  Note that if $(F+cT)\cap T=\emptyset$, then $F+(c+m)T$ has $m$ components that are parallel copies of $T$.  Since $T$ is disjoint from all other tori in $\mathcal{T}'$, this means that $S$ has $m$ components that are parallel copies of $T$.  However, $S$ is an almost strongly irreducible Heegaard surface with (total) genus at most $g$, so $m<g$ in this case.  Suppose $(F+cT)\cap T\ne\emptyset$.  Since the normal torus $T$ bounds a solid torus, a Dehn twist along $T$ is just an isotopy.  Hence $F+(c+m)T$ is isotopic to $F+cT$.   As $T$ is disjoint from all other tori in $\mathcal{T}'$, the argument above means that there is a number $K_{T,F}'$ which can be algorithmically determined, such that $F+\sum_{T_i\in\mathcal{T}'} c_iT_i$ is isotopic to a surface of the same form but with the coefficient of $T$ at most $K_{T,F}'$.

Let $\mathcal{F}_3$ be the set of surfaces of the form $F+\sum cT$, where $F\in\mathcal{F}_2$, $c\le K_{T,F}'$ and each $T$ in the sum is a torus disjoint from all other tori in $\mathcal{T}'$ as above.  Note that, similar to the discussion at the end of the proof of Lemma~\ref{Lfinite}, after isotopy and setting $K_{T,F}'$ to be sufficiently large, we may assume the almost vertical arcs associated to the almost Heegaard surface $S$ are disjoint from the Dehn twist along $T$.  Hence we may assume the Dehn twist along $T$ above does not affect the arcs associated to $S$.  In particular, the total length of the almost vertical arcs associated to $S$ is not changed by the Dehn twist.  Thus after the Dehn twists above, our almost strongly irreducible Heegaard surface $S$ above can be expressed as $S=F+\sum_{T_i\in\mathcal{T}'}c_i T_i$, where 
\begin{enumerate}
\item $F\in\mathcal{F}_3$, 
\item $\mathcal{T}'$ is a regular set of tori from $\hat{\mathcal{T}}$, 
\item for any connected component $\Gamma$ of $\bigcup_{T_i\in\mathcal{T}'}T_i$, we can algorithmically find a good torus $T$ (as in Lemma~\ref{Lgoodone}) such that $T$ non-trivially intersects $\Gamma$ and $T\cup\Gamma$ is a regular set.
\end{enumerate}

Let $S=F+\sum_{T_i\in\mathcal{T}'}c_i T_i$ be as above.  Let $\Gamma$ be a connected component of $\bigcup_{T_i\in\mathcal{T}'}T_i$ and let $\mathcal{T}_\Gamma$ be the set of tori in $\Gamma$.  By Lemma~\ref{Lsolid}, there is a torus $T_\Gamma$ bounding a solid torus that contains $\Gamma$.  By Lemma~\ref{Lengulf}, if each $c_i\ge K$ for some constant $K$ which depends on $F\cap T_\Gamma$, then $S\cap T_\Gamma$ is a union of essential non-meridional curves parallel to the double curves of $\Gamma$.  Now we algorithmically find the number $K$ in Lemma~\ref{Lengulf} and let $\mathcal{F}_4$ be the set of surfaces $F+\sum_{T_i\in\mathcal{T}'}c_i T_i$ where $F\in\mathcal{F}_3$ and each $c\le K$.  
This means that the surface $S$ above can be expressed as $S=F'+\sum_{T_i\in\mathcal{T}''}c_i T_i$, where $\mathcal{T}''$ is a subset of $\mathcal{T}'$ above and $c_i>K$ for each $c_i$.

Note that the subset $\mathcal{T}''$ above is also a regular set of tori.  So we can consider the connected components of $\bigcup_{T_i\in \mathcal{T}''}T_i$ and repeat the arguments above.  Eventually, we can find a finite set of surfaces $\mathcal{F}_5$ such that the almost strongly irreducible Heegaard surface $S$ above has the form $S=F+\sum_{T_i\in\mathcal{T}''}c_i T_i$, where $F$ and $\mathcal{T}''$ have the following properties:
\begin{enumerate}
\item $F\in\mathcal{F}_5$ and $\mathcal{T}''$ is a regular set of tori in $\hat{\mathcal{T}}$
\item for each connected component $\Gamma$ of $\bigcup_{T_i\in \mathcal{T}''}T_i$, let $T_\Gamma$ denote the torus as in Lemma~\ref{Lsolid} that bounds a solid torus containing $\Gamma$, then $F\cap T_\Gamma$ (if not empty) consists of curves parallel to the double curves of $\Gamma$ as in Lemma~\ref{Lengulf}, in particular $F\cap T_\Gamma$ are essential and non-meridional in $T_\Gamma$.
\end{enumerate}

Each connected component of $\bigcup_{T_i\in \mathcal{T}''}T_i$ lies in a solid torus as above.  We can find a collection of disjoint solid tori that contain $\bigcup_{T_i\in \mathcal{T}''}T_i$.  Let $W$ be the union of these solid tori. 
Our surface $S=F+\sum_{T_i\in\mathcal{T}''}c_i T_i$ is an almost strongly irreducible Heegaard surface.  Let $\Sigma_s$ be the almost vertical arcs associated to $S$ and by the argument above, we already know an upper bound $K'$ of the total length of the arcs in $\Sigma_s$.  Each arc in $\Sigma_s$ is an arc properly embedded in $\overline{M-S}$.  As $S=F+\sum_{T_i\in\mathcal{T}''}c_i T_i$, we may view $\Sigma_s$ as a set of arcs between $F$ and these $c_i$ copies of $T_i$'s ($T_i\in\mathcal{T}''$).  As the total number of components of $\Sigma_s$ is at most $g$ (where $g$ is the genus of the Heegaard surface), if $c_i\ge 2g$, at least one copy of $T_i$ in the sum above is not incident to $\Sigma_s$.   Thus after enlarging $\mathcal{F}_5$ to include surfaces of the form $F+\sum_{T_i\in\mathcal{T}''}n_i T_i$ with each $n_i<2g$, we may assume, for each summand $c_iT_i$ in our expression of $S$ above, we have $c_i\ge 2g$.  So we may assume that, after isotopy, each arc in $\Sigma_s$ lies either totally inside the solid tori $W$ or totally outside $W$.

Since $\mathcal{F}_5$ is a finite set and $\bigcup_{T_i\in \mathcal{T}''}T_i\subset W$, there are only finitely many possible configurations for $S-W$.  By listing all the surfaces in $\mathcal{F}_5$, we can algorithmically list all possible configurations for $S$ outside the solid tori $W$.  
  As the total length of the arcs in $\Sigma_s$ is bounded, using the finite possible configurations of $S$ outside $W$, we can use Lemma~\ref{Lal} to algorithmically list all possible almost vertical arcs outside $W$.  Thus we can list all possible configurations for the Heegaard surface $S'$ (derived from $S$) outside the solid tori $W$.

For any connected component $\Gamma$ of $\bigcup_{T_i\in \mathcal{T}''}T_i$, let $W_\Gamma$ be the solid torus containing $\Gamma$ and with $T_\Gamma=\partial W_\Gamma$.  We may assume $W_\Gamma$ is a component of $W$.  By our assumption, $F\cap T_\Gamma$ consists of essential and non-meridional curves in $T_\Gamma$.  So by a theorem of Scharlemann (i.e.~Theorem~\ref{Tsch} above), $S'\cap W_\Gamma$ is standard.  Since there are only finitely many possible configuration for a strongly irreducible Heegaard surface $S'$ outside the solid tori $W$,  we can use the curves $F\cap T_\Gamma$ to list all possible surface types in $S'\cap W$.  Thus, up to isotopy, $F+\sum_{T_i\in\mathcal{T}''}c_i T_i$ can produce only finitely many different strongly irreducible Heegaard surfaces and we can algorithmically produce a list of surfaces containing all of them.  As $\mathcal{F}_5$ and the set $\hat{\mathcal{T}}$ are finite, we can use all possible $F\in\mathcal{F}_5$ and all possible subsets $\mathcal{T}''$ as above to produce a final list of surfaces that contain all strongly irreducible Heegaard surfaces of genus at most $g$.

Using Haken's algorithm \cite{Ha} and the algorithm to recognize a 3-ball \cite{Th}, we can determine whether or not each side of a surface is a handlebody.  So we can determine which surfaces in our list are Heegaard surfaces.  Although some surfaces in our list may be isotopic, the list is a complete list of all possible strongly irreducible Heegaard surfaces of genus at most $g$.  In a non-Haken 3-manifold, every Heegaard splitting is either strongly irreducible or a stabilization of a strongly irreducible Heegaard splitting.  So after some stabilizations, we obtain a complete list of Heegaard splittings of genus at most $g$.

\section{The Haken case}\label{SHaken}

In this section, we discuss the case that our manifold $M$ is an atoroidal Haken 3-manifold.  We still assume $M$ is closed, orientable, irreducible and atoroidal.  By \cite{JR}, since $M$ is atoroidal, we can assume $M$ has a 0-efficient triangulation.  By the argument above, we can produce a finite list of Heegaard splittings of genus at most $g$ such that all the strongly irreducible Heegaard splittings of genus at most $g$ are in this list.  To prove Theorem~\ref{Tmain} for Haken atoroidal manifolds, we still need to find all weakly reducible Heegaard splittings of genus at most $g$. 

By \cite{SchT}, any unstabilized weakly reducible Heegaard splitting can be expressed as an amalgamation along a collection of incompressible surfaces in $M$ using strongly irreducible Heegaard splittings of the submanifolds of $M$ bounded by these incompressible surfaces.  In particular, the sum of the genera of the incompressible surfaces is at most the genus of the weakly reducible Heegaard splitting.

It follows from \cite{FO, O1} that, up to isotopy, an atoroidal 3-manifold has only finitely many incompressible surfaces of each genus.  In fact, in the algorithm above for strongly irreducible Heegaard surfaces, one can easily check that all the lemmas remains true if we replace strongly irreducible Heegaard surface by incompressible surface (most cases we considered do not exist for incompressible surfaces).  This means that we can algorithmically find all incompressible surfaces in $M$ of genus at most $g$.  Note that the algorithm for incompressible surfaces is much simpler, since we can assume the branched surface that fully carries an incompressible surface does not have any monogon, see \cite{FO}.

We can list all possible incompressible surfaces in $M$ with genus at most $g$ and consider the strongly irreducible Heegaard splittings of a submanifold $N$ bounded by these incompressible surfaces.  We may assume the incompressible surface $\partial N$ is a normal surface with respect to the 0-efficient triangulation of $M$.  So $N$ has an induced cell decomposition from the 0-efficient triangulation.  By \cite{St} (see Lemmas 4 and 5 of \cite{St}), we may assume a strongly irreducible Heegaard surface of $N$ is normal or almost normal with respect to this cell decomposition.  Since (1) each (induced) 3-cell in $N$ is a block in a tetrahedron bounded by normal disks, (2) the Heegaard surface of $N$ is disjoint from $\partial N$ and (3) the boundary curve of each almost normal piece consists of normal arcs in the 2-skeleton, it is easy to see that any normal disk or almost normal piece (that is disjoint from $\partial N$) of each 3-cell in $N$ remains a normal disk or an almost normal piece of the corresponding tetrahedron of $M$.  Thus any normal or almost normal surface in $\Int(N)$ is also a normal or an almost normal surface in $M$.  In particular, since $\partial N$ is incompressible in $N$, any normal torus in $\Int(N)$ must bound a solid torus in $N$.  Thus we can consider branched surfaces in $\Int(N)$ formed by normal disks and at most one almost normal piece and apply our algorithm in section~\ref{Sproof} to list all strongly irreducible Heegaard surfaces of $N$ with genus at most $g$.  

Therefore, we can list all possible disjoint incompressible surfaces of genus sum at most $g$ and list all possible strongly irreducible Heegaard surfaces (of genus at most $g$) in each submanifold bounded by the incompressible surfaces.  By amalgamating these Heegaard surfaces along incompressible surfaces, we obtain a list of Heegaard surfaces of $M$ which contains all weakly reducible Heegaard surfaces of $M$ with genus at most $g$.

\end{psfrags}

\end{document}